\documentclass[twoside,draft,reqno,12pt]{amsart}
\usepackage[english]{babel}
\usepackage{amsmath}
\usepackage{amssymb}
\usepackage{enumerate}

\usepackage{amsfonts}
\usepackage{graphicx}

\usepackage[ citecolor=black, linkcolor=black,
colorlinks, hypertexnames=false, plainpages=false]{hyperref}

\usepackage{color}
\newtheorem{lemma}     {Lemma}[section]
\newtheorem{thm}        [lemma]{Theorem}
\newtheorem{teorema1}   [lemma]{Theorem}

\newtheorem{coro}       [lemma]{Corollary}
\newtheorem{cong1}      [lemma]{Conjecture}
\newtheorem{remark1}    [lemma]{Remark}
\newtheorem{defin}      [lemma]{Definition}

\numberwithin{equation}{section}

\newcommand{\de}{\delta}
\newcommand{\Om}{\Omega}

\newcommand{\x}{\mathbf x}

\newcommand{\al}{\alpha}
\newcommand{\be}{\beta}
\newcommand{\ep}{\varepsilon}

\newcommand{\ka}{\kappa}
\newcommand{\la}{\lambda}

\newcommand{\ga}{\gamma}
\newcommand{\Ga}{\Gamma}
\newcommand{\si}{\sigma}

\newcommand{\Th}{\Theta}

\newcommand{\om}{{\omega}}
\newcommand{\Up}{\Upsilon}

\newcommand{\bH}{{\text{\bf H}}}
\newcommand{\cA}{\mathcal A}
\newcommand{\cC}{\mathcal C}
\newcommand{\cD}{\mathcal D}
\newcommand{\cE}{\mathcal E}
\newcommand{\cG}{\mathcal G}
\newcommand{\cF}{\mathcal F}
\newcommand{\cH}{\mathcal H}
\newcommand{\cK}{\mathcal K}

\newcommand{\cM}{\mathcal M}
\newcommand{\cR}{\mathcal R}

\newcommand{\cL}{\mathcal L}

\newcommand{\cZ}{\mathcal Z}

\newcommand{\ZZ}{\mathbb{Z}}
\newcommand{\NN}{\mathbb{N}}

\newcommand{\wh}{\widehat}
\newcommand{\wt}{\widetilde}
\renewcommand{\l}{\ell}

\newcommand{\dis}{\displaystyle}

\newcommand{\mmmintone}[1]{{\dis{\int\kern -.38cm
-}}_{\kern-.21cm\substack{#1}}\;\;}
\newcommand{\mmmintwo}[2]{{\dis{\int\kern -.43cm
-}}_{\kern-.21cm\substack{#1}}^{\substack{#2}}\;\;}
\newcommand{\submint}{{\scriptstyle{\int\kern -.66em -}}}
\newcommand{\submintone}[1]{{\scriptstyle{\int\kern -.66em
-}}_{\scriptscriptstyle{\kern-.21em\substack{#1}}}}
\newcommand{\fracmint}{{\textstyle{\int\kern -.88em -}}}
\newcommand{\fracmintone}[1]{{\textstyle{\int\kern -.88em
-}}_{\scriptscriptstyle{\kern-.21em\substack{#1}}}\;}

\newcommand{\nada}[1]{}

\begin{document}

\title[Oriented percolation]
{Oriented percolation in a random environment}

\author{Harry Kesten}
\address{Harry Kesten,
Department of Mathematics, 310 Malott Hall, Cornell University, Ithaca, NY 14853-4201 USA.}
\email{hk21@cornell.edu}

\author{Vladas Sidoravicius}
\address{Vladas Sidoravicius,
IMPA,
Estrada Dona Castorina 110, 22460-320, Rio de Janeiro, Brasil.}%\newline \indent
\email{vladas@impa.br}

\author{Maria Eulalia Vares}
\address{ Maria Eulalia Vares,
IM-Universidade Federal do Rio de Janeiro.\newline Av. Athos da
Silveira Ramos 149, 21941-909, Rio de Janeiro, RJ, Brasil.}
\email{eulalia@im.ufrj.br}
\begin{abstract}
On the lattice $\widetilde{\mathbb Z}^2_+:=\{(x,y)\in \mathbb
Z \times \mathbb Z_+\colon x+y \text{ is even}\}$ we consider the
following oriented (northwest-northeast) site percolation: the
lines $H_i:=\{(x,y)\in \widetilde {\mathbb Z}^2_+ \colon y=i\}$ are
first declared to be {\it bad} or {\it good} with probabilities
$\de$ and $1-\de$ respectively, independently of each other. Given
the configuration of lines, sites on good lines are open with
probability $p_{_G}>p_c$, the critical probability for the
standard oriented site percolation on $\mathbb Z_+ \times \mathbb Z_+$,
and sites on bad lines are open with probability $p_{_B}$, some
small positive number, independently of each other. We show
that given any pair $p_{_G}>p_c$ and $p_{_B}>0$,  there exists a
$\de (p_{_G}, p_{_B})>0$ small enough, so that for $\de \le
\de(p_G,p_B)$ there is a strictly
positive probability of oriented percolation to infinity from the
origin.

\end{abstract}

%\makeindex
\maketitle



\tableofcontents

\section{Introduction}
\label{intro}

On the lattice $\widetilde{\mathbb Z}^2_+:=\{(x,y)\in \mathbb Z \times
\mathbb Z_+\colon x+y \text{ is even}\}$ we consider the following
oriented (northwest-northeast) site percolation: the lines
$H_i:=\{(x,y)\in \widetilde {\mathbb Z}^2_+ \colon y=i\}$ are first
declared to be {\it bad} or {\it good} with probabilities $\de$ and
$1-\de$ respectively, independently of each other. Given the
configuration of lines, sites on good lines are open with
probability $p_G$, and sites on bad lines are open with
probability $p_B$, independently of each other. More formally, on
a suitable probability space $(\Om,\cA,P)$, we consider a
Bernoulli sequence $\xi=(\xi_i\colon i \in \mathbb Z_+)$ with
$P(\xi_i=1)=\delta= 1-P(\xi_i=0)$, which determines $H_i$ to be bad or
good, and a family of occupation variables $(\eta_z\colon z \in
\widetilde{\mathbb Z}^2_+ )$ which are conditionally independent
given $\xi$, with $P(\eta_z=1\mid \xi)=p_B= 1-P(\eta_z=0\mid \xi)$
if $z \in H_i$ with  $\xi_i=1$, and $P(\eta_z=1\mid
\xi)=p_G=1-P(\eta_z=0\mid \xi)$ if $z \in H_i$ with $\xi_i=0$. If
$\eta_z=1$ the site $z$ is {\it open}, and otherwise it is {\it closed}. An
{\it open oriented path} on $\widetilde {\mathbb Z}^2$ is a path along which
the second coordinate is strictly increasing and all of whose
vertices are open. The {\it open cluster} of a vertex $z \in \widetilde
{\mathbb Z}^2$ is the collection of sites which can be reached from
$z$ by an open oriented path. This cluster is denoted by $C_z$ and
the open cluster of the origin is denoted by $C_0$. We always
include $z$ itself in $C_z$, whether $z$ is open or not. We say
that {\it percolation occurs} if $C_0$ is infinite with positive
probability. This description of percolation is of course obtained
by rotating the standard picture on $\mathbb Z_+ \times \mathbb Z_+$ by
$\pi/4$ counterclockwise.

The interesting situation is when $p_G>p_c$, the critical
probability for the standard oriented site percolation on $\mathbb
Z_+ \times \mathbb Z_+$, and $p_B$ is some small positive number.
Given $p_G$ and $p_B$ we ask if $\delta>0$ may be taken small enough
so that there is a positive probability of oriented percolation to
infinity from the origin. We prove the answer to be positive,
provided $p_G>p_c$, as stated in the next theorem, which is the
main result of this article.
\bigskip
\begin{thm} \label{main} In the setup described above, let
\begin{equation*}
\Theta(p_G,p_B,\de)\,=\,P(C_0 \text{ is infinite }).
\end{equation*}
Then, if $p_G>p_c$ and $p_B>0$, we can find $\de_0=\de_0(p_G,p_B)>0$
so that $\Theta(p_G,p_B,\de)>0$ for all $\de \le \de_0$. In fact,
for $\de \le \de_0$,
\begin{equation}
P(C_0 \text{ is infinite}| \xi) > 0 \text{ for almost all }\xi.
\label{1.1}
\end{equation}
\end{thm}
\medskip
Most of our effort in the article will involve proving Theorem
\ref{main}  for $p_G$ close to one. The extension to any
supercritical $p_G$ is discussed at the end.

%The authors there consider the contact process but,
%would we shift their setup to oriented percolation, the main
%difference is that the  ({\it good/bad}) layers in \cite {BDS} are
%parallel to the growth direction. Survival, or percolation,  is then
%achieved by pushing the good lines to be  {\it good enough}, given
%$p_B$ and the frequency $\delta$ of bad lines. It is very simple to
%see that this result cannot hold in the current situation, with the
%layers being transversal to the growth. Other related results will
%be discussed at the end of the paper.

\medskip

This work stems from attempts to understand and answer various questions
which were naturally raised in probability,
theoretical computer science and statistical physics. These questions lie
on crossroads of various fields and have several quite
distinct roots.

\noindent $\bullet$ Spatial growth processes such as percolation or
contact process in random environment is a very well established topic. The situation is reasonably well understood when the environment
has good space-time mixing properties. Much less is known
 for environments with long range dependencies. One source of
inspiration is \cite {BDS}, where the contact process with spatial disorder
persisting in time is considered. Shifting their setup to oriented percolation,
the difference is that the  ({\it good/bad}) layers in \cite {BDS} are
parallel to the growth direction. Our environment is somehow ``orthogonal" and
it generates more global effects.

It is worth comparing the situation treated here with that in
\cite{BDS}, where survival (or percolation)  is achieved by pushing the good
lines to be  {\it good enough}, given $p_B$ and the frequency $\delta$ of bad lines.
It is very simple to see that this result cannot hold in the current situation, with
the layers being transversal to the growth.

% i.e. supercritical and subcritical regions
%do not change during the evolution. Our process could be thought as "orthogonal" to it, since
%we have "supercritical" and "subcritical" times, which in some
%sense in spirit is more close to the dynamic percolation, however with
%much global effects.

\noindent $\bullet$ In late sixties, McCoy and Wu (\cite{MW1,MW2,MW3,MW4}) started the study of
a specific class of disordered ferromagnets  with random couplings that are constant
along each horizontal line, for instance with randomly located layers of strongly and
weakly coupled spin systems.

\noindent $\bullet$ A third set of questions comes from theoretical computer
science. Among them, the clairvoyant scheduling problem or
coordinate percolation, introduced by P. Winkler in early nineties:  is it possible, in a complete graph with $n$ vertices, to schedule two independently sampled random walks (by suitably delaying jumps), so that they never collide? This has a representation in terms of planar oriented percolation (due to Noga Alon). For results in this direction see \cite{Wi, BBS, Gac1}. The answer is negative for $n=2$ or $n=3$. Numerical simulations suggest a positive answer for $n\ge 4$. Recent progress in \cite{BSS} gives a positive answer for $n$ large enough.

In parallel, several questions of similar nature, such as compatibility of binary sequences, Lipshitz
embedding and rough isometries
of random one dimensional objects have been considered and recently answered in \cite{BS}. (See also \cite{H, P, Gac2, compat}.)

The approach undertaken in \cite{BS} is, as ours, based on multi-scale analysis. While the
general concept is similar, both methods are quite different in its technical execution. The scheme developed in
\cite{BS} relies more on the fine probabilistic block estimates. The approach taken in our work gives a precise geometric 
description of the random environment, describing the global picture in terms of increasing hierarchies and inter-relation between them.

\section{Construction of renormalized lattices. Step 1: clusters}
\label{sec2constr}

Define $\Ga \equiv \Ga (\om) = \{x \in \Bbb Z_+ \colon \xi_x = 1
\}$ and label the elements of $\Ga$ in increasing order $\Ga = \{
x_j\}_{j\ge 1}$. The sequence $\Ga$ is called the {\it
environment}.

We will build an infinite sequence $\{{\mathbf{C}}_k\}_{k\ge 0}$ of
partitions of $\Ga$. Each partition ${\mathbf{C}}_k$ is a collection
${\mathbf{C}}_k = \{ {\cC}_{k,j} \}_{j\ge 1}$ of subsets of $\Ga$.
We call the  elements of ${\mathbf{C}}_k$ {\it clusters}. The
construction depends on a parameter $L$, a positive integer which
will be fixed later so that the property in Lemma 2.1 below is
satisfied. The clusters will be constructed so as to have the
properties
\begin{equation}
\text{each ${\cC}_{k,j}$ is of the form $I\cap \Ga$ for an
interval $I$}, \label{2.z}
\end{equation}
and
\begin{equation}
\text{span}({\cC}_{k,j}) \cap \text{span}({\cC}_{k,j'}) =
\emptyset \quad \text {if }\; j \neq j', \label{2.one}
\end{equation}
where $span(C)$ is simply the smallest interval (in $\Bbb Z_+$) that
contains $C$. To each cluster ${\cC}_{k,j}$  we will attribute a
{\it mass}, $m({\cC}_{k,j})$, in such a way that
\begin{equation}
d({\cC}_{k,j}, {\cC}_{k,j'}) \ge L^r, \label{2.two}
\end{equation}
if $\min
\{m({\cC}_{k,j}), m ({\cC}_{k,j'})  \} \ge r, \;  \text{for}
\;  r = 1, \dots, k \text{ and } j \ne j'. $
Here the distance $d(D_1,D_2)$ is the usual Euclidean distance
between two sets $D_1$ and $D_2$. To each cluster ${\cC}_{k,j}$
we shall further associate a number $\ell({\cC}_{k,j}) \in
\{0,1,\dots,k\}$ which will be called the {\it level of the
cluster}. This level will satisfy
\begin{equation}
0 \le \ell ({\cC}_{k,j}) < m({\cC}_{k,j}).
\label{lev}
\end{equation}
Finally we construct the (limiting) partition ${\mathbf{C}}_{\infty}= \{ {\cC}_{\infty ,j} \}_{j\ge 1}$ of $\Ga$ with the
property that
\begin{equation}
\text{span}({\cC}_{\infty,j}) \cap \text{span}({\cC}_{\infty,j'}) = \emptyset \quad \text {if} \; j \neq j'.
\label{2.onea}
\end{equation}
To each cluster ${\cC}_{\infty,j}$ of ${\mathbf{C}}_{\infty}$
we will attribute a mass $m ({\cC}_{\infty,j})$, and a level
$\ell({\cC}_{\infty,j})$, in such a way that
\begin{equation}
0 \le \ell({\cC}_{\infty,j}) < m ({\cC}_{\infty,j}), \label{2.a}
\end{equation}
and the following property holds:
\begin{equation}
d({\cC}_{\infty,j}, {\cC}_{\infty,j'}) \ge L^r, \; \; \text{if}
\; \min \{m({\cC}_{\infty,j}), m({\cC}_{\infty,j'})  \} \ge r,
\; \; \text{for} \; \; r \ge 1 \text{ and } j \ne j',\label{2.four}
\end{equation}
or equivalently,
\begin{equation}
d({\cC}_{\infty,j}, {\cC}_{\infty,j'}) \ge L^{\min \{m({\cC}_{\infty,j}), m({\cC}_{\infty,j'})  \}}, \text {for} \; j \neq
j'. \label{2.five}
\end{equation}

\noindent {\bf The construction.}

Recall that $\Ga \equiv \Ga (\om) = \{ x \in \Bbb Z_+ \colon \xi_x
= 1 \}$, and that the elements of $\Ga$ are labeled in increasing
order. $\Ga = \{ x_j\}_{j\ge 1} $ with $x_1 < x_2 < \dots$.

\medskip
\noindent {\bf Level 0.} The clusters of level 0 are just the
subsets of $\Ga$ of cardinality one. We take ${\cC}_{0,j} =
\{x_j\}$ and attribute a unit mass to each such cluster. That is,
$m ({\cC}_{0,j})=1$ and $\ell({\cC}_{0,j})=0$. Set  ${\mathbf{C}}_{0,0}=
{\mathbf{C}}_0 = \{ {\cC}_{0,j} \}_{j\ge 1}$. Further
define $\al(\cC) = \om(\cC) = x$  when ${\cC} = \{x\}$ is a
cluster of level 0.

\medskip

\noindent {\bf  Level 1.} We say that $x_i, x_{ i+1}, \dots
x_{i+n-1}$ form  a {\it maximal $1$-run of length $n\ge 2$} if
\begin{equation*}
x_{j+1}-x_j\, < \, L,  \; j=i, \dots ,i+n-2,
\end{equation*}
and
\begin{equation*}
x_{j+1}-x_j\, \ge \, L  \; \;
\begin{cases} \text{for} \; \; j=i-1, j= i+n-1, \; &\text{if } \; i>1 \\
\text{for} \; \; j= i+n-1, \; &\text{if } \; i=1. \end{cases}
\end{equation*}
The level 0 clusters $\{x_i\}, \{x_{i+1}\}, \dots \{x_{i+n-1}\}$
will be called {\it constituents} of the run. Note that there are
no points in $\Ga$ between two consecutive points of a maximal
1-run. Also note that if $x_{j+1}-x_j \ge L$ and $x_j-x_{j-1} \ge
L$, then $x_j$ does not appear in any maximal 1-run of length at
least 2.

For any pair of distinct maximal runs, $r'$ and $r''$ say, all
clusters in $r'$ lie to the left of all clusters in $r''$ or vice
versa. It therefore makes sense to
label the  consecutive maximal $1$-runs of length at least $2$ in
increasing order of appearance: $r^1_1, r^1_2, \dots $. It is
immediate that $P$-a.s. all runs are finite, and that infinitely
many such runs exist. We write $r^1_i = r^1_i(x_{s_i}, x_{s_i +1},
\dots x_{s_i+n_i-1}), \; i=1, \dots$, if the $i$-th run consists
of $x_{s_i}, x_{s_i +1}, \dots x_{s_i+n_i-1}$. Note that $n_i\ge
2$ and $s_i+n_i \le  s_{i+1}$ for each $i$. The set
\begin{equation*}
  \cC^1_{i} =\{x_{s_i},x_{s_i+1},\dots,x_{s_i+n_i-1}\}
\end{equation*}
is called a {\it level 1-cluster}, i.e., $\ell( \cC^1_i)=1$. We attribute to $\cC^1_i$ the mass given by
its cardinality:
\begin{equation*}
m (\cC^1_i) = n_i.
\end{equation*}
The points
\begin{equation*}
\alpha^1_i = x_{s_i}\quad \text{and}\quad \omega^1_i =
x_{s_i+n_i-1}
\end{equation*}
are called, respectively, the {\it start-point} and {\it
end-point} of the run, as well as of the cluster $\cC^1_i$.
To avoid confusion we sometimes  write more explicitly
$\alpha(\cC^1_i), \;\omega(\cC^1_i)$.

 By ${\text{\bf C}}_{1,1}$ we denote the set of clusters of level 1.
Let ${\text{\bf C}}'_{0,1}=\{\{x_i\}\colon   x_i \in \Ga \setminus
\cC \text{ for all } \cC \in {\text{\bf C}}_{1,1}\}$ and
${\text{\bf C}}_1={\text{\bf C}}_{1,1}\cup{\text{\bf C}}'_{0,1}$.
Note that $\text{\bf C}_{1,1}$ and  $\text{\bf C}'_{0,1}$ consist
of level 1 and level 0 clusters, respectively, and that the union
of all points in these clusters is exactly $\Ga$. We label the
elements of ${\text{\bf C}}_1$ in increasing order as $\cC_{1,j}, j \ge 1$.
For later use we also define $\text{\bf
C}_{0,1} = {\text{\bf C}}_0$. \noindent {\bf Notation.} In our
notation $\cC ^1_j$ denotes the $j^{th}$ level 1-cluster,
and $\cC_{1,j}$ denotes the $j^{th}$ element in $\text{\bf
C}_1$ (always in increasing order).

\medskip
\noindent {\bf Level k+1.} Let  $k \ge 1$ and assume that the
partitions ${\text{\bf C}}_{k'} = \{\cC_{k',j}\colon \; j \ge 1
\}$, and the masses of the $\cC_{k',j}$ have already been defined
for $k'\le k$, and satisfy the properties \eqref{2.z}-\eqref{2.two} and that
${\text{\bf C}}_{k}$ consists of clusters $\cC$ of
levels $\ell \in \{0,1,\dots,k\}$, i.e.,
\begin{equation}
\text{\bf C}_{k}\subset \cup_{\l=0}^k  \text{\bf C}_{\l,\l},
\label{2.zz}
\end{equation}
where for  $\l \ge 0$, ${\text{\bf C}}_{\l,\l}$ is the set of level
$\l$ clusters. We assume, as before, that the labeling goes in
increasing order of appearance. Define
\begin{equation}
{\text{\bf C}}_{k,k+1} = \{ \cC \in {\text{\bf C}}_k \colon  m
(\cC)\ge k+1 \}. \label{2.4e}
\end{equation}
Notice that ${\text{\bf C}}_{1,2}={\text{\bf C}}_{1,1}$ and
${\text{\bf C}}_{k,k+1}\subseteq\cup_{\ell=1}^k{\text{\bf
C}}_{\ell,\ell}$, if $k \ge 1$.

In the previous enumeration of ${\text{\bf C}}_{k}$, let $j_1 <
j_2 < \dots$ be the labels of the clusters in ${\text{\bf
C}}_{k,k+1}$, so that $\text{\bf C}_{k,k+1} = \{\cC_{k,j_1},
\cC_{k,j_2},\dots\}$. In ${\text{\bf C}}_{k, k+1}$ we consider
consecutive maximal $(k+1)$-runs, where we say that the clusters
$\cC_{k,j_s}, \cC_{k,j_{s+1}}, \dots \cC_{k,j_{s+n-1}}
\in {\text{\bf C}}_{k,k+1}$ form a {\it maximal $(k+1)$-run} of
length $n\ge 2$ if:
\begin{equation*}
d(\cC_{k,j_i}, \,  \cC_{k,j_{i+1}})\, < \, L^{k+1}, \; i=s,
\dots ,s+ n-2,
\end{equation*}
and in addition
\begin{equation*}
d(\cC_{k,j_i}, \,  \cC_{k,j_{i+1}})\, \ge \, L^{k+1}  \; \;
\begin{cases} \text{for} \; \; i=s-1, i= s+n-1, \; &\text{if} \; j_s>1 \\
\text{for} \; \; i= s+n-1, \; &\text{if} \; j_s=1.
\end{cases}
\end{equation*}

Again it is immediate that $P$-a.s. all $(k+1)$-runs are finite
and that infinitely many such runs exist. Again we can label them in
increasing order and write  $r^{k+1}_i = r^{k+1}_i (\cC_{k,j_{s_i}}, \cC_{k,j_{s_i+1}},$
$ \dots, \cC_{k,j_{s_i+n_i-1}})$ for the $i$-th $(k+1)$-run, for suitable
$s_i,n_i$ such that $n_i\ge 2$ and $s_i+n_i \le  s_{i+1}$ for all
$i$. ($s_i,n_i$ have nothing to do with those in the previous
steps of the construction.) We set
\begin{equation*}
\alpha^{k+1}_i = \alpha(\cC_{k,j_{s_i}})\quad \text{and}\quad
\omega^{k+1}_i =\om (\cC_{k,j_{s_i+n_i-1}}),
\end{equation*}
and call these the start-point and end-point of the run,
respectively. We define the span of the run
\begin{equation*}
\text{span}(r^{k+1}_i)=[\al^{k+1}_i, \, \om^{k+1}_i],
\end{equation*}
and  associate to it a cluster $\cC^{k+1}_i$ of level
$k+1$, defined as
\begin{equation*}
  \cC^{k+1}_i = \text{span}(r^{k+1}_i)\cap \Ga.
\end{equation*}
%It is made up from the clusters $\cC_{k,j_{s_i}}, \cC_{k,j_{s_i+1}}, \dots, \cC_{k,j_{s_i + n_i-1}}$.
In this case, the clusters $\cC_{k,j_{s_i}}, \cC_{k,j_{s_i+1}},
 \dots \cC_{k,j_{s_{i}+n_i-1}}$ are called {\it constituents}
of $\cC^{k+1}_i$. To the cluster $\cC^{k+1}_i$ we
attribute the mass $m(\cC^{k+1}_i)$  by the following
rule:
\begin{equation}
m (\cC^{k+1}_i) = m (\cC_{k,j_{s_i}})+
\sum_{s=s_i+1}^{s_i+n_i-1} (m (\cC_{k,j_s}) - k) =
\sum_{s=s_i}^{s_i+n_i-1} m (\cC_{k,j_s}) - k(n_i-1).
\label{2.foura}
\end{equation}

The points $\alpha^{k+1}_i$ and $\omega^{k+1}_i $ will also be
called, respectively,
{\it start}- and {\it end-point} of the cluster $\cC^{k+1}_i$, and
are also written as $\alpha(\cC^{k+1}_i)$ and
$\omega(\cC^{k+1}_i)$.

By ${\text{\bf C}}_{k+1,k+1}$ we denote the set of all level
$({k+1})$ clusters. Take $\text{\bf C}'_{k,k+1} = \{ \cC \in
{\text{\bf C}}_{k}\colon \cC \cap \text{span}(r^{k+1}_i) =
\emptyset, \; i=1,2, \dots \}$. Finally we define ${\text{\bf
C}}_{k+1} := {\text{\bf C}}_{k+1,k+1} \cup \text{\bf C}'_{k,k+1}$.
We label the elements of ${\text{\bf C}}_{k+1}$ as $\cC_{k+1,j}, j \ge 1$,
in increasing order. Note that a cluster in
${\text{\bf C}}_k$ is also a cluster in ${\text{\bf C}}_{k+1}$ if
and only if it is disjoint from the span of each maximal
$(k+1)$-run of length at least 2. Thus ${\text{\bf C}}_{k+1}$ may
contain some clusters of level no more than $k$, but
 some clusters (of level $\le k$) in ${\text{\bf C}}_{k}$
no longer appear in ${\text{\bf C}}_{k+1}$ (or any ${\text{\bf
C}}_{k+j}$ with $j \ge 1$).

Note also that in the formation of a cluster of level $(k+1)$,
clusters of mass at most $k$ might be incorporated while taking
the span of a $(k+1)$-run; they form what we call {\it dust} or {\it
porous medium} of level at most $k-1$ in between the
constituents, which have mass at least $k+1$.

This describes the construction of the $\mathbf{C}_k$. We next show by
induction that
\begin{equation}
\mathbf{C}_k \text{ is a partition of $\Ga$ and $\mathbf{C}_k$ is a refinement of $\mathbf{C}_{k+1}$}
\label{nested}
\end{equation}
for $k \ge 0$. This is clear
for $k=0$, since $\mathbf{C}_0$ is just the partition of $\Ga$ into
singletons. If we already know \eqref{nested} for $0 \le k \le K$, then
it follows also for $k = K+1$ from the fact that clusters in $\mathbf{C}_{k+1}$
are formed from the clusters in $\mathbf{C}_k$ by combining the
consecutive clusters between the start- and end-point of
a maximal $(k+1)$-run into one cluster. Thus it takes a
number of successive clusters in $\mathbf{C}_k$ and combines them into one
cluster. This establishes \eqref{nested} for all $k$.

The definition of $\mathbf{C}_k$ shows that
\begin{equation*}
\mathbf{C}_{k,k} \subset \mathbf{C}_k \subset \mathbf{C}_{k,k} \cup \mathbf{C}'_{k-1,k}
\subset \mathbf{C}_{k,k} \cup \mathbf{C}_{k-1},
\end{equation*}
from which we obtain by
induction that \eqref{2.zz} holds, as well as
\begin{equation*}
\l(\cC) = \min\{k\colon\cC \in \text{\bf C}_k\}
\end{equation*}
for any $\cC \in \cup_{k\ge 1} \text{\bf C}_k$.

We use induction once more to show that for any $k \ge 0$
\begin{equation}
m(\cC) \ge \ell(\cC) +1  \text{ for any } \cC \in \cup_{0 \le\ell\le k} \mathbf{C}_\ell,
\label{2.fourc}
\end{equation}
and if $\cC^{k+1}_i$ is formed from the constituents
$\cC_{k,j_{s_i}},\dots,\cC_{k,j_{s_i+n_i-1}}$
with $n_i \ge 2$, then
\begin{equation}
m (\cC^{k+1}_i) \ge \max_{s_i \le s \le s_i+n_i-1}
m(\cC_{k,j_s})+n_i-1 >  \max_{s_i \le s \le s_i+n_i-1}
m(\cC_{k,j_s}).
\label{2.fourb}
\end{equation}
Indeed, \eqref{2.fourc} trivially holds for $k = 0$. Moreover, if
\eqref{2.fourc} holds for $k \le K$, then \eqref{2.fourb} for $k =
K$ follows from the rule \eqref{2.foura} (and $n_i \ge 2$). In turn,
\eqref{2.fourb} and \eqref{2.fourc} for $k \le K$ imply
\begin{equation}
m(\cC_i^{k+1}) \ge \max_{s_i \le s \le s_i+n_i-1}m(\cC_{k,j_s})+1,
\label{2.zy}
\end{equation}
and hence also \eqref{2.fourc} for $k = K+1$.

So far we have shown that ${\text{\bf C}}_{k+1}$ is a
partition of $\Ga$ which satisfies \eqref{2.z} and \eqref{2.one} with
$k$ replaced  by  $k+1$ (by the definition of ${\text{\bf
C}}'_{k,k+1}$ and induction on $k$).
We next show by an indirect proof that this is also true for
\eqref{2.two}. It is convenient to first prove the following claim:
\newline
 {\bf Claim.} If $t \ge 1$,
 $\cC \in  \cup_{j \ge 0}\text{{\bf C}}_{t+j}$ and
$\ell(\cC) \le t$, then we have (see definition \eqref{2.4e})
\begin{equation}
\cC \in \text{{\bf C}}_{s,s+1} \text{ for } \l(\cC) \le s \le
(m(\cC)-1)\land t. \label{2.fourff}
\end{equation}
To see this, define $\wh s$ as the
smallest $s \ge \l(\cC)$ for which $\cC \notin \text{{\bf
C}}_{s,s+1}$, and assume that $\wh s \le (m(\cC)-1)\land t$. Then
$m(\cC) \ge \wh s+1$, so that we must have $\cC \notin \text{{\bf
C}} _{\wh s}$. But also $\cC \in\text{{\bf C}}_{\wh s - 1,\wh s}
\subseteq \text{{\bf C}}_{\wh s-1}$. (Note that $\wh s = \l(\cC)$
cannot occur, because one always has $\cC \in \mathbf{C}_{\l,\l+1}$
for $\l = \l(\cC)$, by virtue of \eqref{2.fourc}.) Then it must be
the case that $\cC$ intersects $\text{span}(r_i^{\wh s})$ for some
$i$. In fact, by our construction, $\cC$ must then be a constituent
of some cluster in  $\text{{\bf C}}_{\wh s}$ corresponding to a
maximal $\wh s$-run of length at least 2. But then $\cC$ does not
appear in $\text{{\bf C}}_{\wh s +j}$ for any $j \ge 0$, and in
particular $\cC \notin \cup_{j \ge 0} \text{{\bf C}}_{t+j}$,
contrary to our assumption. Thus, $\wh s \le (m(\cC)-1) \land t$ is
impossible and our claim must hold.
\smallskip

We now turn to the proof of \eqref{2.two}. This is obvious for $k =
0$ or $k=1$. Assume then that \eqref{2.two} has been proven for some
$k \ge 1$. Assume further, to derive a contradiction, that $\cC'$
and $\cC''$ are two distinct clusters in $\text{{\bf C}}_{k+1}$
such that $\min\{m(\cC'), m(\cC'')\} \ge  r$ but $d(\cC',\cC'') <
L^r$  for some $r \le k+1$. Without loss of generality we take $r
= m(\cC') \land m(\cC'') \land (k+1)$. Let $\cC'$ and $\cC''$ have
level $\l'$ and $\l''$, respectively. Since these clusters belong
to ${\text{\bf C}}_{k+1}$ we must have $\max(\l',\l'') \le k+1$.
For the sake of argument, let $\l' \le \l''$. If $\l'=\l''=k+1$,
then $d(\cC', \cC'') \ge L^{k+1}$, because, by construction, two
distinct clusters of level $k+1$ have distance at least $L^{k+1}$.
In this case we don't have $d(\cC',\cC'') < L^r$, so that we may
assume $\l' < k+1$.

Now first assume that $r-1 \ge \max(\l',\l'')= \l''$. Since $r-1
\le k$ we then have by \eqref{2.fourff} (with $t = k$) that $\cC'$
and $\cC''$ both belong to $\text{{\bf C}}_{r-1,r}$. If the
distance from $\cC'$ to the nearest cluster in $\text{{\bf
C}}_{r-1,r}$ is less than $L^r$, then $\cC'$ will be a constituent
of a cluster of level $r$ and $\cC'$ will not be an element of
${\text{\bf C}}_{k+1}$. Thus it must be the case that the distance
from $\cC'$ to the nearest cluster in $\text{{\bf C}}_{r-1,r}$ is
at least $L^{r}$. A fortiori, $d(\cC', \cC'') \ge L^r$. This
contradicts our choice of $\cC', \cC''$.

The only case left to consider is when $r-1 < \max(\l',\l'') =
\l''$. Since $r-1 = ( m(\cC') -1) \land (m(\cC'')-1) \land k \ge \l'
\land \l'' \land k = \l'$ (by \eqref{2.fourc}; recall that $\l' <
k+1$ now) this means $\l' \le r-1 < \l''$. We still have as in the
last paragraph that $\cC' \in \text{{\bf C}}_{r-1,r}$, and that the
distance between $\cC'$ and the nearest cluster in $\text{{\bf
C}}_{r-1,r}$ is at least $L^r$.  By \eqref{2.one}
$\text{span}(\cC')$ and $\text{span}(\cC'')$ have to be disjoint.
For the sake of argument let us further assume that $\cC'$ lies to
the left of $\cC''$, that is, $\om(\cC') < \al(\cC'')$. We claim
that $\al(\cC'')= \al(\cC)$ for some cluster $\cC \in \text{{\bf
C}}_{r-1,r}$. Indeed, the start-point of a cluster of level $\wt
\l\ge 2$ equals the start-point of one of its constituents, which
belongs to $\text{{\bf C}}_{\wt \l -1,\wt \l} \subseteq \text{{\bf
C}}_{\wt \l-1}$. Repetition of this argument shows that $\al(\cC'')$
is also the start-point of a cluster $\cC$ which is a constituent of
some cluster $\wh {\cC}$ such that $s :=\l(\cC) \le r-1$ but $t+1 :=
\l(\wh {\cC}) \ge r$. In particular, $\cC \in \text{{\bf
C}}_{t,t+1}$, so that $\cC \in \mathbf{C}_t$ and $m(\cC) \ge t+1 \ge
r$.
%% and $\wh {\cC} \in \text{{\bf C}}_{t+1}$.
Thus $\l(\cC) \le r-1 \le (m(\cC) -1) \wedge t$.
It then follows from \eqref{2.fourff} that $\cC \in \text{{\bf
C}}_{r-1,r}$. As in the preceding case we then have
\begin{eqnarray*}
d(\cC',\cC'') &\ge& \al(\cC'') - \om(\cC')= \al(\cC) - \om(\cC')\\
&\ge &\text{ the distance from $\cC'$ to the nearest cluster
in \bf C}_{r-1,r}\ge L^r.
\end{eqnarray*}

Of course the inequality for $d(\cC',\cC'')$ remains valid if $\cC'$
lies to the right of $\cC''$, so that we have arrived at a
contradiction in all cases, and \eqref{2.two} with $k$ replaced by
$k+1$ must hold. This completes the proof of \eqref{2.two}.

\medskip
\noindent {\bf Construction of ${\text{\bf C}}_{\infty}$.} Observe
that each $x\in \Ga$ may belong to clusters of several levels, but
not to different clusters of the same level (see \eqref{2.one}). If
$\cC'$ and $\cC''$ are two clusters of levels $\l'$ and
$\l''$, respectively, with $\l' < \l''$, then
\begin{equation}
\text{span}(\cC') \cap \text{span}(\cC'') \ne \emptyset \text{
implies }
 \text{span}(\cC') \subseteq \text{ span}(\cC'').
\label{2.y}
\end{equation}

There will even exist a sequence $\cC_0= \cC', \cC_1, \dots, \cC_s,
\cC_{s+1} = \cC''$ such that $\cC_i$ is a constituent of $\cC_{i+1},
0 \le i\le s$. This follows from the fact that each $\mathbf{C}_k$
is a partition of $\Ga$ and that $\mathbf{C}_k$ is a refinement of
$\mathbf{C}_{k+1}$. In fact, each element of $\mathbf{C}_{k+1}$ is
obtained by combining several consecutive elements of
$\mathbf{C}_k$. (We allow here that an element of $\mathbf{C}_k$ is
already an element of $\mathbf{C}_{k+1}$ by itself.) In turn, we see
then from \eqref{2.zy} that $m(\cC'') > m(\cC')$. In particular, no
point belongs to two different clusters with the same mass. We shall
use this fact in the proof of the next lemma.

We define the random index
\begin{equation*}
\kappa (x) = \sup\{\l\colon x\in \cC \;{\text{ for some}}\;
\cC \in {\text{\bf C}}_{\l,\l}  \}.
\end{equation*}
If we allow the value $\infty$ for $\ka(x)$, then this index is
always well defined, since each $x \in \Ga$ belongs at least to
the cluster $\{x\}$ of level 0.

\begin{lemma}
\label{lemma1.1}
For $\de > 0$ and $3 \le L < (64 \de)^{-1/2}$ we have a.s. $\ka(x) <
\infty$ for all $x \in \Ga$.
\end{lemma}
Before proving this lemma we show how to use it for the
construction of ${\text{\bf C}}_{\infty}$. Lemma \ref{lemma1.1}
tells that for each $x \in \Ga$, there exists a cluster of level
$\kappa (x)\in \Bbb Z_+$ which contains $x$. This cluster is
unique, since the elements of  ${\text{\bf C}}_{k,k}$ are pairwise
disjoint. We call it
 the {\it maximal cluster of} $x$ and denote it by $\cD_x$.
Moreover, for $x,x' \in \Ga$, if $x' \in \cD_x$, then $\kappa
(x)= \kappa (x')$ and $\cD_x = \cD_{x'}$. Indeed, $\ka(x)
\ne \ka(x')$ would contradict \eqref{2.y} and the definition of
$\ka$, while $\kappa (x)= \kappa (x')$ but $\cD_x \ne \cD_{x'}$
is impossible by \eqref{2.one}.

Take $\hat x_1 = x_1 \in \Ga = \{ x_j\}_{j\ge 1} $ and define
$\cC_{\infty ,1} = \cD_{\hat x_1}$. Having defined $\cC_{\infty ,j} =
\cD_{\hat x_j}$ for $j=1,\dots,k$, we set $\hat
x_{k+1} = \min \{x_j \in \Ga \colon x_j \notin \cup_{i=1}^k \cC_{\infty ,i} \}$,
and  $\cC_{\infty ,k+1} = \cD_{\hat
x_{k+1}}$. Define ${\text{\bf C}}_{\infty} = \{ \cC_{\infty ,k}
\}_{k \ge 1}$. Clearly, $\Ga = \cup_{k \ge 1}\; \cC_{\infty
,k}$. It is also routine to check that ${\text{\bf C}}_{\infty}$
satisfies \eqref{2.onea} and \eqref{2.a}. As for \eqref{2.four}, this
follows from \eqref{2.two} and the fact that $\cD_x \in \text{\bf
C}'_{k,k+1} \subseteq \text{\bf C}_{k+1}$ for all $k \ge \ka(x)$
(by the definitions of $\ka(x)$ and $\text{\bf C}'_{k,k+1}$).
\medskip

\noindent {\it Proof of Lemma \ref{lemma1.1}.} $\kappa (x)=
+\infty$ can occur only if there exists an infinite increasing
subsequence of indices $\{k_i \}_{i\ge 1}$ such that the point $x$
becomes ``incorporated'' into some cluster of level $k_i$ for all
$i\ge 1$. We will show that
\begin{equation}
P\big( x \text{ belongs to an infinite sequence of
clusters}\big)=0. \label{2.3aa}
\end{equation}
Notice that each cluster of level $k$ necessarily has mass at least
$k+1$ and no point belongs to two different  clusters of the same
mass, as observed above. Setting
\begin{equation}
A_k(x)=[x \text { belongs to a cluster of mass } k], \label{2.4}
\end{equation}
we shall show that for each fixed $x$
\begin{equation}
P(A_k(x) \text { i.o. in } k)=0, \label{2.4a}
\end{equation}
which will prove \eqref{2.3aa}.

We will carry out the proof  in two steps. All constants $c_i$
below are strictly positive and independent of $k$. First we
estimate the probability that a given point $z\in \Bbb Z_+$ is the
start-point of a cluster of mass $k\ge 2$. Specifically, we show
that
\begin{equation}
 P \big(\exists\, \cC \in \bigcup_{\l \ge 1}\text{\bf C}_\l
\colon \al(\cC)=z, m(\cC)=k\big) \le c_1 e^{-c_2 k} \label{2.5}
\end{equation}
for some strictly positive constants $c_1, c_2 >0$ and for each
 $k$ and each $z$. In fact we can take $c_2 > \log L$ so that
\begin{equation}
c_1(L^k+1)e^{-c_2k} \le 2c_1e^{-c_3k} \label{2.7}
\end{equation}
for some constant $c_3 > 0$. This is the most involved part of the
proof.
In the second step of the proof we show  that if $\cC \in
\cup_{\ell \ge 1} \text{\bf C}_\ell$ and $m({\cC}) = k$, then
\begin{equation}
\text{diam} \big( {\cC} \big) \; \le  3L^{k-1}. \label{2.6}
\end{equation}
Due to \eqref{2.6} we will have the following inclusion:
\begin{equation}
A_k(x) \subseteq \big[\exists z\in [x- L^k, x]\colon \al({\cC})=z, \; \text{for some}\; \cC \in \cup_{\ell \ge 1} \text{\bf
C}_\ell \; \text{ and  }\; m(\cC)=k \big]. \label{2.8}
\end{equation}
From \eqref{2.5}, \eqref{2.7} and \eqref{2.8} we will have
\begin{equation}
P\big(A_k(x)\big) \le  c_1(L^k +1) e^{-c_2 k} \le 2c_1e^{-c_3k},
\label{2.9}
\end{equation}
which, by the Borel-Cantelli lemma, gives  \eqref{2.4a}, and so
\eqref{2.3aa}.

Let us now prove \eqref{2.5}, where $k \ge 2$ and $z \in \Bbb Z_+$.
To any given cluster ${\mathcal C} \in \cup_{\ell\ge 1}{\mathbf C}_\ell$ we
associate a ``genealogical weighted tree". It describes the
successive merging processes which lead to the creation of $\mathcal{C}$,
i.e., it tells the levels at which some clusters form runs,
merging into larger clusters and how many constituents entered
each run, down to level 1, and finally the masses of such level 1
clusters. So we represent it as a tree with the root corresponding
to ${\mathcal C}$; the leaves correspond to clusters of level 1, which
are the basic constituents at level 1. This weighted tree gives
the basic information on the cluster, neglecting what was
incorporated as ``dust", on the way.

More formally, we construct the tree iteratively. The root of the
tree corresponds to the cluster $\mathcal C$. If this cluster is of
level 1, the procedure is stopped. For notational consistency such
a tree will be called a 1-leaf tree. To the root we attribute the
index 1, as well as another index which equals the mass of the
cluster.

If the resulting cluster $\mathcal C$ is of level $\ell\,>1$,
 we attribute to the root the
index $\ell$ and add to the graph $n_1$ edges (children) going out
from the root, where $n_1 \ge 2$ is the number of constituents
which form the $\ell$-run leading to $\mathcal{C}$. Each
 endvertex of a newly added edge will correspond to a constituent of
the run, i.e., if $\mathcal C$ has constituents
 ${\mathcal C}_{\ell-1,i_1},\dots, {\mathcal C}_{\ell-1,i_{n_1}} \in {\mathbf
C}_{\ell-1,\ell}$, for suitable $i_1,\dots, i_{n_1}$, then there
is a vertex at the end of an edge going out from the root
corresponding to $\mathcal{C}_{\l-1, i_j}$ for each $j
 = 1, \dots, n_1$. If the constituent corresponding to a given
endvertex is a level 1-cluster, the procedure at this endvertex is
stopped (producing a leaf on the tree), and to this leaf we
attribute an index, which equals the mass of the corresponding
constituent.

If a given endvertex corresponds to a cluster $\wt {\mathcal C}$ of
level $\ell'$ with $1< \ell' < \ell$, then to this endvertex we
attribute the index $\ell' $, and add to the graph $n_2$ new edges
going out of this endvertex, where $n_2$ is the number of
constituents of $\wt {\mathcal C}$ in ${\text {\bf C}}_{\ell'-1,\ell'}$
which make up $\wt {\mathcal{C}}$.

The procedure continues until we reach the state that all
constituents corresponding to newly added edges are level 1
clusters. In this way we obtain a tree with the following
properties:

 i) each vertex of the tree has either 0 or at least two offspring;
in case of 0 offspring we say that the vertex is a {\it leaf} of
the tree. Otherwise we call it a {\it branch node}.

 ii) to each branch node $x$ we attribute an index $\ell_x$; these
indices are strictly decreasing to 1 along any self-avoiding path
from the root to a leaf of the tree.

iii) to each leaf is associated a mass $m \ge 1$. This defines a
map

\begin{equation*}
\gamma\colon \mathcal C \in \cup_{\ell\ge 1}{\mathbf C}_\ell \mapsto
\gamma(\mathcal C) \equiv (\Upsilon (\mathcal C) , \bar l(\mathcal C) , \overline m
(\mathcal C)),
\end{equation*}
where $\Upsilon (\mathcal C)$ is a finite tree with $\mathcal L (\Upsilon
(\mathcal C))$ leaves and $\mathcal N (\Upsilon (\mathcal C))$ branching
nodes. We use the following notation:
\medskip

\noindent $\overline l(\mathcal C)= \{ \ell_1(\mathcal C), \dots , \ell_{\mathcal N
(\Upsilon (\mathcal C))}(\mathcal C) \}$  is a multi-index with one
component for each branching node of $\Upsilon (\mathcal C)$, which
indicates the level at which  branches ``merge" into the cluster
corresponding to the node;

\medskip
\noindent $\overline m (\mathcal C)= \{ m_1(\mathcal C), \dots , m_{\mathcal L
(\Upsilon (\mathcal C))}(\mathcal C) \}$ a multi-index with one component
for each leaf of $\Upsilon (\mathcal C)$, which gives to the mass of
the cluster corresponding to the leaf;

\medskip
\noindent $\bar n (\mathcal C)= \{ n_1(\mathcal C), \dots , n_{\mathcal N
(\Upsilon (\mathcal C))}(\mathcal C) \}$  is a multi-index with one
component for each vertex of $\Upsilon (\mathcal C)$, which gives the
degree of the vertex minus 1. Note that $\bar n (\mathcal C)$ is determined by
$\Up(\mathcal C)$.

To lighten the notation, we will omit the argument $\mathcal C$ in
situations where confusion is unlikely. Thus we occasionally write
$\gamma(\mathcal C) \equiv (\Upsilon, \bar l, \overline m)$ instead of
$(\Upsilon (\mathcal C) , \bar l(\mathcal C) , \overline m (\mathcal C))$.

In order to prove \eqref{2.5} we decompose the event
\begin{equation}
\big [\exists\, \mathcal C \in \cup_{\ell\ge 1}{\mathbf C}_\ell
\colon \alpha(\mathcal C) = z, m(\mathcal C)=k\big] \label{2.12}
\end{equation}
according to the possible values for $\gamma(\mathcal C)$; we shall
abbreviate the number of leaves of $\Up(\mathcal C)$ by $\mathcal L$. Since the
resulting cluster $\mathcal C$, obtained after all merging process ``along
the tree'', has mass $k$, it imposes the following relation
between the multi-indices $\overline m$ and $\bar l$:
\begin{equation}
\sum_{i=1}^{\mathcal L} m_i - \sum_{j=1}^{\mathcal N} (n_j-1)(\ell_j-1) = k
\label{2.12aa}
\end{equation}
Here the first sum runs over all leaves, while the second sum runs
over all branching nodes.
%% (or over all nodes, since the leaves have $n_j = 1$ anyway).
This relation follows from \eqref{2.foura} by induction on the number of
vertices, by writing the tree as the ``union'' of the root and
the subtrees which remain after removing the root. We note that
$\Up$ also has to satisfy
\begin{equation}
\sum_{j=1}^\mathcal N  (n_j-1) = \mathcal L-1, \label{2.14aa}
\end{equation}
because it is a tree, as one easily sees by induction on the
number of leaves. This implies the further restriction
\begin{equation*}
\sum_{i=1}^\mathcal L m_i \ge k+\mathcal L-1,
\end{equation*}
because $\l_j \ge 2$ in each term of the second sum in
\eqref{2.12aa} (recall that we stop our tree construction at each
node corresponding to a cluster of level 1). Thus the probability
of the event in \eqref{2.12} equals to

\begin{equation}
\sum_{r\ge 1} \sum_{{\substack{\Upsilon\!\colon\! \\ \mathcal L (\Upsilon) =r}}}
 {\sum_{\bar l, \overline m}}^{\Upsilon}  P\big(\exists\, \mathcal C \in
\cup_{\ell\ge 1}{\mathbf C}_\ell \colon z=\al(\mathcal C), m(\mathcal
C)=k, \gamma (\mathcal C)
 = (\Upsilon, \bar l, \overline m)\big ),
\label{2.12bb}
\end{equation}
where the third sum $\sum_{\bar l, \overline m}^{\Upsilon}$ is
taken over all possible values of $\bar l, \overline m$,
satisfying \eqref{2.12aa}.

A decomposition according to the value of the sum $\sum_i m_i$,
shows that the expression \eqref{2.12bb} equals
\begin{equation}
\sum_{{r\ge 1}} \sum_{{\substack{\Upsilon\!\colon\! \\ \mathcal L (\Upsilon) =r}}} \sum_{{s \ge r-1}} \sum_{{\substack{\overline m\!\colon\!\! \\ \sum_i m_i  \\
= k +s }}} \sum_{ \bar l}  P\left( \exists \;  \mathcal C \in
\cup_{_{\ell\ge 1}}{\mathbf C}_\ell \colon \alpha(\mathcal C)=z, \; m({\mathcal
C}) = k , \gamma (\mathcal C) = (\Upsilon, \bar l, \overline m)\right),
\label{2.12cc}
\end{equation}
the sum $\sum_{\bar l}$ being taken over possible choices
of $\bar l$ such that $\sum_j (n_j -1)(\ell_j-1)  = s$. The multiple
sum in \eqref{2.12cc} can be bounded from above by
\begin{equation}
 \sum_{r\ge 1} \sum_{\substack{\Upsilon\!\colon\! \\ \mathcal L (\Upsilon) =r}}
\sum_{s \ge r-1} \sum_{\substack{\overline m\!\colon\! \\ \sum_i m_i  = k +s}}
\sum_{ \bar l} \delta^{k+s} L^{k+2s}. \label{2.14}
\end{equation}
Indeed, for fixed $z, k$ and $(\Up, \bar l, \overline m)$, the
probability
\begin{equation*}
P \left(\exists\;  \mathcal C \colon \alpha(\mathcal C)=z, \;
m ({\mathcal C}) = k , \gamma (\mathcal C) = (\Upsilon, \bar l, \overline m)
\right)
\end{equation*}
is easily estimated by the following argument: the
probability to find a level 1 cluster  of mass $m_i$ which
corresponds to some leaf of the tree, and which starts at a given
point $x$, is bounded from above by $\delta^{m_i}L^{m_i -1}$. Indeed,
such a cluster has to come from a maximal level 1 run $x_s,
x_{s+1},\dots, x_{s+m_i-1}$ of elements of $\Gamma$, with $x_s = x$
and $x_{j+1}-x_j \le L$ for $j=s,\dots,s+m_i-2$. The number of
choices for such a run is at most $L^{m_i-1}$, and given the
$x_j$, the probability that they all lie in $\Gamma$ is $\delta^{m_i}$.
Similarly, the probability to find two level 1 clusters of mass
$m_{i_1}$ and $m_{i_2}$ which merge at level $\ell_j$ can be bounded
above by $\delta^{m_{i_1}}L^{m_{i_1} -1}\delta^{m_{i_2}}L^{m_{i_2}
-1}L^{\ell_j}$. The factor $L^{\ell_j}$ here is an upper bound for the
number of choices for the distance between the two clusters; if
they are to merge at level $\ell_j$, their distance can be at most
$L^{\ell_j}$. Iterating this argument we get that
\begin{eqnarray*}
P\left(\exists \mathcal C \in \cup_{\ell\ge 1}{\mathbf C}_\ell
\colon \alpha(\mathcal C)=z, \; m({\mathcal C}) = k , \gamma (\mathcal C) =
(\Upsilon, \bar l, \overline m)\right)\le \delta^{\sum_i m_i}
L^{\sum_i (m_i-1)} L^{\sum_j (n_j-1)\l_j},
\end{eqnarray*}
and taking into account that
\begin{eqnarray*}
\sum_i (m_i-1)  + \sum_j (n_j-1)\ell_j =k+s + s - r + \sum_j (n_j-1),
\end{eqnarray*}
%\begin{eqnarray*}
%\sum_i (m_i-1)  + \sum_j (n_j-1)\ell_j &=& \! \! \! \sum_i m_i + \sum_j
%(n_j-1)(\ell_j-1) -\sum_i 1 + \sum_j (n_j-1) \\  &=& k+s + s - r +
%\sum_j (n_j-1),
%\end{eqnarray*}
as well as \eqref{2.14aa}, we get the bound \eqref{2.14}.

The number of terms in the sums  over $\overline m$
and $\bar l$ in \eqref{2.14} are respectively bounded by $2^{k +s}$ and $2^{s}$
(since $\sum_j(\ell_j-1) \le \sum_j(n_j-1)(\ell_j-1)= s$ and $\ell_j \ge
2$). Thus we can bound \eqref{2.14} from above by
\begin{eqnarray}
\label{2.16}
 &\sum_{r\ge 1} \sum_{\Upsilon\!\colon\! \mathcal L (\Upsilon) =r}
 \sum_{s \ge r-1}
2^{k +s} 2^{s}  \delta^{k+s} L^{k+2s} \nonumber\\ \nonumber
&\le   (2\delta L)^k \sum_{r\ge 1} \sum_{\Upsilon \!\colon \mathcal L (\Upsilon) =r} \sum_{s \ge r-1}
  (4 \delta L^2)^{s}\\
&\le   (2\delta L)^k \sum_{r\ge 1} \sum_{\Upsilon \!\colon\! \mathcal L
(\Upsilon) =r}
  \frac{(4 \delta L^2)^{r-1}}{1-4 \delta L^2},
\end{eqnarray}
provided we take $4 \delta L^2 <1 $. Now the number of planted plane
trees of $u$ vertices is at most $4^u$ (see \cite{HPT}). Our trees have $r$ leaves, but all vertices which
are not leaves have degree at least 3 (except, possibly, the
root). Thus, by virtue of \eqref{2.14aa}, these trees have at most
$2r$ vertices. The number of possibilities for $\Up$ in the last
sum is therefore at most $\sum_{u=r+1}^{2r} 4^u \le \frac43 4^{2r}
\le 2\cdot 4^{2r}$. It follows that \eqref{2.16} is further bounded
by
\begin{equation*}
2\frac{(2\delta L)^k}{1-4 \delta L^2} \sum_{r\ge 1} 4^{2r}
  (4 \delta L^2)^{r-1}
=   \frac{32(2\delta L)^k}{1-4 \delta L^2} \sum_{r\ge 1}
  (64 \delta L^2)^{r-1}.
%\label{2.17}
\end{equation*}
If we take $64\delta L^2 <1 $, this can be bounded by
$$
\frac{32(2\delta L)^k}{(1-4 \delta L^2)(1-64\delta L^2)},
$$
which proves \eqref{2.5} and \eqref{2.7} with $c_2 = -\log(2\delta) -
\log  L > \log L$ for our choice of $\delta,L$.
It remains to show \eqref{2.6}. It is trivially correct for $k=1$;
in fact a cluster of mass 1 has to be a singleton by
\eqref{2.fourb}. We will use induction on $k$. Assume \eqref{2.6}
holds for all clusters with mass at most $k-1$, where $k\ge 2$.
Let $\mathcal C$ be a cluster with $m(\mathcal C)=k$ and of level $\ell$.
Thus $\mathcal C \in {\mathbf C}_{\ell,\ell}$, and $1 \le \ell \le
k-1$, by virtue of \eqref{2.fourc}. If $\ell=1$ then diam$(\mathcal C)
\leq (k-1)L < 3L^{k-1}$, for $k\ge 2$, provided we take $L \ge 2$.\
If $\ell \ge 2$, then there exist $n\ge 2$, and $\mathcal C_{\ell-1,
i_1}, \dots \mathcal C_{\ell-1, i_n} \in {\mathbf C}_{\ell-1
,\ell}$ such that $\mathcal C$ is made up from the constituents $\mathcal
C_{\ell-1 , i_1}, \dots \mathcal C_{\ell-1 , i_n}$ (where, for
simplicity, we have omitted the indication of the level of the
constituents). If $m_j = m (\mathcal C_{\ell -1, i_j})$, then $m_j \ge
\ell$ (by \eqref{2.fourc}), and from \eqref{2.foura} we see that $m_j
\le k-n+1$ for each $j$. From this and the induction hypothesis we
get
\begin{equation}
\text{diam}(\mathcal C) \le \sum_{j=1}^n {\text{diam}}(\mathcal C_{\ell-1,
i_j}) + (n-1)L^{\ell} < 3nL^{k-n}+(n-1)L^{k-n+1}\le 3L^{k-1}
\label{2.25}
\end{equation}
for all $L\ge 3$ and $n \ge 2$. This proves \eqref{2.6} and the
lemma. \qed

\medskip

Note that $\text{\bf C}_{\infty}$ depends on the collection $\Ga$
only. We shall occasionally write $\text{\bf C}_{\infty}(\ga)$ for
the partition $\text{\bf C}_{\infty}$ at a sample point with $\Ga
= \ga$. We further define
\begin{equation}
\chi (\ga) = \inf \{ k \ge 0\colon\;  d(\cC,  0 )\ge L^{m (\cC)}\text{ for all $ \cC \in \text{\bf C}_{\infty}(\ga)$ with }
 m (\cC) > k \},
\label{2.25z}
\end{equation}
and set $\chi(\ga)=\infty$ if the above set is empty.

The preceding  proof has the following corollary:
\begin{coro}
\label{cor1.2}
Under the conditions of Lemma \ref{lemma1.1} we have $\chi (\ga) < +\infty$ a.s.
\end{coro}

\noindent In fact, \eqref{2.5} and \eqref{2.7} show that
\begin{eqnarray}
P(\exists \; \cC_{\infty,j} \colon m(\cC_{\infty,j}) > k, d(\cC_{\infty,j}, 0) < L^{m(\cC_{\infty,j})}) \le \sum_{m >  k} c_1L^me^{-c_2m} \to 0 \text{
as } k \to \infty.
\end{eqnarray}

In the next lemma $P(A|\ga)$ denotes the value of
the conditional probability of an event $A$ with respect to the
$\si$-field generated by $\Ga$ on the event $\{\Ga = \ga\}$.

\begin{lemma}\label{lemma2.3} Under the conditions of Lemma \ref{lemma1.1}
it suffices for Theorem \ref{main} to prove that
\begin{equation}
P(C_0 \text{ is infinite}|\ga) > 0 \text{ a.s. on the event
}\{\chi(\ga) = 0\}. \label{2.30}
\end{equation}
\end{lemma}

\begin{proof} We shall show here that
\begin{equation}
P(\chi(\Ga)= 0) > 0. \label{2.31}
\end{equation}
Together with \eqref{2.30}, this will imply that
\begin{eqnarray*}
P(C_0 \text{ is infinite})&\ge& P(C_0 \text{ is infinite and
}\chi(\Ga)
 =0)\\\nonumber &=& \int_{\{\chi(\ga) = 0\}} P(C_0 \text{ is infinite}|\ga)P(\Ga \in
d\ga) > 0,
\end{eqnarray*}
and hence $P(C_0 \text{ is infinite}) > 0$. This will also prove
\eqref{1.1}, because $\{P(C_0 \text{ is infinite}|\xi) > 0\}$
is a tail event (in $\xi$), and \eqref{2.30}, \eqref{2.31} imply \eqref{1.1}.

%(conditioning on the $\si$-field
%generated by the $\xi_i$ is of course the same as conditioning on
%$\si$-field generated by $\Ga$).

\medskip

Thus, we merely have to prove \eqref{2.31}. Now if $\chi(\ga)$ is
finite and non-zero, then there exists a unique cluster $\cC^\ast \in \text{\bf C}_{\infty}(\ga)$ such that $m (\cC^\ast)
= \chi(\ga)$ and $d(\cC^\ast, 0) < L^{\chi(\ga)}$. The
existence of $C^*$ follows at once from the definition of $\chi$.
For the uniqueness we observe that if two such clusters, say
$\cC'$ and $\cC''$, would exist, then they would have to satisfy
$d(\cC', \cC'') < L^{\chi(\ga)} = L^{\min\{m(\cC'),m(\cC'')\}}$, which
contradicts \eqref{2.five} by virtue of the assumption
$\cC', \cC'' \in \mathbf{C}_\infty$.

We use $\cC^*$ to construct a new environment $\wt \ga$
corresponding to the following sequence $\{\wt \xi\}_{i \ge 0}$ of
zeroes and ones:
\begin{equation*}
\text{if } \chi(\ga) = 0, \text{ then }\wt \xi_i = \xi_i \text{
for all } i \ge 0;
\end{equation*}
\begin{equation}
\text{if } 0 < \chi(\ga) < \infty, \text{ then }\wt \xi_i =
\begin{cases}
0 &\text{ if } i \le \om(\cC^*)\\ \xi_i &\text{ if } i >
\om(\cC^*).
\end{cases}
\label{2.39}
\end{equation}
We shall now show that
\begin{equation}
\chi(\wt \ga) = 0. \label{2.37}
\end{equation}
Of course we only have to check this in the case $0 <
\chi(\ga) < \infty$. We claim that in this case all clusters in
$\mathbf{C}_\infty(\wt \ga)$
are also clusters in $\mathbf{C}_\infty(\ga)$ (which
are located in $[\om(\cC^*)+1, \infty)$) and the masses of such a
cluster in the two environments $\ga$ and $\wt \ga$ are the same.
To see this we simply run
through the construction of the clusters in $\cup_{\l \ge 1}\mathbf{C}_\l$ in the environment $\wt \ga$, until
there arises a difference between these this construction and the
construction in the environment $\ga$. More precisely, we apply
induction with respect to the level of the clusters. Clearly any
cluster of level 0 in $\wt \ga$ is simply a single point of $\Ga$
which lies in $[\om(\cC^*)+1, \infty)$, and has mass 1. This is also
a cluster of level 0 and mass 1
in $\ga$. Assume now that we already know that any cluster in $\wt
\ga$ of level $\le k$ is a cluster of $\ga$ of level $k$ and located
in $[\om(\cC^*)+1, \infty)$ and with the same mass in $\ga$ and $\wt \ga$.

Since $\xi_i = 0$ for $i \le \om(\cC^*)$ in  the environment $\wt \ga$,
the span of any $(k+1)$-run in $\wt \ga$ has to be contained in
$[\om(\cC^*)+1,\infty)$. Therefore the span of any cluster of level $k+1$ in
environment $\wt \ga$ also has to be contained in $[\om(\cC^*)+1,\infty)$.
In addition, since the two environments $\ga$
and $\wt \ga$ agree in this interval, a difference in the
constructions or masses of some cluster of level $k+1$
can arise only because in $\ga$ there
is a $(k+1)$-run which contains clusters of level $k$ which lie in
$[\om(\cC^*)+1, \infty)$ as well as clusters which intersect
$[0, \om(\cC^*)]$.
But then these  clusters of level $k$ will be constituents of a single
$(k+1)$-cluster, $\cC \in \mathbf{C}_\infty(\ga)$ say. span($\cC$)
has to contain points of both
$[0,\om(\cC^*)]$ and of $[\om(\cC^*)+1, \infty)$ in $\ga$.
Consequently, span($\cC)$ has to contain both points $\om(\cC^*)$
and $\om(\cC^*)+1$. Since $\om(\cC^*) \in \cC^*$ we then have from
\eqref{2.y} that span($\cC^*) \subset \text{span}(\cC)$ and $\cC^* \ne
\cC$ (because $\om(\cC^*)+1 \notin \text{span} (\cC^*$)). But no such $\cC$
can exist, because $\cC^* \in \mathbf{C}_\infty$.

This establishes our last claim. Now, by definition of $\chi$,
\eqref{2.37} is equivalent to
\begin{equation}
d(\cC,0) \ge L^{m(\cC)} \label{2.38}
\end{equation}
for all clusters $\cC$ in $\mathbf{C}_\infty(\wt \ga)$. In view of
our claim this will be implied by \eqref{2.38} for all
clusters $\cC$ in $\mathbf{C}_\infty(\ga)$ located in $[\om(\cC^*),
\infty)$. Now, if $\cC$ is such a cluster with $m(\cC) \le m(\cC^*)$,
then \eqref{2.38} holds, because, by virtue of \eqref{2.five},
\begin{equation*}
d(\cC,0)= \al(\cC) \ge \al(\cC) - \om(\cC^*) = d(\cC, \cC^*) \ge
L^{m(\cC)}.
\end{equation*}
On the other hand, if $m(\cC) > m(\cC^*) = \chi(\ga)$, then the
definition of $\chi$ shows that we have $\al(\cC) \ge L^{m(\cC)}$.
This proves \eqref{2.38} in all cases, and therefore also proves
\eqref{2.37}.

We now have
\begin{equation*}
1 = P(\chi(\Ga) < \infty ) \le P(\chi(\Ga) = 0) + \sum_{n
=0}^\infty P(\om(\cC^*) =n,\chi(\Ga^{(n)})=0),
\end{equation*}
where $\om(\cC^*)$ is as described above with $\ga$ denoting the
value of $\Ga$, and if $\Ga$ corresponds to the sequence
$\{\xi_i\}$, then $\Ga^{(n)}$ corresponds to the sequence
$\xi_i^{(n)}$ given by
\begin{equation*}
\xi_i^{(n)} = \begin{cases} 0 &\text{ if } i\le n\\
\xi_i & \text{ if } i >n.
\end{cases}
\end{equation*}
Thus, either $P(\chi(\Ga)=0 ) >0$ or there is some non-random $n
\in \Bbb Z_+$ for which $P(\chi(\Ga^{(n)})= 0) >0$. However,
\begin{eqnarray*}
P(\chi(\Ga) = 0)  &\ge& P(\chi(\Ga) = 0,\; \xi_i = 0 \text{
for } 0 \le i \le n)\\
&=& P(\chi(\Ga^{(n)}) = 0,\; \xi_i = 0 \text{ for }
0 \le i \le n)\\
&=& P(\xi_i = 0 \text{ for } 0 \le i \le n)P(\chi(\Ga^{(n)}) = 0)
\end{eqnarray*}
(since $\Ga^{(n)}$ is determined by $(\xi_i;i > n)$).
This proves the validity of \eqref{2.31} and concludes the argument.\end{proof}

\section {Construction of renormalized lattices. Step 2: layers and sites}
\label{lattice}

From now on  we restrict ourselves to the set of environments
$\ga$ such that $\chi (\ga) =0$. In this section we construct a
sequence of partitions $\{\bH^k\}_{k\ge 0}$ of $\widetilde{\Bbb Z}^2_+$
into horizontal layers, which will be used to define renormalized sites.
As in the preceding section we will do the construction in a recursive way.
The elements of $\bH^k$ are called {\it $k$-layers}, and they will be
associated to the clusters of level at most $k$.

The horizontal layers in correspondence with the span of each cluster of level at most $k$ and
mass greater than $k$ will be called {\it bad} $k$-layers. The other $k$-layers will be called
{\it good} $k$-layers, and are further subdivided into good layers of {\it type 1} and of
{\it type 2}:  The good $k$-layers of type 1 ``contain"\footnote{Identifying a cluster with
the horizontal layer whose projection on the second coordinate is the cluster.} a cluster of level at most $k$ but with mass equal
to $k$, and the remaining good $k$-layers will be considered of type 2. Thus good $k$-layers
of type 2 are ``disjoint"\footnote{idem} from the clusters of level at most $k$ with mass at least $k$.
Thus, bad layers are those which contain the largest (in the sense of mass) clusters of
a new level. These are the most difficult for a percolation path to traverse, and have to
be treated differently from the others (as we shall see later on). The good $k$-layers of
type 2 correspond to the smallest clusters (in the sense of mass).

The good $k$-layers have ``height'' of
order $L^k$ (see \eqref{3.50} below). However, the bad layers can
have different heights. Between two successive bad $k$-layers there
are at least two good $k$-layers of type 2, but not necessarily
any good $k$-layer of type 1. The total number of good $k$-layers
between successive bad $k$-layers is also variable. This is
reflected in the many cases covered in the following definitions.
Checking some of the stated properties requires somewhat tedious
verification of the different cases separately, but we have found no
way to avoid this. Lemma \ref{lemma3.1} and the remark following it summarize
some further properties of the partitions $\{\bH^k\}$.

\smallskip

We shall take $L\ge 12$ and divisible by 3, with $\de$ small enough
for the assumptions of Lemma \ref{lemma1.1} to be verified.

\medskip
\noindent {\bf Step 0.} To begin, we define 0-layers  in the
following way:
 $H_{j}^0=H_j = \{(x,y)\in\widetilde{\Bbb Z}^2_+\colon y=j\}, j \ge 0$.
If $j\in \Ga$ we say that the 0-layer $H_{j}^0$ is {\it bad}, and
otherwise it is called {\it good}. These names are justified by
the fact that sites which belong to good 0-layers are open with
large probability (namely, $p_G$), and  sites which belong to bad
0-layers are open with small probability (namely, $p_B$).
\medskip

\noindent {\bf Step 1.}  Let ${\text{\bf C}}_{1} = \{ \cC_{1
,j} \}_{j\ge 1} $.  We recall that  a.s. each cluster $\cC_{1
,j}$ is finite, and $\al(\cC_{1 ,j})$ and $\om(\cC_{1 ,j})$
are respectively, its start- and end-points. Due to property
\eqref{2.two} we have that $\al(\cC_{1 ,j+1})-\om(\cC_{1
,j})\ge L$, and since we assumed $\chi(\ga)=0$, also $\al(\cC_{1 ,1})\ge L$. We set $\widetilde \al^{1}_{0,0}=0, \widetilde
\om^{1}_{0,0}= 1$,
\begin{equation}
\widetilde b_1= \Big\lfloor \frac{\al(\cC_{1 ,1}
)}{L/3}\Big\rfloor,\quad \widetilde b_j=\Big\lfloor \frac{\al(\cC_{1 ,j} ) -  \om(\cC_{1 ,j-1})}{L/3}\Big\rfloor \quad
\text{for  } j\ge 2,  \label{3.3}
\end{equation}
($\lfloor \cdot \rfloor$ denotes the integer part; in particular
$\widetilde b_j\ge 3$), while
\begin{equation*}
b_j = \begin{cases}\widetilde b_j , \quad &\text{if} \quad m (\cC_{1
,j}) =1;\cr \widetilde b_j +1 , \quad &\text{if} \quad m (\cC_{1 ,j}) > 1,
\end{cases}
\end{equation*}
for  $j \ge 1$. Also
%%PLEASE EXPLAIN WHY YOU ADD A 3 IN FIRST LINE OF NEXT DEFINITION. AN
%%INEQUALITY SIGN INSERTED ON SECOND LINE
\begin{equation}
\widetilde \om^{1}_{j,0} = \begin{cases}   \om(\cC_{1 ,j})+3, \quad &
\text{if} \quad m (\cC_{1 ,j}) =1; \cr  \om(\cC_{1 ,j}),
\quad & \text{if} \quad m (\cC_{1 ,j}) >1, \end{cases} \label{3.1}
\end{equation}
for $j \ge 1$.

Now, if $j\ge 0,1\le i \le \widetilde b_{j+1}-1$, we set
\begin{equation}
\widetilde \om^{1}_{j,i} =\begin{cases} iL/3,&\quad \text{if } j=0;\\
\om(\cC_{1,j}) + iL/3,&\quad \text{if } j \ge 1.
\end{cases}
\label{3.4a}
\end{equation}
Note that this definition holds for $i \le \wt b_{j+1}-1$ only and
not necessarily for $i = b_{j+1}-1$. In fact, if $j \ge 0$ and
$m(\cC_{1,j+1})>1$ we make the special definition
\begin{equation}
\widetilde \om^{1}_{j,b_{j+1}-1} =\alpha(\cC_{1,j+1}) -1.
\label{3.4b}
\end{equation}
For $j\ge 1$, we also define:
\begin{equation}
\widetilde \al^{1}_{j,0} = \begin{cases}   \widetilde
\om^{1}_{j-1,b_{j}-1}+1, \quad & \text{if} \quad m (\cC_{1 ,j})
=1, \cr \al(\cC_{1 ,j}), \quad & \text{if} \quad m (\cC_{1
,j})> 1.
 \end{cases}  \label{3.2}
\end{equation}
Finally,  let
\begin{equation}
\wt \al^{1}_{j,i} =  \widetilde \om^{1}_{j,i-1}+1, \quad
\text{for}\quad j \ge 0,\,1 \le i \le b_{j+1}-1. \label{3.44}
\end{equation}
For consistency with the notation to be introduced in the next
step we write $b^1_j$ and $\wt b^1_j$ for $b_j$ and $\wt b_j$,
respectively. We also write $\wt \cC_{t,j}, j = 1, 2, \dots$ for
the clusters in $\bold C_t$ of mass at least $t$, in increasing
order. For $t=1$, $\wt \cC_{1,j} = \cC_{1,j}$.

\smallskip

From the facts that $L \ge 12$, property \eqref{2.two} for $k=1$,
and $\alpha(\cC_{1,1})\ge L$, we see that the following properties
hold for $t=1$
\begin{equation}
\widetilde\al^{t}_{j,i} \le \widetilde \om^{t}_{j,i} \text{ for }
j \ge 0,\; 0 \le i \le  b_{j+1}^t-1; \label{3.4c}
\end{equation}
\begin{eqnarray}
&&\nonumber\text{the intervals }[\widetilde\al^{t}_{j,i} ,\widetilde
\om^{t}_{j,i}], \;j \ge 0, 0 \le i \le b_{j+1}^t-1 \;(\text{in the
order } (j,i) =(0,0),
 \dots,\\
&&(0,b_1^t-1), (1,0), \dots, (1,b_2^t-1),(2,0),\dots) \text{ form a
partition of $\Bbb Z_+$}; \label{3.4d}
\end{eqnarray}
\begin{eqnarray}
&&\text{each $[\wt \al^t_{j,i} , \wt \om^t_{j,i}]$ is a
union of intervals $[\wt \al^{t-1}_{u,v}, \wt \om^{t-1}_{u,v}]$}\\\nonumber
&&\text{over a finite number of suitable pairs $(u,v)$}\\\nonumber
&&\text{(for $t=1$ this condition is taken to be fulfilled by
convention)}; \label{3.4g}
\end{eqnarray}
\begin{equation}
\wt \cC_{t,j} \subset[\widetilde \al^{t}_{j,0},\widetilde
\om^t_{j,0}],   \text{ for each }j \ge 1; \label{3.4e}
\end{equation}
\begin{equation}
\text{span}(\wt\cC_{t,j})=[\widetilde \al^{t}_{j,0}, \widetilde
\om^t_{j,0}]\text{  when } m(\wt\cC_{t,j})>t, \text{ for each
}j\ge 1. \label{3.44bb}
\end{equation}
As we shall see, suitable analogues of these properties will hold
at all steps in the recursive construction.
\medskip
For $j\ge 0$ and $0 \le i \le b_{j+1}-1$ define the 1-layers
\begin{equation}
\widetilde H_{j,i}^1 = \bigcup_{s\in [\widetilde \al^{1}_{j,i},
\widetilde \om^{1}_{j,i}] } H_{s}^0. \label{3.44aa}
\end{equation}

In particular, from the initial convention, $\widetilde
H^1_{0,0}=H_{0}^0\cup H_{1}^0$.

Notice that the 1-layers $\widetilde H_{j,i}^1, \;j \ge 0, 1\le i
\le b_{j+1}-1$, are contained in
\begin{equation*}
\bigcup_{\om(\cC_{1,j})+1 \le s \le
\al(\cC_{1,j+1})-1}H_s^0
\end{equation*}
and therefore do not contain any bad
0-layers. They will be called {\it
good} 1-layers of {\it type 2}.
For these values of $i,j$ we set
\begin{equation}
\cD_{j,i}^1=\{\widetilde \om^{1}_{j,i}-1,\widetilde
\om^{1}_{j,i}\}\quad \text{ and} \quad
\cD^{1,\cK}_{j,i}=\{\widetilde \om^{1}_{j,i}\}.
\label{3.44aaa}
\end{equation}
If a layer $\widetilde H_{j,i}^1$ does not contain a bad line we
simply write $\cK_{j,i}^1=[\widetilde \alpha^1_{j,i},\widetilde
\omega^1_{j,i}]$. This applies to each $\wt H^1_{j,i}$ with $i \ge 1$.
On the other hand, if $j \ge 1$, each of the
layers $\widetilde H_{j,0}^1$ contains exactly $m(\wt\cC_{1 ,j})$
bad 0-layers. When $m(\wt \cC_{1 ,j})=1$, (i.e. $\ell(\wt \cC_{1
,j})=0$), $\wt H_{j,0}^1$ is still said to be a {\it good} 1-layer
(of {\it type 1}, in this case). We then set
\begin{eqnarray}
\label{3.44abc}
&&\cD_{j,0}^1=\{\widetilde
\om^{1}_{j,0}-1,\widetilde
\om^{1}_{j,0}\}, \;\cD^{1,\cK}_{j,0}=\;\{\alpha(\wt \cC_{1,j})-1\},\\
&&\text{and } \cK_{j,0}^1=[\om^{1}_{j-1,b_{j}-1}+1, \alpha(\wt
\cC_{1,j})-1 ].\nonumber
\end{eqnarray}
If $m(\wt \cC_{1 ,j}) >1$, then $\widetilde H_{j,0}^1$ is called a
{\it bad} 1-layer.

\noindent {\bf Remark.} The layer $\widetilde H_{0,0}^1$ is
exceptional. It has not been classified as good or bad above,
though there is no harm in calling it good.
\smallskip

The family $\{\widetilde H_{j,i}^1, \; j \ge 0, \; 0 \le i \le
b_{j+1}-1\}$ forms a partition of $\widetilde {\Bbb Z}^2_+$ into
horizontal layers. We relabel these layers in increasing order (upwards)
as $\bH^1 = \{ H_{j}^1 \}_{j\ge 1}$. In particular, we see that $H_{1}^1=\widetilde H_{0,0}^1$. When the
layer $H_{j}^1$ is good, the associated $\cD,\,\cD^K$ and $\cK$
defined above, also carry the superscript 1 and subscript $j$.

Now define  $\al_{j}^1$ and $\om_{j}^1$ by the requirement
\begin{equation}
H_{j}^1 = \bigcup _{s\in [\al_{j}^1, \om_{j}^1]}H_{s}^0,
\label{3.4h}
\end{equation}
and write $\cH_{j}^1=[\al_{j}^1, \om_{j}^1]$, and $\cF_{j}^1=\{\alpha_{j}^1\}$ for each $j\ge 1$. When
$H^1_j$ does not contain a bad line we have $\cH_{j}^1=\cK_{j}^1$, as follows from the definition given above.

\smallskip
%INSERT PICTURE WITH RELATIVE LOCATIONS LAYERS AND CLUSTERS
\smallskip

\noindent {\bf Remark.} The intervals $\cH_{j}^1, j\ge 1$, are
simply  the intervals $[\wt \al^1_{j,i}, \wt \om^1_{j,i}]$ of
\eqref{3.4d} simply relabeled in upward order. (From the
construction $\cH_{1}^1=\{0,1\}$.)
\medskip

\noindent {\bf Step k.} Let $k \ge 2$. Recall that $\wt \cC_{t,j},
j=1,2,\dots$ are just the clusters in $\bold C_t$ of mass at least
$t$ labeled in increasing space order. Note that $\{\wt \cC_{t,j}\colon j \ge 1\}$
is only a subset of $\bold C_t$. Assume now that we have carried
out the preceding construction through step $k-1$ in such a way
that \eqref{3.4c}-\eqref{3.44bb} hold for $1 \le t \le k-1$ with
\begin{equation}
\widetilde b_1^t= \Big\lfloor \frac{\al(\wt \cC_{t ,1} )}
{L^t/3}\Big\rfloor,\quad \widetilde b_j^t=\Big\lfloor
\frac{\al(\wt\cC_{t ,j} ) -  \om(\wt \cC_{t
,j-1})}{L^t/3}\Big\rfloor \quad \text{for  } j\ge 2  \label{3.4f}
\end{equation}
(in particular $\widetilde b^t_j\ge 3$ due to our assumption $\chi(\ga) = 0$ and
property \eqref{2.two}), while
\begin{equation}
b_j^t = \begin{cases}\widetilde b_j^t , \quad &\text{if} \quad m
(\wt\cC_{t ,j}) =t,\cr \widetilde b_j^t +1 , \quad &\text{if}
\quad m (\wt\cC_{t ,j}) > t,
\end{cases}
\label{3Ag}
\end{equation}
for  $j \ge 1$.

As in \eqref{3.44aa} and \eqref{3.4h} we can form successively for $2
\le t \le k-1$ the $t$-layers
\begin{equation}
\wt H^t_{j,i} = \bigcup_{s: \cH^{t-1}_s \subset [\wt \al^t_{j,i}, \wt
\om^t_{j,i}]} H^{t-1}_s \label{3Ak}
\end{equation}
and then relabel these layers in increasing upward order as $H^t_j,\; j
\ge 1$. The partition $\bold H^t$ is then $\{H^t_j\}_{j \ge 1}$.
Again following the case $t=1$ we define $\al^t_j$ and $\om^t_j$
by the requirement
\begin{equation}
H^t_j = \bigcup_{s:\cH^{t-1}_s \subset [\al^t_j,\, \om^t_j]}
H^{t-1}_s. \label{3Al}
\end{equation}
Finally we take
\begin{equation}
\cH^t_j = [\al^t_j, \om^t_j]. \label{3Am}
\end{equation}

We now consider a cluster $ \wt \cC_{k,j}$ which lies in  $\bold
C_k$ and has mass at least equal to $k$. If $\l(\wt \cC_{k,j}) \le
k-1$ (this occurs for instance if $m(\wt \cC_{k,j}) =k$, by
\eqref{2.fourc}), then even $\wt \cC_{k,j} \in \bold C_{k-1}$, by
virtue of \eqref{2.fourff} with $s=t = k-1$. Hence, by
\eqref{3.44bb} for $t = k-1$ (recall that we are assuming that this
is valid), span $(\wt \cC_{k,j}) = [\wt \al^{k-1}_{u,0}, \wt
\om^{k-1}_{u,0}]$ for some $u$. In particular, there exists an index
$i_j\ge 1$ such that
\begin{equation*}
\al(\wt \cC_{k,j}) = \al_{i_j}^{k-1} \text{ and }
\om (\wt \cC_{k,j}) = \om_{i_j}^{k-1}.
\end{equation*}
On the other hand, if $\l(\wt \cC_{k,j}) = k$, then
$\wt \cC_{k,j}$ is made up
from constituents $\wt \cC_{k-1, s_1}, \wt \cC_{k-1, s_2},
\dots,$ $\wt \cC_{k-1, s_n}$ with $n \ge 2$,
say. These constituents have mass at least $k$. By applying \eqref{3.44bb}
(with $t$  replaced by $k-1$) to
$\wt \cC_{k-1, s_1}$ and to $\wt\cC_{k-1, s_n}$
we see that for each $j$ there exist indices $i_j\le i'_j$ so that
\begin{equation}
\al(\wt \cC_{k,j})=\al^{k-1}_{i_j} \text{ and }\om(\wt
\cC_{k,j})=\om^{k-1}_{i'_j}.
\label{3Bm}
\end{equation}
(For $k=1$ this is \eqref{3.44bb}). Thus in both cases \eqref{3Bm} holds,
but $i_j = i'_j$ if and only if $\l(\wt\cC_{k,j}) \le k-1$.

\smallskip

\noindent {\bf Remark.} The subscripts $i_j,i'_j$ above are the
labels of the first and last constituents of $\wt \cC_{k,j}$ in
the enumeration of ${\text{\bf H}}^{k-1}$ (not of ${\text{\bf
C}}_{k-1}$!), as recursively constructed. They depend on $k$
but we omit this in the (already heavy) notation.
%(it would be natural to denote them as $i_{k-1,j},i'_{k-1,j}$)

%%\noindent {\bf  Remark.} We shall see at each step that $i_1\ge
%%L/2$ and also $i_{j+1}-i'_j \ge L/2$ for $L \ge 12$. This is so
%%due to the order of magnitude of the good layers of previous step,
%%to the property \equ(2.two) about inter-distances between the
%%clusters, and because we are considering realizations $\ga$ such
%%that $\chi(\ga)=0$. CHECK THE VALUE $L/2$ AGAIN

We make the convention that $\widetilde \al^{k}_{0,0}=0$ and
$\widetilde \om^{k}_{0,0}=\om^{k-1}_3$. We further make the
definitions \eqref{3.4f} and \eqref{3Ag} for $t=k$ as well. We also
set, for $j \ge 1$:
\begin{equation}
\widetilde \om^{k}_{j,0} =
\begin{cases}  \om^{k-1}_{i_j+3}, \quad &
\text{if} \quad m (\wt \cC_{k,j}) =k; \cr \om^{k-1}_{i'_j}, \quad
& \text{if} \quad m (\wt \cC_{k,j})>k.
\end{cases} \label{3.1k}
\end{equation}

To proceed with the definitions of $\wt \al^k_{j,i}$ and $\wt
\om^k_{j,i}$ we need new quantities $s(j,i) = s^k(j,i)$. We take
$s(j,i)= s^k(j,i)$ to be the unique value of $s$ for which
\begin{eqnarray}
\label{3.aaaa}
&&iL^k/3 \in \cH^{k-1}_s  \text { if }j=0,\\
&&\om(\wt \cC_{k,j}) + iL^k/3 \in \cH^{k-1}_s  \text { if }j \ge 1.
\nonumber
\end{eqnarray}
($s(j,i)$ is unique because the intervals in \eqref{3.4d} form a
partition of $\Bbb Z_+$.) Now, letting $\wt b^k_j$ and $b^k_j$ for
$j \ge 1$ be defined by \eqref{3.4f} and \eqref{3Ag} with $t=k$, we
may set, for $j\ge 0$ and $1 \le i \le \widetilde b^k_{j+1} -1$:
\begin{equation}
\widetilde \om^{k}_{j,i} = \om^{k-1}_{s(j,i)}. \label{3.1l}
\end{equation}
Still for  $j \ge 0$ and $1 \le i \le \widetilde b^k_{j+1} -1$ we
set
%%%WHAT ARE THESE INTERVALS FOR ?? SHOULD WE POSTPONE THEIR
%%%INTRODUCTION TILL THEY ARE NEEDED ?? SIMILARLY FOR SOME LATER INTERVALS.
%%% Aug 31,2005: to be used later. for us (e,v) it can stay here for
%%convenience of writing.
\begin{equation}
\cD_{j,i}^k=[\alpha^{k-1}_{s(j,i)-1},\omega^{k-1}_{s_(j,i)}]
\text{ and
}\cD^{k,\cK}_{j,i}=[\alpha^{k-1}_{s(j,i)},\omega^{k-1}_{s(j,i)}].
\label{3.44bbb}
\end{equation}
These sets are listed in the same (lexicographic) order for the
pairs $(j,i)$ as in \eqref{3.4d}; $\cD^{k,\cK}_v$ will be the $v$-th
set in this enumeration of the family $\{\cD^{k,\cK}_{j,i}\}$. This
comment also applies to the family $\{\cD^{1,\cK}_{j,i}\}$
introduced in \eqref{3.44aaa}-\eqref{3.44abc}.

If $j\ge 0$ and $m (\wt \cC_{k,j+1})>k$ we set
\begin{equation}
\widetilde \om^{k}_{j,b^k_{j+1} -1}=\alpha(\wt \cC_{k,j+1})-1.
\label{3.50f}
\end{equation}
We may now define for $j\ge 1$:
\begin{equation}
\widetilde \al^{k}_{j,0} = \begin{cases} \wt \om^{k}_{j-1,b^k_{j}-1}+1,
\quad & \text{if} \quad m (\wt \cC_{k,j}) =k; \cr
\alpha^{k-1}_{i_j}, \quad & \text{if} \quad m (\wt \cC_{k,j})>k,
 \end{cases}
\label{3.2k}
\end{equation}
and, in the case $m (\wt \cC_{k,j}) =k$,
\begin{equation}
\cD_{j,0}^k=[\alpha^{k-1}_{i_j +2},\omega^{k-1}_{i_j+3}],\text{
and } \cD^{k,\cK}_{j,0}=[\alpha^{k-1}_{i_j
-1},\omega^{k-1}_{i_j-1}].
\label{3.44ccc}
\end{equation}
Finally, if $j\ge 0$, we set
\begin{equation}
\widetilde \al^{k}_{j,i} =  \widetilde \om^{k}_{j,i-1}+1 \quad
\text {for} \quad 1 \le i \le b^k_{j+1}-1. \label{333k}
\end{equation}
We then set, if $j \ge 0,\; 0 \le i \le b^k_{j+1}-1$:
\begin{equation*}
\widetilde H^k_{j,i}=\bigcup_{s\colon \cH^{k-1}_s\subset
[\widetilde \al^{k}_{j,i}, \widetilde\om^{k}_{j,i}]}H^{k-1}_s.
\end{equation*}
The $k$-layers $\widetilde H^k_{j,i}$ with $j\ge 0$ and $1 \le i \le
b^k_{j+1}-1$ are said to be {\it good} of {\it type 2}. The layers
$\widetilde H^k_{j,0}$, with a $j\ge 1$ for which $m(\wt
\cC_{k,j})= k$ are also called good (in this case of {\it type}
1), and  we set
\begin{equation}
\cK^k_{j,0}=[\widetilde
\al^{k}_{j,0},\alpha(\wt\cC_{k,j})-1],\text{ and }\cK^k_{j,i}=
[\widetilde \al^{k}_{j,i},\widetilde \omega^{k}_{j,i}].
\label{329}
\end{equation}
The $k$-layers $\widetilde H^k_{j,0}$, with $j \ge 1$ such that
$m(\wt \cC_{k,j})> k$ are called {\it bad}. The layer
$\widetilde H^k_{0,0}$ is again exceptional.
The {\it support} of the layer $\widetilde H^t_{j,i}$ is the interval $[\wt
\al^t_{j,i},\wt  \om^t_{j,i}]$, sometimes written as $\text{supp}(\widetilde H^t_{j,i})$.

We shall prove by induction that \eqref{3.4c}-\eqref{3.44bb} hold
for all $t \ge 1$. For this we make the convention that for $r \ge
0$, $\al^0_r = \om^0_r = r$ and $\cH^0_r = [r,r]= \{r\}$. We  can
then define $s^1(j,i)$ in the same way as $s^k(j,i)$ in
\eqref{3.aaaa} with $k-1$ replaced by 0. We shall also need the
following quantities:
\begin{eqnarray*}
A_t &=& \max\{(\om^t_j - \al^t_j): j \ge 2, H^t_j
\text{ is a good $t$-layer}\}\\
&=& \max\{(\wt\om^t_{j,i} -\wt\al^t_{j,i}):j \ge 0, 0 \le i \le
b^t_{j+1}-1, (j,i) \ne (0,0), \wt H^t_{j,i} \text{ is a good
$t$-layer}\}
\end{eqnarray*}
and
\begin{eqnarray*}
a_t &=& \min\{(\om^t_j - \al^t_j): j \ge 2, H^t_j
\text{ is a good $t$-layer}\}\\\nonumber
&=& \min\{(\wt \om^t_{j,i} -\wt\al^t_{j,i}): j \ge 0, 0 \le i \le
b^t_{j+1}-1, (j,i) \ne (0,0), \wt H^t_{j,i} \text{ is a good
$t$-layer}\}.
\end{eqnarray*}
%(in these definitions a good layer may be of type 1 or of type 2).
% the $t$ in $A^t$ or $a^t$ is just a superscript and does not denote a power).

\begin{lemma} \label{lemma3.1} Let $L \ge 108$ and let $\ga $ be an
environment with $\chi(\ga) = 0$. Then the following properties
hold for all $t\ge 1$:
\begin{eqnarray}
&&\text{The support of a bad } \text{$(t-1)$-layer
cannot intersect the interval}\nonumber\\
&&[1, \al(\wt \cC_{t,1})-1] \text{ or any interval $[\om(\wt
\cC_{t,j})+1, \al(\wt \cC_{t,j+1})-1],\; j\ge 1$}; \label{3.40z}
\end{eqnarray}
\begin{equation}
\text{For all $j \ge 0, 1 \le i \le \wt b^t_{j+1}-1$},
H^{t-1}_{s^{t}(j,i)}\text{ is a good $(t-1)$-layer};\label{3.50a}
\end{equation}
\begin{eqnarray}
\text{If } j \ge 1 \text{ and }m(\wt \cC_{t,j}) =t,\text{ then }H^{t-1}_{i_j+\l}  \text{ is a good $(t-1)$-layer}\nonumber \\
\text{for }1 \le \l \le L/3,\text{ with $i_j$ as defined after \eqref{3Am}},  \text{ and }\wt \al^t_{j,1} \le \wt\om^t_{j,1};
\label{3.50p}
\end{eqnarray}
%%%(we take \equ(3.40z)-\equ(3.50p) to be the empty statements
%%%when $t=1$, so that they hold by convention for $t=1$);
\begin{equation}
L^t/4 \le a_t\le A_t \le 2L^t. \label{3.50}
\end{equation}
Moreover, the properties \eqref{3.4c}-\eqref{3.44bb} hold for all $t
\ge 1$.
\end{lemma}
\begin{proof} Note that it follows directly from the definitions
that if $\wt H^t_{j,i}$ is a bad $t$-layer for some $t \ge 1$,
then it must be the case that $i=0, m(\wt \cC_{t,j}) > t$ and
$\text{span}(\wt \cC_{t,j}) = [\wt \al^t_{j,0}, \wt \om^t_{j,0}]$.

For $t=1$, \eqref{3.40z} is obvious, since each bad line belongs to
some cluster $\wt \cC_{1,j}$.

Now assume that \eqref{3.40z} is false
for some $t \ge 2$ and that $\wt H^{t-1}_{p,0}$ is a bad
$(t-1)$-layer whose support $[\wt \al^{t-1}_{p,0},
\wt \om^{t-1}_{p,0}]$ intersects $[\om(\wt \cC_{t,j})+1,\al(\wt
\cC_{t,j+1})]$ in a point $x$. (For $j=0$ we interpret $\om(\wt
\cC_{t,0})$ as 0, but for simplicity we restrict ourselves to $j
\ge 1$ in the proof of \eqref{3.40z}.) Then, as just observed, $[\wt
\al^{t-1}_{p,0}, \wt \om^{t-1}_{p,0}] = $ span$(\wt \cC_{t-1,p})$
for some cluster $\wt \cC_{t-1,p}$ of mass at least $t$.
When $\bold C_t$ is formed from $\bold C_{t-1}$, then $\wt
\cC_{t-1,p}$ is either in the same maximal $t$-run as the
constituents of $\wt \cC_{t,j}$, or in the same maximal $t$-run as
the constituents of $\wt \cC_{t,j+1}$, or $\wt \cC_{t-1,p}$ is
disjoint from both these maximal $t$-runs. In fact only the last
situation is possible, because the point $x$ of $\wt \cC_{t-1,p}$
lies outside span$(\wt\cC_{t,j}) \cup$ span$(\wt \cC_{t,j+1})$
and $\wt \cC_{t-1,p} \in \bold C_{t-1,t}$.
It then follows that span$(\wt \cC_{t,j})$, span$(\wt
\cC_{t,{j+1}})$ and span$(\wt \cC_{t-1,p})$ are (pairwise)
disjoint (see \eqref{2.y}). But this too is impossible,
because
 $\wt \cC_{t,j}$ and $\wt \cC_{t,j+1}$ are successive clusters of $\bold
 C_t$ of mass $\ge t$. Thus $\wt \cC_{t-1,p}$ cannot lie between
$\wt \cC_{t,j}$ and $\wt \cC_{t,j+1}$. Thus, \eqref{3.40z} holds.

Next we prove \eqref{3.50a}. Assume, to derive a contradiction, that
for some $t \ge 1, j \ge 0, 1\le i \le \wt b^t_{j+1}-1,
H^{t-1}_{s^t(j,i)}$ is a bad $(t-1)$-layer. For the purpose of
this proof use the abbreviation $r = s^t(j,i)$. By definition of
$s^t(j,i)$ we then have
\begin{equation}
\al^{t-1}_r \le \om(\wt \cC_{t,j}) + i\frac{L^t}3 \le \om^{t-1}_r
\label{3.50b}
\end{equation}
if $j \ge 1$, and $\al^{t-1}_r \le iL^t/3 \le \om^{t-1}_r$ if
$j=0$. For the sake of argument we assume that $j \ge 1$; the case
$j=0$ is similar. But for $1 \le i \le \wt b^t_{j+1}-1$, it holds
(by \eqref{3.4f})
\begin{equation*}
\om(\wt \cC_{t,j})+1 \le \om(\wt \cC_{t,j}) + i\frac{L^t}3 \le
\al(\wt \cC_{t,j+1}) - \frac {L^t}3.
\end{equation*}
In other words,
\begin{equation*}
\om(\wt \cC_{t,j}) + i\frac{L^t}3  \in [\om(\wt \cC_{t,j})+1,
\al(\wt \cC_{t,j+1})-1] \cap [\al^{t-1}_r,\om^{t-1}_r].
\end{equation*}
By  virtue of \eqref{3.40z} this contradicts our assumption that
$[\al^{t-1}_r,\om^{t-1}_r]$ is a bad $(t-1)$-layer. Thus
\eqref{3.50a} must hold.

The proof of \eqref{3.50p} has several similarities with that of
\eqref{3.50a}. For this property as well as for
\eqref{3.4c}-\eqref{3.44bb} we use a proof by induction. First, the
first part of \eqref{3.50p} for $t=1$ is easy. Indeed, if
$m(\cC_{1,j}) = 1$, then $(\cC_{1,j})$ is a singleton and $i_j =
i'_j$ is such that $\al(\wt \cC_{1,j}) =\om(\wt \cC_{1,j}) =
\al^0_{i_j} = \om^0_{i_j} = i_j$. The layers $H^0_{i_j+\l}, 1 \le
\l \le L/3$, therefore consist of the lines $H_p$, with $p \in
[\om(\wt \cC_{1,j}) + 1, \om(\wt \cC_{1,j}) + L/3] \subset
[\om(\wt \cC_{1,j})+ 1, \al(\wt \cC_{1,j+1})-1]$, because $\al(\wt
\cC_{1,j+1}) - \om(\wt \cC_{1,j}) \ge L$, by virtue of
\eqref{2.two}. It then follows from \eqref{3.40z} that all the layers
$H^0_{i_j+\l}, 1 \le \l \le L/3$ must be good 0-layers, so that
the first part of \eqref{3.50p} holds for $t=1$.

As for the last part of \eqref{3.50p}, if $m(\wt \cC_{1,j}) = 1$ and
$t=1$, this is equivalent (by \eqref{3.44},\eqref{3.1} and \eqref{3.4a})
to
\begin{equation*}
\wt \om^1_{j,0}+1 = \om(\wt \cC_{1,j})+4 \le \om(\wt
\cC_{1,j})+\frac L3.
\end{equation*}
Clearly this last inequality holds when $L \ge 12$, so that
\eqref{3.50p} holds for $t=1$. One can also check by hand that the
properties \eqref{3.50} and \eqref{3.4c}-\eqref{3.44bb} hold for $t=1$.
E.g., for \eqref{3.50p} use \eqref{3.44} and \eqref{3.4a}.
We therefore only have to verify the induction step for
\eqref{3.50p},\eqref{3.50} and \eqref{3.4c}-\eqref{3.44bb}.

Assume then that \eqref{3.50p}, \eqref{3.50} and \eqref{3.4d} have
already been proven for $t= k-1 \ge 1$. It then follows from
\eqref{3.4d} that
\begin{equation}
\al^{k-1}_{u+1} = \om^{k-1}_u + 1, u \ge 0. \label{3.50q}
\end{equation}
We can therefore write
\begin{equation*}
\om^{k-1}_{i_j+\l} = \om^{k-1}_{i_j} + \sum_{v=1}^\l
[\om^{k-1}_{i_j+v}-\al^{k-1}_{i_j+v}+1].
\end{equation*}
Now assume, to derive a contradiction, that $m(\wt \cC_{k,j}) = k$
and that $H^{k-1}_{i_j+\l}$ is a bad $(k-1)$-layer for some $1 \le
\l \le L/3$. Pick $\l \in [1,L/3]$ to be minimal with this
property. Then $H^{k-1}_{i_j+v}$ is a good $(k-1)$-layer for $1
\le v < \l$, so that
\begin{equation*}
\om^{k-1}_{i_j+v} - \al^{k-1}_{i_j+v} \le A_{k-1} \le 2L^{k-1}, 1
\le v < \l,
\end{equation*}
(by \eqref{3.50} for $t = k-1$). In particular,
\begin{equation}
\al^{k-1}_{i_j+\l} - \om^{k-1}_{i_j} = \om^{k-1}_{i_j+\l-1}+1 -
\om^{k-1}_{i_j} \le (\l-1)[2L^{k-1}+1] + 1 \le 3\l L^{k-1} -1.
\label{3.50r}
\end{equation}
But $\al^{k-1}_u$ is increasing in $u$ by \eqref{3.4d} for $t = k-1$.
Therefore,
\begin{eqnarray}
&&\om(\wt \cC_{k,j}) = \om^{k-1}_{i'_j} = \om^{k-1}_{i_j}
\text{ (because we took }m(\wt \cC_{k,j}) = k\text{)}\\ \nonumber
&& = \al^{k-1}_{i_j+1}-1 \text{(by \eqref{3.50q})} <
\al^{k-1}_{i_j+\l} \le \om(\wt \cC_{k,j}) +3\l L^{k-1} -1< \al(\wt
\cC_{k,j+1}),
\label{3.50qq}
\end{eqnarray}
where the last inequality follows from $\al(\wt \cC_{k,j+1}) -
\om(\wt \cC_{k,j})= d(\wt \cC_{k,j}, \wt \cC_{k,j+1}) \ge L^k$,
which in turn follows from \eqref{2.two}. Thus
\begin{equation*}
\al^{k-1}_{i_j+\l}  \in [\om(\wt \cC_{k,j})+1, \al(\wt
\cC_{k,j+1})-1] \cap [\al^{k-1}_{i_j+\l},\om^{k-1}_{i_j+\l}]
\end{equation*}
and we have again arrived at a contradiction  with \eqref{3.40z}.
This proves that $H^{k-1}_{i_j+\l}$ is a good $(k-1)$-layer for $1
\le \l \le L/3$, which is the first claim in \eqref{3.50p} with
$t=k$.

Now, to prove the last part of \eqref{3.50p} with $t=k$, note that
under the condition $m(\wt \cC_{k,j}) =k$ it is equivalent to
\begin{equation}
\om^k_{j,0}+1 = \om^{k-1}_{i_j+3}+1 \le \om^{k-1}_{s^k(j,1)},
\label{3.50s}
\end{equation}
by virtue of \eqref{333k},\eqref{3.1k} and \eqref{3.1l}. But, just as in
\eqref{3.50r},
\begin{equation*}
\om^{k-1}_{i_j+3} \le \om^{k-1}_{i_j} + 3(2L^{k-1}+1) = \om(\wt
\cC_{k,j}) + 6L^{k-1}+3.
\end{equation*}
On the other hand, by the definition of $s(j,1)$ it holds
\begin{equation*}
\om^{k-1}_{s^k(j,1)} \ge\om(\wt \cC_{k,j}) + \frac {L^k}3.
\end{equation*}
Thus, \eqref{3.50s} does indeed hold (for $L \ge 108$) and
\eqref{3.50p} for $t =k$ follows.

Next we turn to \eqref{3.50} for $t=k$. We only give some
representative parts of the argument, and leave the remaining parts
to the reader. In particular, for the last inequality we only
estimate
\begin{eqnarray*}
\max\{(\wt\om^k_{j,i} -\wt\al^k_{j,i}):j \ge 1, m(\wt \cC_{k,j}) >
k,0 \le i \le b^k_{j+1}-1, (j,i) \ne (0,0), \wt \cH^k_{j,i} \text{
is a good $k$-layer}\}.
\end{eqnarray*}

Note that the case with $j\ge 1, i=0$ does not have to be
considered, because $\wt H^k_{j,0}$ is bad when $j \ge 1, m(\wt
\cC_{k,j}) > k$. We now separately consider the cases
\begin{itemize}
\item[(\it{i})] $ j \ge 1, 2 \le i \le \wt b^k_{j+1}-1$;
\item[(\it{ii})] $j \ge 1, i = 1$;
\item[(\it{iii})] $j \ge 1, \wt b^k_{j+1}\le  i \le b^k_{j+1}-1$.
\end{itemize}
For $(j,i)$ in case $(i)$ we have
\begin{eqnarray*}
&&\wt \om^k_{j,i} - \wt \al^k_{j,i} = \om^{k-1}_{s^k(j,i)} -
\wt
\om^k_{j,i-1} -1 = \om^{k-1}_{s^k(j,i)}-\om^{k-1}_{s^k(j,i-1)} - 1\\
&&\le \om^{k-1}_{s^k(j,i)} - [\om(\wt \cC_{k,j}) + i \frac{L^k}3] +
[\om(\wt \cC_{k,j}) + i \frac{L^k}3]
-[\om(\wt \cC_{k,j}) + (i-1) \frac{L^k}3]\\
&&\text{(compare with the right hand inequality
in \eqref{3.50b})}\\
&&\le A_{k-1} + \frac {L^k}3\\
&& \text{(note that this step uses \eqref{3.50a} for $t=k$ and the
induction hypothesis)}.
\end{eqnarray*}
For $(j,i)$ in case $(ii)$ we have similarly
\begin{eqnarray*}
&&\wt \om^k_{j,i} - \wt \al^k_{j,i} = \om^{k-1}_{s^k(j,1)} -
\wt
\om^k_{j,0} -1\\
&&\le  A_{k-1}+ \om(\wt \cC_{k,j}) + \frac {L^k}3 -\om^{k-1}_{i'_j} - 1\\
&&=  A_{k-1}+ \om(\wt \cC_{k,j}) + \frac {L^k}3 -\om(\wt
\cC_{k,j}) - 1.
\end{eqnarray*}
Finally, case $(iii)$ is non-empty if and only if $m(\wt
\cC_{k,j+1}) > k$ so that $b_{j+1}^k = \wt b_{j+1}^k+1$. We then
have for $i = b^k_{j+1}-1$,
\begin{eqnarray*}
&&\wt \om^k_{j,i} - \wt \al^k_{j,i} = \al(\wt \cC_{k,j+1})-1 - \wt
\om^k_{j,b^k_{j+1}-2} -1 \text{ (by \eqref{3.50f} and
\eqref{333k})}\\ &&\le \al(\wt \cC_{k,j+1}) -
\om^{k-1}_{s(j,b^k_{j+1}-2)} \text{ (by
\eqref{3.1l}}\le \al(\wt \cC_{k,j+1}) - \om(\wt \cC_{k,j}) - [b^k_{j+1}-2]\frac{L^k}3\\
&&\le  L^k \text{ (by \eqref{3.4f})}.
\end{eqnarray*}

Thus in all cases checked so far we found that
$$
\wt \om^k_{j,i} - \wt \al^k_{j,i} \le A_{k-1} + L^k,
$$
and it turns out that this inequality holds in all cases. Together
with the induction hypothesis this shows that for $L \ge 12$
$$
A_k \le A_{k-1} + L^k \le 2L^{k-1} + L^k \le 2L^k.
$$
This is the desired first inequality in \eqref{3.50} with $t=k$.

The last case we check is the case of the second inequality of
\eqref{3.50} when $j = 0,\; i=1$, and still $m(\wt \cC_{k,j}) > k$.
In this case we find (using the definition of $s^k(0,1)$):
\begin{eqnarray}
\wt \om^k_{0,1} - \wt \al^k_{0,1} = \om^{k-1}_{s^k(0,1)} -
\wt \om^k_{0,0}-1
= \om^{k-1}_{s^k(0,1)} - \om^{k-1}_3 -1\ge \frac {L^k}3 - \om^{k-1}_3 -1. \label{3.50d}
\end{eqnarray}

To use this estimate we need an upper bound for $\om^{k-1}_3$.
Now, by \eqref{3.4d} for $t=k-1$, $\om^{k-1}_\l = \wt
\om^{k-1}_{0,\l-1}$, for $1 \le \l \le \wt b^{k-1}_1$ (and $\wt
b^{k-1}_1 \ge 3$). Therefore, $\al^{k-1}_\l = \wt
\al^{k-1}_{0,\l-1} = \wt \om^{k-1}_{0,\l-2}+1$ for $\l = 2$ or 3
(see \eqref{333k}). Thus,
\begin{eqnarray*}
\om^{k-1}_3 &&=\; \wt \om^{k-1}_{0,2} = [ \wt
\om^{k-1}_{0,2} -
 \wt \al^{k-1}_{0,2}]+  \wt \om^{k-1}_{0,1} +1 \\\nonumber
&&\le A_{k-1} + \wt \om^{k-1}_{0,1}+1 \text{ (use \eqref{3.50p})} \le 2A_{k-1} + \wt \om^{k-1}_{0,0} +2\\ \nonumber
&&= 2A_{k-1} + \om^{k-2}_3 +2 \text{ (by definition)}.
\end{eqnarray*}
Iteration of this inequality shows that
\begin{eqnarray}
\om^{k-1}_3
% &\le(2A_{k-1}+2) + (2A_{k-2}+2)+\cdots + (2A_1+2) + \om^0_3\nonumber\\
\le \sum_{t= 1}^{k-1} (2L^t+2) + 2 \le 8L^{k-1}, \label{3.50m}
\end{eqnarray}
by the first part of \eqref{3.50} and the choice of $\om^0_3$. We
note in passing that this last estimate also gives
\begin{equation}
\max\{A_k, (\om^k_{0,0} - \al^k_{0,0})\} \le 2 L^k,
\label{3.00}
\end{equation}
by our choice of $\al^t_{0,0}, \om^t_{0,0}$. Substitution of the estimate
\eqref{3.50m} into \eqref{3.50d} shows that
\begin{equation*}
\wt \om^k_{j,i}- \wt \al^k_{j,i} \ge \frac {L^k}3 - 8L^{k-1} -1
\ge \frac {L^k}4,
\end{equation*}
provided $L \ge 108$, $j=0,
i=1$, and $m(\wt \cC_{k,j})
> k$. In fact this inequality remains valid for all good $k$-layers
$\wt H^k_{j,i}, \;0 \le i \le b^k_{j+1}-1, (j,i) \ne (0,0)$. Thus
\eqref{3.50} holds also for $t=k$.

Next we must prove \eqref{3.4c} for $t=k$. For $(j,i) = (0,0)$ this
is immediate from our choice of $\wt \al^k_{0,0}$ and $\wt
\om^k_{0,0}$. For $(j,i)$ = (0,1), \eqref{3.4c} requires that
\begin{equation}
\wt \al^k_{0,1} = \wt \om^k_{0,0}+1= \om^{k-1}_3 +1 \le \wt
\om^k_{0,1} \label{3.50n}
\end{equation}
(see \eqref{333k}). But we already saw in \eqref{3.50m} that
$\om^{k-1}_3 \le 8L^{k-1}$, while $\wt \om^k_{0,1} =
\om^{k-1}_{s^k(0,1)}$ (see \eqref{3.1l}) $\ge L^k/3 $ (by definition
of $s^k(0,1)$). Thus \eqref{3.50n} holds and $\wt \al^k_{0,1} \le
\wt \om^k_{0,1}$ when $L \ge 108$.

We also proved \eqref{3.4c} when $j \ge 1, i=1$ and $m(\wt
\cC_{t,j}) = t$, in \eqref{3.50p}.

The remaining cases of \eqref{3.4c} are routine. Also
\eqref{3.4d}-\eqref{3.44bb} do not require any new ideas, but only tedious
definition pushing. We leave the verification to
the reader.
\end{proof}

\medskip

\noindent {\bf Remark.} Property \eqref{3.40z} and the arguments for
its proof also show that a $t$-layer $\wt H^t_{j,i}$ with $j \ge 1$
and which is good of type 2 does not contain any bad $(t-1)$-layer.
On the other hand, if for some $j \ge 1$, $m(\wt \cC_{t,j})=t$, then
the $t$-layer $\wt H^t_{j,0}$ is good of type 1, and contains
exactly one bad $(t-1)$-layer.

Indeed, the support of good $t$-layers of type 2 are of the form
$[\wt \al^t_{j,i},\wt \om^t_{j,i}]$ with $1 \le i \le b^t_{j+1}-1$
and one can check that these are contained in the interval $[\om(\wt
\cC_{t,j}+1, \al(\wt \cC_{t,j+1})-1]$ which cannot intersect any
cluster of mass at least $t$ by \eqref{3.40z}.

If, however, $[\wt \al^t_{j,0},\wt\om^t_{j,0}]$ corresponds to a
good $t$-layer of type 1, i.e., if $m(\wt\cC_{t,j})=t$, then the
formulae \eqref{3.2k} and \eqref{3.1k} (see also \eqref{3.4e})
can be used to show that $[\wt \al^t_{j,0},\wt \om^t_{j,0}]
\supset  [\al(\wt \cC_{t,j}), \om(\wt \cC_{t,j})]$. By
\eqref{3Bm} and its proof we then have $[\al(\wt \cC_{t,j}),
\om(\wt \cC_{t,j})] = [\wt\al^{t-1}_{u,0}, \wt \om^{t-1}_{u,0}]$
for some $u \ge 1$. Moreover, the layer $\wt H^{t-1}_{u,0}$ is a
bad $(t-1)$-layer, because $m(\wt \cC_{t,j}) > t-1$. In this case
$[\wt \al^{t-1}_{u,0}, \wt\om^{t-1}_{u,0}]$ necessarily is the only
bad $(t-1)$-layer in $[\wt \al^t_{j,0},\wt \om^t_{j,0}]$. To see this,
note that if $[\wt \al^t_{j,0},\wt \om^t_{j,0}]$ would contain another
bad $(t-1)$-layer, then it would have to contain a cluster $\wt \cC_{t,j'}$
from $\bold C_t$ of mass $t$, but $j' \ne j$. However, that cluster would
be contained in $[\wt \al^t_{j',0},\wt \om^t_{j',0}]$, and this interval
is disjoint from $[\wt\al^t_{j,0},\wt \om^t_{j,0}]$ (by \eqref{3.4d}),
so that $\wt \cC_{t,j'}$ cannot lie in $[\wt \al^t_{j,0},\wt \om^t_{j,0}]$
after all.

Finally we note that $\wt b_j^k \ge 3$ and \eqref{3.4d} show that
any bad $k$-layer $\wt H^k_{j,0}$ is followed by the two $k$-layers
$\wt H^k_{j,i},\; i=1,2$, which are good of type 2 (as we had
claimed already in the second paragraph of this section).
\bigskip

We have now proven that the family $\{\widetilde H^k_{j,i},\; j \ge
0, 0 \le i \le b^k_{j+1}-1\}$ again gives a partition of $\widetilde
{\Bbb Z}^2_+$ and we relabel it as $\{H^k_j,\; j \ge 1\}$, in
increasing order. (In this notation $\widetilde H^k_{0,0}=H^k_1$.)
As before, when $H^k_j$ is a good layer, the associated $\cD$,
$\cD^K$, and $\cK$ are written as $\cD^k_j$ etc., i.e., they carry
the superscript $k$ and subscript $j$. Now define $\al^k_j$ and
$\om^k_j$ as in \eqref{3Al} with $t=k$, and set $\cH^k_j= [\al^k_j,
\om^k_j]$, as in \eqref{3Am}. We set
\begin{equation*}
s_j=\min\{s\colon H^{k-1}_s \subset H^k_j\} \text{ and }
\cF^k_j=\cH^{k-1}_{s_j}.
\end{equation*}
The intervals $\cH^k_j, j \ge 1$, give a partition of $\Bbb Z_+$.
(From the construction $\cH^k_1=\{0,\dots,\om^{k-1}_3\}$.)

\medskip

\noindent {\bf Remark.}
Some of the preceding arguments can be used to
give a lower bound for $i_{j+1}-i'_j$ as follows. By \eqref{2.two}
\begin{equation*}
\al^{k-1}_{i_{j+1}} - \om^{k-1}_{i'_j} = \al(\wt \cC_{k,j+1})
-\om(\wt \cC_{k,j}) \ge L^k \text{ for } j \ge 1.
\end{equation*}
Moreover, by \eqref{3.40z} and \eqref{3.4d} with $t = k-1$, the
interval $[\om(\wt \cC_{k,j})+1,\al(\wt \cC_{k,j+1})-1]$ is a
disjoint union of the supports of a some good layers. More
precisely,
\begin{equation*}
L^k \le \al(\wt \cC_{k,j+1})-\om(\wt \cC_{k,j})=
\al^{k-1}_{i_{j+1}} - \om^{k-1}_{i'_j}   = 1 + \sum [\al^{k-1}_u
- \om^{k-1}_u + 1],
\end{equation*}
where the sum over $u$ runs over all $u$ for which $\cH^{k-1}_u
\subset [ \om(\wt \cC_{k,j})+1,\al(\wt \cC_{k,j+1})-1]$.
There are $i_{j+1}-i'_j$ such summands and each one of them is at most
$2L^{k-1}+1$, by virtue of \eqref{3.50}. Thus
\begin{equation}
i_{j+1} - i'_j \ge \frac{L^k-1}{2L^{k-1}+1} \ge \frac L6 \text{ for
each } j \ge 0.
\label{3Cm}
\end{equation}
($i'_0$ is to be taken as 0 here; $\chi(\ga) = 0$ should be used
instead of \eqref{2.two} to derive this estimate.)
\bigskip

\noindent{\bf Renormalized sites.}
\medskip

To define renormalized sites we first choose constants $c$ and $L
\ge 108$ such that:
%\footnote{The 3/14 here might have to be reduced.
%See proof of Theorem \ref{indstep}.}:

\begin{equation}
c < \frac 3{14}s(p_G),\; c^{-1} \in \mathbb{N} \text{ and } \frac 12 cL \in \mathbb{N},
\label{4.00}
\end{equation}
where $s(p_G)$ is the asymptotic right-edge speed of homogeneous
oriented site percolation on $\wt {\mathbb{Z}}^2_+$ with the
probability of a site being open equal to $p_G$. (see \cite{Durrett}
for the definition of right edge speed). We then define the
renormalized $k$-sites  $S^k_{u,v}$ as follows:

\noindent {\bf Step 0.} $S^0_{u,v}=(u,v)$, for $(u,v) \in \wt{
\mathbb{Z}}^2_+$;

\noindent {\bf Step k.} For $k \ge 1, (u,v) \in \wt {\Bbb
Z}^2_+\setminus (0,0)$
\begin{equation}
S^{k}_{u,v}=\left(\big(\frac {u-1}{2} (cL)^k ,\frac {u+1}{2} (cL)^k \big]\times
\cH^k_v\right)\; \bigcap \widetilde {\Bbb Z}^2_+; \label{sitek}
\end{equation}

\noindent {\bf Remark.} It is clear that for fixed $k \ge 1$ the
$S^k_{u,v}$ with $(u,v)$ running over $\wt{\Bbb Z}^2_+ \setminus
(0,0)$ form a partition of $\wt{\Bbb Z}^2_+$ (see \eqref{3.4d}). We
further define
\begin{equation*}
{\mathring{S}}^k_{u,v} = \bigcup  S^{k-1}_{x,y}
\end{equation*}
with the union running over all $(x,y)$ such that $S^{k-1}_{x,y}
\subset S^k_{u,v}$.  For $k \ge 2$ and fixed $(u,v)$ it is generally
not the case that ${\mathring{S}}^k_{u,v} = S^k_{u,v}$. Thus $S^k_{u,v}$
is not quite a (disjoint) union of $(k-1)$-sites. However,
\begin{equation}
{\mathring{S}}^k_{u,v} \subset S^k_{u,v} \subset \bigcup S^{k-1}_{x,y},
\label{part}
\end{equation}
where the union in the right hand side is over all $(x,y)$ for which
$S^{k-1}_{x,y} \cap S^k_{u,v} \ne \emptyset$
Thus
the points of $S^k_{u,v}$ which are not
contained in ${\mathring{S}}^k_{u,v}$ lie within distance
$(cL)^{k-1}$ of the vertical boundary of
$S^k_{u,v}$, because the projection of any $(k-1)$-site on the
horizontal axis has diameter at most $(cL)^{k-1}$.

%%%I TRIED TO FIGURE OUT YOUR REMARK HERE ON H^k_1 BEING ``THE UNION
%%%OF TWO LAYERS... FROM THE ORIGIN.
%%%PLEASE CHECK WHAT I WROTE IN THE NEXT FEW LINES
Note that by the definitions \eqref{3Ak}-\eqref{3Am} and $\wt \al
^k_{0,0} = 0, \wt \om^k _{0,0} = \om^{k-1}_3$ one has for $k \ge 2$,

\begin{eqnarray*}
H^k_1 &=& \;\wt H^k_{0,0} = \bigcup_{s:\cH^{k-1}_s \subset [\wt
\al^k_{0,0}, \wt \om^k_{0,0}]} H^{k-1}_s=H^{k-1}_1\cup H^{k-1}_2 \cup H^{k-1}_3.
\end{eqnarray*}
Thus, by iterating this relation we see that for any $k\ge 1$, the layer
$H^k_1$ is the union of $H^t_2 \cup H^t_3$ for $1 \le t \le k-1$
and  $H^1_1 = H_0 \cup H_1$.
%%%WHAT DOES THE NEXT SENTENCE MEAN ??
%%%Aug 31,2005 THE NEXT TWO SENTENCES COULD BE DELETED.
%On $H^k_1$ we shall construct the $(k-1)$-seed defined
%below, from which the system percolates into $S^k_{0,2}$. That two
%such layers are enough for our purposes depends on taking $p_G$
%close to one.
%%%WHAT IS THE NEXT SENTENCE FOR ?? DO WE WANT TO KEEP IT ???
%From \eqref{3.50} and the above decomposition of $H^k_1$ into layers
%of smaller scale, we see that $H^k_1$ contains at most $c_4L^{k-1}$
%$0$-layers, where $c_4$ is a suitable positive constant, independent
%of $L,k$, provided $L \ge 108$. ($L \ge 108$???)

%%%%WIDTH AND HEIGHT HAVE NOT BEEN DEFINED AND I DON'T REALLY
%%%%UNDERSTAND WHAT YOU ARE TRYING TO SAY IN THESE TWO PROPERTIES.
%%%%PERHAPS IT IS SIMPLER TO JUST GIVE A STATEMENT ABOUT THE WIDTH AND
%%%%HEIGHT IN LATTICE UNITS. PLEASE CLARIFY THIS.
%%%%ALSO THE FORMULA FOR GENERAL k DOES NOT SEEM TO REDUCE TO THE
%%%%FORMULA FOR k=1, WHEN k=1.
%%%%IN ANY CASE THE PROOF SHOULD JUST BE A REFERENCE TO (3.27)

A renormalized site $S^k_{u,v}$ with $v \ge 2$ is called {\it good (of type
$\l$)} when the $k$-layer $H^k_v$ is good (of type $\l$).
We state two elementary properties of the {\it good} renormalized
sites $S^{k}_{u,v}$ with $v \ge 2$:

i) the number of horizontal $(k-1)$-layers intersecting
$S^{k}_{u,v}$ does not depend on $u$. Any given good $k$-layer $H^k$ intersects at most $8L$
good $(k-1)$-layers (by \eqref{3.50} and the remark a few lines after \eqref{3.50n})
and at most one bad $(k-1)$-layer (by \eqref{3.44bb} and \eqref{3.40z}).
The support
of any bad $(k-1)$-layer which could intersect $H^k$ has diameter
at most $3L^{k-1}$ (since $H^k$ being good, its support cannot contain a cluster
of mass greater than $k$, and \eqref{2.6} holds).
By \eqref{3.50} and the remark following \eqref{3.50n} again,
we see that the number of good $(k-1)$-layers intersecting $H^k$ is at
least $(L^k/4-3L^{k-1})/(2L^{k-1}) \ge L/9$ (provided $L \ge 108$).

ii) the intersection of $S^{k}_{u,v}$ with a $(k-1)$-layer, if
not empty, contains exactly $cL$ $(k-1)$-sites.
\medskip

\noindent {\bf Structure of good sites.}

As stated above, a good $k$-site $S^k_{u,v}$ is called good
of {\it type 1} or {\it type
2} according as the  corresponding $k$-layer $H^k_v$ is good of type 1
or 2 ($v\ge 2$).
It follows from the remark after Lemma \ref{lemma3.1} that
good $k$-sites of type 1 are those good $k$-sites that
contain a layer of bad $(k-1)$-sites. On the other hand,
good $k$-sites of type 2
contain only good $(k-1)$-sites.  If $S^k_{u,v}$ is a good $k$-site
we let
\begin{eqnarray*}
D_l(S^k_{u,v})= \left[\left(\frac {u-1}{2}
+\frac{1}{12}\right)(cL)^k,
\left(\frac {u}{2} -\frac{1}{3}\right)(cL)^k
\right] \times \cD^k_v, \\
D_r(S^k_{u,v})= \left[\left(\frac {u}{2} +\frac{1}{3}\right)(cL)^k,
\left(\frac {u+1}{2} -\frac{1}{12}\right)(cL)^k \right] \times \cD^k_v,
\end{eqnarray*}
as well as
\begin{equation*}
  D^\cK_l(S^k_{u,v})=\left[\left(\frac{u-1}{2}
+\frac{1}{12}\right)(cL)^k,
\left(\frac{u}{2} -\frac{1}{3}\right)(cL)^k \right]
\times \cD^{k,\cK}_v,
\end{equation*}
\begin{equation}
D_r^\cK(S^k_{u,v})= \left[\left(\frac {u}{2} +\frac{1}{3}\right)(cL)^k
,\left(\frac {u+1}{2} -\frac{1}{12}\right)(cL)^k \right] \times \cD^{k,\cK}_v,
\label{Dkv}
\end{equation}
where $\cD^k_v, \cD^{k,\cK}_v$ have been defined before (see
\eqref{3.44aaa},\eqref{3.44abc}, \eqref{3.44bbb} and the lines
following it, and \eqref{3.44ccc}).

We remind the reader of \eqref{329}. After rearranging the pairs
$(j,i)$ in the order of \eqref{3.4d} this says for a good layer
$H^k_v$ we have
%%next define $\cK^k_v$ for $k \ge 1$ if $H^k_v$ is good by
\begin{equation}
\cK^k_v = \cH^k_v = [\al^k_v, \om^k_v] \text{ if $H^k_v$ is good of
type 2},
\label{348}
\end{equation}
and
\begin{equation}
\cK^k_v = [\al^k_v, \al(\wt \cC_{k,j})-1] \text{ if $H^k_v$ is good of
type 1 and $\cH^k_v = [\wt \al^k_{j,0},\wt \om^k_{j,0}]$}.
\label{349}
\end{equation}
We next define
\begin{equation}
Ker(S^k_{u,v})=S^k_{u,v}\cap (\Bbb Z \times \cK^k_v);
\label{ker}
\end{equation}
this set is called the {\it kernel of }$S^k_{u,v}$.
Note that $Ker(S^k_{u,v})$ equals $S^k_{u,v}$ if this site is
good of type 2, but is a
strict subset of $S^k_{u,v}$ if $S^k_{u,v}$ is good of type 1.
Basically, if $S^k(u,v)$ is a good $k$-site of type 1, then
its projection on the vertical axis contains exactly one
cluster $\wt \cC_{k,j}$ of mass $k$ (and none of mass greater than
$k$). The kernel of $S^k$
is then the part of $S^k(u,v)$ whose projection on the vertical axis
lies below $\wt\cC_{k,j}$ (compare \eqref{348} and \eqref{349}). We
observe that if $H^k_v$ is a good layer and $k \ge 1$, then
\begin{equation}
\text{top line of } Ker(S^k_{u,v}) = \text{ top line of }\cK^k_v =
\text{top line of }\cD^{k,\cK}_v = \text{ top line of }D_{\theta}^{\cK}(S^k_{u,v}).
\label{3.50abc}
\end{equation}
The second equality here follows from \eqref{3.44aaa},
\eqref{3.44abc}, \eqref{3.44bbb}, \eqref{3.44ccc} and the fact that
$\wt \om^{k-1}_{i_j-1} = \wt \al^{k-1}_{i_j}-1$ (by \eqref{3.4d} for
$t=k-1$).

We define further
%%Oct 6(Harry) CHANGED s TO w IN NEXT 4 LINES. THE LETTER s IS USED
%%TOO MUCH ALREADY
\begin{equation}
\cF^k_v=\cH^{k-1}_{w_v},
\label{3.50zyy}
\end{equation}
where $w_v=\min\{w\colon H^{k-1}_w \subset H^k_v\}$, and
\begin{equation}
F(S^k_{u,v})= [(u/2-1/6)(cL)^k,(u/2+1/6)(cL)^k]  \times \cF^k_v.
\label{cen}
\end{equation}
$F(S^k_{u,v})$ is, roughly speaking,
the middle third of the lowest $(k-1)$-layer in
$S^k_{u,v}$. The $(k-1)$-sites contained in $F(S^k_{u,v})$ are
said to be {\it centrally located} in $S^k_{u,v}$.

\bigskip
\noindent {\bf A reversed partition.}

%%Aug 31,2005: CHANGE THE TEXT A LITTLE. dOES IT SOUND BETTER?
%%Whenever a good $k$-site contains bad $(k-1)$-sites, they are
%%located at its upper part, which ensures high crossing
%%%%WHAT DO YOU MEAN BY ``UPPER PART'' AND WHY IS THIS TRUE ???
%% Aug 31,2005. this is not a precise statement. it is true
%% from the construction we did before I changeD the sentence

The good $k$-sites were constructed in a way which ensures high crossing
probabilities for open paths from the bottom of the $k$-site. A
technical tool used in the proof requires also good crossing
probabilities in the downwards direction. This is achieved by an
appropriate modification of the horizontal layers.
\medskip

\noindent {\bf Reversed sites.} We begin with the definition of the
reversed layers.
%%(Sep 5 (Harry)I CHANGED NEXT LINE SOMEWHAT
\medskip

\noindent {\bf Step 0.} The 0-layers are $\widehat H^0_j=H_j$.
\medskip

\noindent {\bf Step 1.} Take ${\text{\bf C}}_{1} = \{ \cC_{1
,j} \}_{j\ge 1}$ as before. Take $\widetilde b_j$ as in
\eqref{3.3}, and  set
\begin{equation}
\widehat b_1=\widetilde b_1 +1\text{  and }
\widehat b_j = \begin{cases}\widetilde b_j , \quad &\text{if} \quad m
(\cC_{1 ,j-1}) =1 \cr \widetilde b_j +1 , \quad &\text{if}
\quad m (\cC_{1 ,j-1}) > 1
\end{cases} \quad \text{ for }j\ge 2.
\label{3.1ks}
\end{equation}
With the convention $\widehat\om^{1}_{0,0}=0$ and $\widehat
\al^{1}_{0,0}=1$, we set, for each $j \ge 1$,
\begin{equation}
\widehat \om^{1}_{j,0} = \begin{cases}   \al(\cC_{1 ,j})-3, \quad &
\text{if} \quad m (\cC_{1 ,j}) =1, \cr  \al(\cC_{1 ,j}),
\quad & \text{if} \quad m (\cC_{1 ,j}) >1, \end{cases}
\label{3.1'}
\end{equation}
and
\begin{equation}
\widehat \al^{1}_{j,0} = \begin{cases}   \widetilde \om^1_{j,1}, \quad &
\text{if} \quad m (\cC_{1 ,j}) =1, \cr \om(\cC_{1 ,j}),
\quad & \text{if} \quad m (\cC_{1 ,j})>
1,
 \end{cases}  \label{3.2'}
\end{equation}
with $ \widetilde\om^1_{j,1}$  given by \eqref{3.4a}. If $m(\cC_{1,j}) =
1$, then we set, for $j \ge 1$,
\begin{equation*}
 \widehat \al^{1}_{j,i} =\begin{cases} \widetilde \om^1_{j,i+1}, \quad &\text{if}
\quad  1 \le i \le \widehat b_{j+1}-2,\cr \widehat
\om^{1}_{j+1,0}-1, \quad & \text{if} \quad i = \widehat
b_{j+1}-1,\end{cases}
\end{equation*}
while, for  $m (\cC_{1 ,j})>1$ or $j = 0$ (and $m(\cC_{1,0}) = 1$),
we set
\begin{equation}
 \widehat \al^{1}_{j,i} =\begin{cases} \widetilde \om^1_{j,i}, \quad &\text{if}
\quad  1 \le i \le \widehat b_{j+1}-2,\cr \widehat
\om^{1}_{j+1,0}-1, \quad & \text{if} \quad i = \widehat
b_{j+1}-1.\end{cases}
\label{3.39aa}
\end{equation}
Irrespective of the value of $m(\cC_{1,j})$ we further take for $j \ge 0$
\begin{equation}
\widehat \om^{1}_{j,i} =\widehat \al^{1}_{j,i-1}+1 \quad\text{if}
\quad  1 \le i \le \widehat b_{j+1}-1.
\label{3.39bb}
\end{equation}
%%Sep4(Harry) THE NEXT 2 SENTENCES SEEMS MISPLACED. I TOOK THEM OUT.
%%In the above we have assumed that $L \ge 108$ so that
%%$\widehat\al^{1}_{j,i} \ge \widehat\om^{1}_{j,i}$. In particular,
%%$\cC_{1,j} \subseteq [\widehat\om^{1}_{j,0},\widetilde\om^1_{j,0}]$.

For $j\ge 0$ and $0 \le i \le \widehat b_{j+1}-1$  define
$$
\widehat H^1_{j,i} = \bigcup_{s\in [\widehat
\om^{1}_{j,i}, \widehat \al^{1}_{j,i}]}\widehat H^0_s.
$$
The family
$\{\widehat H^1_{j,i}, \; j \ge 0, \; 0 \le i \le \widehat
b_{j+1}-1\}$ is relabeled as $\widehat \bH^1 = \{ \widehat H^1_j
\}_{j\ge 1}$ in increasing order.
We then define $\widehat \om^1_j$ and $\widehat \al^1_j$ through the
requirement
 $$\widehat H^1_j =
\bigcup_{s \in [\widehat \om^1_j,
\widehat \al^1_j]}\widehat H^0_s,
$$ for each $j$, and we set $\widehat {\cH}^1_j=[\widehat \om^1_j, \widehat \al^1_j]$.

\medskip
\noindent {\bf Step k.} We now consider the clusters in
${\text{\bf C}}_{k} = \{ \cC_{k ,j} \}_{j\ge 1}$
which have mass at least $k \ge 2$,  and rename them as
$\wt \cC_{k,j}, j \ge 1$  (always in increasing order).
 We make the convention that $\widehat \om^k_{0,0}=0$
and $\widehat \al^k_{0,0}= \om^{k-1}_3$.
\begin{equation}
\widehat
b^k_1=\widetilde b^k_1+1,\text{ and }
\widehat b^k_j = \begin{cases}\widetilde b^k_j , \quad &\text{if} \quad m
(\wt \cC_{k,j-1}) =k,\cr \widetilde b^k_j +1 , \quad &\text{if}
\quad m (\wt\cC_{k,j-1}) >k,\end{cases}
\quad \text{ for }j \ge 2.
\label{3.1kr}
\end{equation}
Since $m(\wt \cC_{k,j})\ge k$, we can use the same argument as for
 \eqref{3Bm} to show that
for each $j$ there exist $h_j\le h'_j$ such that $\wt
\cC_{k,j}= [\widehat\om^{k-1}_{h_j},\widehat \al^{k-1}_{h'_j}] \cap
\Ga$ (use \eqref{3.39dd} below instead of \eqref{3.44bb}).
Again we avoid double indexing, but $h_j,h'_j$  depend on
the step.

We now define, for $j \ge 1$:
\begin{equation}
\widehat \om^{k}_{j,0} = \begin{cases}   \widehat \om^{k-1}_{h_j-3}, \quad
& \text{if} \quad m (\wt\cC_{k,j}) =k, \cr  \widehat
\om^{k-1}_{h_j}, \quad & \text{if} \quad m (\wt\cC_{k,j}) >k,
\end{cases} \label{3.1rk}
\end{equation}
and
\begin{equation}
\widehat \al^{k}_{j,0} = \begin{cases}   \widehat\al^{k-1}_{\hat s(j,1)},
\quad & \text{if} \quad m (\wt\cC_{k,j}) =k, \cr
\widehat\al^{k-1}_{h'_j}, \quad & \text{if} \quad m (\wt\cC_{k,j}) > k,
 \end{cases} \label{3.2rk}
\end{equation}
where $\hat s(j,i)$ is such that $\om(\wt \cC_{k,j})+iL^k/3 \in
\widehat{\cH}^{k-1}_{\hat s_(j,i)}$, for $1 \le i \le \widehat
b^k_{j+1}-1$,
%%Oct 6(Harry) INSERTED NEXT LINE
and $\om(\wt \cC_{k,0})$ is interpreted as zero.
We also set, for $j \ge 1$, and if  $m(\wt
\cC_{k,j})=k$,
\begin{equation}
 \widehat \al^{k}_{j,i} =\begin{cases} \widehat \al^{k-1}_{\hat s(j,i+1)},
\quad &\text{if}
\quad  1 \le i \le \widehat b^k_{j+1}-2,\cr \widehat
\om^{k}_{j+1,0}-1, \quad & \text{if} \quad i = \widehat
b^k_{j+1}-1,\end{cases}
\label{3.39ee}
\end{equation}
while if $m(\tilde \cC_{k,j})>k$ or $j=0$,
\begin{equation}
 \widehat \al^{k}_{j,i} =\begin{cases} \widehat \al^{k-1}_{\hat s(j,i)},
\quad &\text{if}
\quad  1 \le i \le \widehat b^k_{j+1}-2,\cr \widehat
\om^{k}_{j+1,0}-1, \quad & \text{if} \quad i = \widehat
b^k_{j+1}-1,\end{cases}
\label{3.39cc}
\end{equation}
and finally,
for $j \ge 0$,
\begin{equation}
\widehat \om^{k}_{j,i} =  \widehat \al^{k}_{j,i-1}+1 \quad \text {if}
\quad 1 \le i \le \widehat b^k_{j+1}-1.
\label{3.39dd}
\end{equation}
We then set
\begin{equation*}
\widehat H^k_{j,i}=\bigcup_{s\colon \widehat
\cH^{k-1}_s\subset [\widehat \om^{k}_{j,i},\widehat \al^{k}_{j,i}]}
\widehat H^{k-1}_s,
\end{equation*}
if $j \ge 0,\; 0 \le i \le \wh b^k_{j+1}-1$.
In the next lemma we relate the layers $\{\wh H^k_{j,i}, j \ge 0, 0 \le
i \le \wh b^k_{j+1}-1\}$ to the layers $\{\wt H^k_{j,i}, j \ge 0, 0 \le
i \le \wh b^k_{j+1}-1\}$. In particular, we state that for
fixed $k$ the intervals
$[\wh\al^{k}_{j,i} ,\wh \om^{k}_{j,i}],
\;j \ge 0, 0 \le i \le \wh b_{j+1}^k-1$
form a partition of $\Bbb Z_+$.
As before we can then re-label the layers
$\{\widehat H^k_{j,i}, j \ge 0, 0 \le i \le
\wh b^k_{j+1}-1\}$ as $\widehat \bH^k=\{\widehat H^k_j, j \ge 1\}$ in
increasing order, and define $\widehat \om^k_j$ and $\widehat
\al^k_j$ through the identity
\begin{equation*}
\widehat H^k_j = \bigcup_{s\colon \widehat H^{k-1}_s\subset
[\widehat \om^k_j, \widehat \al^k_j]} \widehat H^{k-1}_s,
\end{equation*}
for each $j$, and $\widehat{\cH}^k_j=[\widehat \om^k_j,
\widehat \al^k_j]$. As we shall see in Lemma 3.3,
many of the layers in $\wh H^k$ ``nearly''
coincide with a layer in $H^k$.

 We shall need some properties of the reversed layers which parallel
 the properties in Lemma \ref{lemma3.1}. First some definitions. The {\it
 support} of the reversed layer $\wh H^t_{j,i}$ is the interval
$[\wh \om^t_{j,i}, \wh \al^t_{j,i}]$. The reversed layer $\wh
 H^t_{j,i}$ is called {\it good of type 2} if $1 \le i \le \wh
 b^t_{j+1}-1$, and {\it good of type 1} if $i = 0$ and $m(\wt \cC_{t,j}) =
 t$. It is called {\it bad} if $i=0$ and $m(\wt \cC_{t,j}) > t$.
\begin{eqnarray*}
\wh A_t :&=& \sup\{(\wh \al^t_j-\wh \om^t_j): j \ge 1, \wh H^t_j \text{
 is a good reversed $t$-layer}\}\\
&=& \sup\{(\wh \al^t_{j,i}-\wh \om^t_{j,i}): j \ge 0, 0 \le i \le \wh
b^t_{j+1} -1, \wh H^t_{j,i} \text{
 is a good reversed $t$-layer}\}.
\end{eqnarray*}
\begin{lemma}\label{lemma3.2}
Let $L \ge 108$ and let $\ga $ be an environment
with $\chi(\ga) = 0$. Then the following properties hold:
\begin{eqnarray}
&\nonumber\text{for all $t, j \ge 1$  the support of a bad} \text{ reversed
$(t-1)$-layer cannot intersect the interval}\\
&[1, \al(\wt \cC_{t,1})-1] \text{ or any interval }
[\om(\wt \cC_{t,j})+1, \al(\wt \cC_{t,j+1})-1]\; ;
\label{3.40zb}
\end{eqnarray}
\begin{equation}
\text{for all $t\ge 1,  j \ge 0, 1 \le i \le \wt b^t_{j+1}-1$},
\wh H^{t-1}_{\wh s^{t}(j,i)}\text{ is a good reversed $(t-1)$-layer}\;;
\label{3.50ab}
\end{equation}
\begin{eqnarray}
\nonumber\text{for }&t \ge 1, j \ge 1, 1 \le \l \le L/3
\text{ and }m(\wt \cC_{t,j}) = t,
\wh H^{t-1}_{h_j-\l}\\
& \text{ is a good reversed $(t-1)$-layer}\;;
\label{3.50pb}
\end{eqnarray}
\begin{equation}
\text{for }t \ge 1, \wh A_t \le 2L^t;
\label{3.50bbb}
\end{equation}
\begin{equation}
\wh\om^{t}_{j,i} \le \wh \al^{t}_{j,i} \text{ for }
j \ge 0,\; 0 \le i \le  \wh b_{j+1}^t-1;
\label{3.4cb}
\end{equation}
\begin{eqnarray}
\nonumber&\text{the intervals }[\wh\om^{t}_{j,i} ,\wh \al^{t}_{j,i}],
\;j \ge 0, 0 \le i \le \wh b_{j+1}^t-1 \;(\text{ordered as } (j,i) =(0,0),
\dots,\\
&(0,\wh b_1^t-1), (1,0), \dots, (1,\wh b_2^t-1),(2,0),\dots) \text{ form
a partition of $\Bbb Z_+$};
\label{3.4db}
\end{eqnarray}
\begin{eqnarray}
&\text{for $ t \ge 1, 0 \le i \le \wh b^t_{j+1}$,
each $[\wh \om^t_{j,i} , \wh \al^t_{j,i}]$ is a
union of intervals}\\\nonumber
&[\wh \om^{t-1}_{u,v}, \wh \al^{t-1}_{u,v}]
\text{ over a finite number of suitable pairs $(u,v)$}\\
&\text{(for $t=1$ this condition is taken to be fulfilled by convention)}\;;
\label{3.4gb}
\end{eqnarray}
\begin{equation}
\text{for each } j \ge 1,\; \wt \cC_{t,j}
\subset[\wh \om^{t}_{j,0},\wh \al^t_{j,0}]\;;
\label{3.4eb}
\end{equation}
\begin{equation}
\text{for each }j\ge 1,
\text{span}(\wt\cC_{t,j})=[\wh \om^{t}_{j,0},
\wh \al^t_{j,0}]\text{  when } m(\wt\cC_{t,j})>t.
\label{3.44bbc}
\end{equation}
\end{lemma}
We shall not prove this lemma. Its proof is essentially the same as
that of Lemma \ref{lemma3.1} with much tedious definition pushing.
We also note without proof that the analogue of \eqref{3Cm} for
reversed layers is
\begin{equation*}
h_{j+1} - h'_j \ge \frac L6 \text{ for } j \ge 0.
\end{equation*}

\bigskip
We repeat that {\it we always take $L \ge 108$ and $\ga$ such that
$\chi(\ga) =0$}. For $k \ge 1, j \ge 0$ and $ 0 \le i \le b^k_{j+1}-1$
we define $r(k,j,i)$ to be the rank number of the layer $\wt
H^k_{j,i}$ in the partition $\cH^k$ of $\Bbb Z_+$, that is,
$\wt H^k_{j,i} = \cH_{r(k,j,i)}$. Similarly we define $\wh r(k,j,i)$
to be the unique number such that $\wh H^k_{j,i} = \wh\cH^k_{\wh
r(k,j,i)}$ for $k \ge 1, j \ge 0, 0 \le i \le \wh b^k_{j+1}-1$.

The next lemma shows the relation between the partitions $\cH^k$ and $\wh \cH^k$.
Each interval in one of these partitions differs not too much from one of the intervals
in the other partition. This is especially important for the intervals
which contain the clusters $\wt \cC_{k,j}$. If such a cluster has
mass $m(\wt \cC_{k,j}) > k$, then there even exists in each of the
partitions an interval equal to the span of $\wt \cC_{k,j}$ (see \eqref{3.44bb}
and \eqref{3.44bbc}).
\begin{lemma} \label{lemma3.3}
For $k \ge 1, j \ge 0$ and $ 0 \le i \le b^k_{j+1}-1$,
\begin{equation}
r(k,j,i) = \sum_{\l = 0}^{j-1} b^k_{\l+1} + i+1 =
\sum_{\l = 0}^{j-1}\wt b^k_{\l+1} + \sum_{\l=0}^{j-1} I[m(\wt
\cC_{k,\l+1}) > k] +i+1.
\label{3.70}
\end{equation}
For $k \ge 1, j \ge 1$ and $ 0 \le i \le \wh b^k_{j+1}-1$,
\begin{equation}
\wh r(k,j,i) = \sum_{\l = 0}^{j-1} \wh b^k_{\l+1} + i+1 =
\sum_{\l = 0}^{j-1}\wt b^k_{\l+1} + 1+\sum_{\l=1}^{j-1} I[m(\wt
\cC_{k,\l}) > k] +i+1
\label{3.71}
\end{equation}
If $j = 0$, then
\begin{equation}
r(k,0,i) = i+1 \text{ for } 0 \le i \le b^k_1-1 \text{ and } \wh
r(k,0,i) = i+1 \text{ for } 0 \le i \le \wh b^k_1-1.
\label{3.71aa}
\end{equation}
For all $k,u \ge 1$,
\begin{equation}
|\cH^k_u \triangle \widehat {\cH}^k_u| \le   20L^{k-1}
\label{3.72}
\end{equation}
($\triangle$
denotes  symmetric difference and $|\cdot|$ the cardinality).
\end{lemma}
\begin{proof} Recall that $\cH^k$ is just the increasing rearrangement
of the intervals $[\wt \al^k_{j,i}, \wt \om^k_{j,i}]$ with $0 \le i \le
b^k_{j+1}, j \ge 0$. The first equality in \eqref{3.71} is therefore
immediate from \eqref{3.4d}
and the fact that there are exactly $b^k_{\l+1}$ layers $\wt H^k_{\l,i}$
with $0 \le i \le b^k_{\l+1}$. The second equality then follows from \eqref{3Ag}.
The equalities in \eqref{3.71} follow in the same way from \eqref{3.4db}
and \eqref{3.1kr}. The first equality in \eqref{3.71aa} is immediate,
because $\cH^k$ begins with the supports of the layers $\wt H^k_{0,i},
0 \le i \le b^k_1-1$. The second equality in \eqref{3.71aa} follows in a
similar manner.

To prove \eqref{3.72} we note first that for $j \ge 1, 0 \le i \le
(b^k_{j+1} \land \wh b^k_{j+1})-1$,
\begin{equation}
\wh r(k,j,i) = \sum_{\l = 0}^{j-1}\wt b^k_{\l+1} + 1+\sum_{\l=0}^{j-2} I[m(\wt
\cC_{k,\l+1}) > k] +i+1 = r(k,j,i) + I[m(\wt \cC_{k,j}) = k],
\label{3.77}
\end{equation}
by virtue of \eqref{3.71}, \eqref{3.72}. We further note that for all $k
\ge 1, j \ge 0$ it holds
\begin{equation}
b^k_{j+1}\in \{\wt b^k_{j+1}, \wt b^k_{j+1}+1\} \text{ as well as }
\wh b^k_{j+1} \in \{\wt b^k_{j+1}, \wt b^k_{j+1}+1\}.
\label{3.77a}
\end{equation}
%%Consequently,
%%$$
%%|b^k_{j+1} - \wh b^k_{j+1}| \text{ for } k \ge 1, j \ge 0.
%%\teq(3.77a)
%%$$
Now fix $k \ge 1$. As $(i,j)$ traverses all pairs
$j \ge 0, 0 \le i \le b^k_{j+1}-1$
in the order given in \eqref{3.4d} $r(k,j,i)$ runs though the positive
integers in order. Equivalently, $\wt \cH_{r(k,j,i)}^k$ runs through
the intervals $\cH_u^k, u=1,2, \dots$ in order. We now must
distinguish different cases. Let $r(k,j,i) = u$ and hence
$\wt \cH^k_u = \wt \cH^k_{j,i}$
for a certain $(j,i)$
with $0 \le i \le b^k_{j+1}-1$. Then we have the
cases (i) $m(\wt {\cC}_{k,j}) = k, j \ge 1, 1\le i \le b^k_{j+1}-1$;
%%%\le \wh b^k_{j+1}-1$;
(ii) $m(\wt {\cC}_{k,j})=k, j \ge 1$ but $i = 0$;
(iii)\;$m(\wt {\cC}_{k,j}) > k, j \ge 1, 0 \le i \le b^k_{j+1}-1$;
(iv)\; j=0.
%%(iv) $m(\wt {\cC}_{k,j}) > k, i =  \wh b^k_{j+1}$. Note that these
%%exhaust all possibilities because $i \le b^k_{j+1}-1$ plus
%%\equ(3.77a)
%%implies $i\le \wh b^k_{j+1}$.
To make the argument clear, let us start with case
(iii), which seems to be the simplest case. In this case \eqref{3.77}
shows that $\wh r(k,j,i) = r(k,j,i) = u$, but there is one proviso. We
can apply \eqref{3.77} only if $(j,i)$ is a legitimate pair, that is, if $i
\le \wh b^k_{j+1}-1$. However, by \eqref{3.1ks} and \eqref{3.1kr} $\wh b^k_{j+1}
= \wt b^k_{j+1} +1$, and hence  by \eqref{3.77a} and $m(\wt \cC_{k,j})>k$, we
automatically have $i \le b^k_{j+1}-1 \le \wh b^k_{j+1}-1$ in case (iii)
(recall the we have $i \le b^k_{j+1}-1$ by choice of i; see \eqref{3.4d}).
Thus $\wh r(k,j,i) = r(k,j,i)$ for all pairs $(j,i)$ in case (iii).
Consequently, \eqref{3.72} reduces in case (iii) to
\begin{equation}
\big|[\wt \al^k_{j,i},\wt \om^k_{j,i}] \triangle
[\wh \om^k_{j,i},\wh \al^k_{j,i}]\big| \le 20L^{k-1}.
\label{3.78}
\end{equation}
We shall verify this inequality in some subcases below, but first let
us consider cases (i) and (ii).
In case (i) $\wh r(k,j,i) = r(k,j,i) +1= u+1$, so $\wh H^k_u$ is the
predecessor of $\cH^k_{j,i}$ in the order of \eqref{3.4db}, i.e. $\wh
\cH^k_u = \wh \cH^k_{j,i-1}$. Thus in case (i) \eqref{3.72} reduces to
\begin{equation}
\big|[\wt \al^k_{j,i},\wt \om^k_{j,i}] \triangle
[\wh \om^k_{j,i-1},\wh \al^k_{j,i-1}]\big| \le 20L^{k-1}.
\label{3.78c}
\end{equation}
(Note that $(j,i-1)$ is a legitimate pair, because $i \le b^k_{j+1}-1$
together with \eqref{3.77a} implies $i-1 \le \wh b^k_{j+1}-1$.) Finally,
in case (ii) with the further restriction $j \ge 1$ we have $\wh
r(k,j,0) = r(k,j,0)+1$, and consequently we will have $\wh
\cH^k_{j',i'} = \wh \cH^k_u$ for $u= r(k,j,0)$ if $(j',i')$ is
the immediate predecessor
of $(j,0)$, that is, if $(j',i') = (j-1, \wh b^k_j-1)$.
Thus case (ii) reduces to
\begin{equation}
\big|[\wt \al^k_{j,0},\wt \om^k_{j,0}] \triangle
[\wh \om^k_{j-1,\wh b^k_j-1},\wh \al^k_{j-1,\wh b^k_j-1}]
\big| \le 20L^{k-1}.
\label{3.78cz}
\end{equation}
Case (iv) with $j=0$ needs special treatment, but this is easy by means
of \eqref{3.71aa}.
%% because $r(k,0,0) = \wh r(k,0,0) = 1$ and $\cH^k_1 = \wh \cH^k_1 = [0,
%% \om_3^{k-1}]$.

To complete the proof we now prove \eqref{3.72} in a few subcases.
We shall only treat some typical examples.
In all cases we shall use the estimate
\begin{equation}
\big|[a_1,b_1] \triangle [a_2,b_2]\big| \le |a_1-a_2|+|b_1-b_2|,
\label{3.77ab}
\end{equation}
which is valid for any intervals $[a_i, b_i], i=1,2$
Thus to prove \eqref{3.78} and \eqref{3.72} in case (iii) it suffices to verify
\begin{equation}
\big|\wt \al^k_{j,i}-\wh \om^k_{j,i}\big|+
\big|\wt \om^k_{j,i} -\wh \al^k_{j,i}\big| \le 20L^{k-1}.
\label{3.77bc}
\end{equation}

Now consider the subcase of (iii)
\begin{equation}
k \ge 2, j \ge 1, m(\wt \cC_{k,j}) > k,2 \le  i \le b^k_{j+1}-2.
\label{3.77c}
\end{equation}
In this subcase
\begin{equation}
[\wt \al^k_{j,i},\wt \om^k_{j,i}] = [\om^{k-1}_{s(j,i-1)}+1,
\om^{k-1}_{s(j,i)}]\text{ and } [\wh \om_{j,i}^k, \wh \al_{j,i}^k] =
[\wh \al^{k-1}_{\wh s(j,i-1)}+1, \wh \al^{k-1}_{\wh s(j,i)}].
\end{equation}
Moreover, by definition, $s(j,i)$ satisfies \eqref{3.aaaa}, while
\begin{equation*}
 \sup\{(\om^{k-1}_j - \al^{k-1}_j): j \ge 2, H^{k-1}_j
\text{ is a good $(k-1)$-layer}\} \le A^{k-1} \le 2L^{k-1}
\end{equation*}
(by \eqref{3.00} and \eqref{3.50a}). Therefore,
\begin{equation}
\big|\wt \al^{k-1}_{s(j,i)} - \om(\wt \cC_{k,j}) - (i-1)L^k/3\big| \le
2L^{k-1} \text{ and }
\big|\wt\om^{k-1}_{s(j,i)} - \om(\wt \cC_{k,j}) - iL^k/3\big| \le
2L^{k-1}
\label{3.77b}
\end{equation}
(see also \eqref{3.50b}).
In turn, this implies that the left hand side of \eqref{3.77bc} is
changed by at most $4L^{k-1}+1$
if we replace $\wt \al^k_{j,i}$ and $\wt \om^k_{j,i}$ by
$\om(\wt \cC_{k,j}) + (i-1)L^k/3$ and
$\om(\wt \cC_{k,j}) + iL^k/3$, respectively.
Similarly, in the subcase \eqref{3.77c} it holds
\begin{equation}
\wh \om^{k-1}_{\wh s(j,i)} \le \om(\wt \cC_{k,j}) + iL^k/3 \le
\wh \al^{k-1}_{\wh s(j,i)} ,
\label{3.77cs}
\end{equation}
so that we have
\begin{equation}
\big|\wh \al^{k-1}_{\wh s(j,i)} - \om(\wt \cC_{k,j}) - iL^k/3\big|
\le 2L^{k-1}.
\label{3.77d}
\end{equation}
Thus, if we replace $\wh \om^k_{j,i}$ and $\wh \al^k_{j,i}$ by $\om (\wt
\cC_{k,j}) + (i-1)L^k/3$ and $\om (\wt \cC_{k,j}) + iL^k/3$,
respectively, then the left hand side of \eqref{3.77bc}
is again not changed by more than $4L^{k-1}+1$.
These considerations prove \eqref{3.78}, and hence \eqref{3.72},
when \eqref{3.77c} holds.

We also consider some subcases of case (ii).
To verify that \eqref{3.78cz} holds in case (ii) we shall use the bounds
\begin{equation}
\om(\wt \cC_{k,j}) = \om^{k-1}_{i_j}\le \om^{k-1}_{i_j+3}
\le \om^{k-1}_{i_j} + 3(2L^{k-1}+1)
\le \om(\wt \cC_{k,j}) + 3(2L^{k-1} + 1),
\label{3.75ea}
\end{equation}
and
\begin{equation}
\wh \om^{k-1}_{h_j}  - 3(2L^{k-1} + 1) \le \wh \om^{k-1}_{h_{j-3}}
\le \wh \om^{k-1}_{h_j} =
\al(\wt \cC_{k,j}).
\label{3.75eb}
\end{equation}
The first inequality in \eqref{3.75ea} is obvious and the second one
follows by the argument for \eqref{3.50qq}. The third inequality is
again obvious since $\om_{i_j} \le \om_{i'_j} = \om(\wt \cC_{k,j})$.
The inequalities in \eqref{3.75eb} follow by similar arguments.

Case (ii) is split into two subcases, namely  (iia) $j \ge 2,
m(\wt \cC_{k,j-1}) = k,m(\wt \cC_{k,j}) =
k, i=0$ and (iib) $j=1$ or ($j\ge 2$ and $m(\wt \cC_{k,j-1}) > k$),
$m(\wt \cC_{k,j}) = k, i=0$.
In each subcase the further subcase $k=1$ requires the use of the
special definitions of Step 1. These will allow us to verify
\eqref{3.72} without use of $\al^0$ or $\om^0$ (i.e., a superscript of 0
combined with further subscripts) in intermediate steps.
The steps for $k=1$ are essentially the same as for $k \ge 2$ and
we shall restrict ourselves here to the cases with $k \ge 2$.

In subcase (iia) we have $b^k_j = \wt b^k_j = \wh b^k_j$ and $k \ge 2,
j \ge 2$ and for suitable $\theta_\l \in [-1,1]$ (which may depend on $k,j$)
\begin{eqnarray}
&&\wt \al^k_{j,0} =\wt \om^k_{j-1, b^k_j-1}+1
= \om^{k-1}_{s(j-1, b^k_j-1)}+1
= \om(\wt \cC_{k,j-1}) +(b^k_j-1)L^k/3 + \theta_1(2L^{k-1}+1),\\\nonumber
&&\wt \om^k_{j,0} = \om^{k-1}_{i_j+3}
= \om (\wt \cC_{k,j})
+3\theta_2(2L^{k-1}+1) \text{ (by \eqref{3.75ea})},\\\nonumber
&&\wh \om^k_{j-1, \wh b^k_j-1}= \wh \al^k_{j-1, \wh b^k_j-2}+1=
\wh \al^{k-1}_{\wh s(j-1,\wh b^k_j-1)} +1 = \om(\wt \cC_{k,j-1}) +(\wh
b^k_j-1)L^k/3 +\theta_3(2L^{k-1}+1),\\\nonumber
&&\wh \al^k_{j-1, \wh b^k_j-1} = \wh \om^k_{j,0}-1 = \wh
\om^{k-1}_{h_j-3} =\wh \om^{k-1}_{h_j} + 3\theta_4(2L^{k-1}+1)
=\al(\wt \cC_{k,j}) + 3\theta_4(2L^{k-1}+1).
\label{3.75fb}
\end{eqnarray}

In subcase (iia) the inequality \eqref{3.78cz}, and hence \eqref{3.72},
follows from these relations together with $b^k_j =
\wh b^k_j$ and $|\al(\wt \cC_{k,j}) - \om(\wt \cC_{k,j})| \le 3L^{k-1}$
(see \eqref{2.6}).

In subcase (iib) we have $b^k_j = \wt b^k_j = \wh b^k_j-1$.
The first, second and fourth line of \eqref{3.75fb} need no change for
this subcase, but in the third line we have
to appeal to \eqref{3.39cc} instead of \eqref{3.39ee}. This third line now
has to be replaced by
\begin{equation*}
\wh \om^k_{j-1, \wh b^k_j-1}= \wh \al^k_{j-1, \wh b^k_j-2}+1=
\wh \al^{k-1}_{\wh s(j-1,\wh b^k_j-2)} +1 = \om(\wt \cC_{k,j-1}) +(\wh
b^k_j-2)L^k/3 +\theta_3(2L^{k-1}+1).
\end{equation*}
The inequalities \eqref{3.78cz} and \eqref{3.72} in case (iib) follow from
these observations and $\wh b^k_j-2 = b^k_j-1$. Any reader who has followed the proof so far
will be able to complete the remaining cases. \end{proof}

The reversed sites $\widehat{S}^k_{u,v}$ are defined as in
\eqref{sitek}, with $\cH^k_j$ replaced by $\widehat {\cH}^k_j$.
\medskip

\begin{remark1} \label{fhkremark} In view of the preceding lemma it
becomes natural to write $\widehat S(S^k_{u,v}):=\widehat
S^k_{u,v}$ as well as $S(\widehat S^k_{u,v}):=S^k_{u,v}$ for $(u,v)
\in \widetilde {\Bbb Z}^2_+$.
\end{remark1}

\section{Passable sites}
\label{passability}

At the end of this section we state key estimates that will lead to
the proof of Theorem \ref{main}. Before doing that we introduce several
key definitions: a rooted seed, passability from the seed ($s$-passable),
and an open cluster.
\medskip

\noindent {\it $s$- and $c$-Passable sites. Rooted seed.}
\medskip

\noindent {\bf Step 0.} A $0$-site is called $s$-{\it passable}
if and only if the site is open.

\noindent {\bf Rooted $0$-seed.} The rooted $0$-seed
$Q^{(0)}(u,v)$, with root at $(u,v)$, is the set of three open
0-sites in $\widetilde {\Bbb Z}^2_+$:
$$
Q^{(0)}=Q^{(0)}(u,v)= \{(u,v),(u+1,v+1),(u-1,v+1)\}.
$$
The site $(u,v)$ is called the {\it root} of $Q^{(0)}_{u,v}$, and we write
$R(Q^{(0)})=\{(u,v)\}$; the sites $(u-1,v+1)$ and $(u+1,v+1)$ are
called the {\it active sites} of $Q^{(0)}$, and we set
$A(Q^{(0)})=\{(u-1,v+1),$ $(u+1,v+1)\}$. (When the location of the
seed is not important we will suppress the subscript.)

We remind the reader that $\wt {\Bbb Z}^2_+$ is oriented upwards in the second
coordinate. We shall therefore say that $A$ is connected to $B$ by an
open path $\pi$ only if $\pi$ is an open path  which respects
the orientation and with initial and endpoint in $A$ and $B$,
respectively. We call such a path simply an {\it open path from $A$
to $B$}. For $A \subset \wt {\Bbb Z}^2_+$ we call the {\it top
line of} $A$ the subset $\{(x,y) \in A:y=y_0\}$, where $y_0$ has the
maximal value for which this subset is nonempty. If $y_0$ takes the
smallest value for which $\{(x,y) \in A:y=y_0\}$ is non empty, then we
call this subset the {\it bottom line} of $A$. Note that the top line
and bottom line as defined here are not complete lines, not even
intervals, in general.

We next define
%%connectivity,
open clusters, passability and rooted seeds for a general $k \ge
1$. These definitions have
to be used in sequence. First we must use them to define
rooted 0-seed and open cluster of a 0-site; then
passability of a good 1-site (this relies on the definition of a
rooted 0-seed and its open cluster already given above);
then a rooted 1-seed and the open
cluster of a 1-site; next
passability of a good 2-site and a rooted 2-seed, etc.
%To familiarize oneself with these definitions the
%reader is encouraged to first look specifically at the case $k=1$.

Let $\theta(p)$ denote the percolation probability for the homogeneous Bernoulli oriented
percolation model with parameter $p$:
\begin{equation}
\theta(p):=  P_p\{\text{there exists an
infinite open path from the origin}\}.
\label{4.90a}
\end{equation}
We first take $p_G$ large enough so that $ \theta(p_G)>1/2$, and $\rho$
will be some fixed number in the interval $(1/2, \theta(p_G))$.
The constant $c$ was chosen in \eqref{4.00}.
\smallskip

\noindent {\bf Open cluster of a rooted $0$-seed.} The open cluster of a rooted
$0$-seed $Q^{(0)} = Q^{(0)}_{u,v}$ is the collection of $0$-sites $w$ for
which there exists an open path of $0$-sites from $A(Q^{(0)})$ to $w$.
This open cluster is denoted by $U(Q^{(0)})$.

We shall soon need the open cluster of a $0$-seed $Q^{(0)}$
restricted to the kernel of a $1$-site $S^1$ which is located such
that all 0-sites of $A(Q^{(0)})$ are below and adjacent to the middle
third of $S^1$, i.e., adjacent to $F(S^1)$. This will simply be the
the collection of good $0$-sites $w$ for
which there exists an open path of good $0$-sites from $A(Q^{(0)})$ to $w$
and inside $S^1$. Note that only the last restriction is added to the
definition of the open cluster of $Q^{(0)}$.

\smallskip
\noindent
{\bf $s$-Passable $k$-site.} A  good $k$-site $S^k$ is
said to be $s$-{\it passable from a rooted $(k-1)$-seed} $Q^{(k-1)}$
if the following conditions (i), (iis) and (iiis) are satisfied (see Figure 5)
%%Dec 6Harry DEF OF s-PASSABILITY IS NOW AS SUGGESTED BY EULALIA.
%%THERE SEEMS TO BE NO NEED TO DO THE DEF FOR k=1 SEPARATELY
%%ACCORDINGLY I REMOVED THE DEFINITIONS FOR k=1. DO YOU THINK THIS IS
%%TOO HARD ON THE READER TO SUBSTITUTE k=1 ???
%%A CONDITION (iiis) HAS BEEN ADDED TO GET ENOUGH SITES CONNECTED TO THE
%%ORIGIN.THIS LINE MAY NEED CHANGES.
\begin{itemize}
\item[{(s1)}] All 0-sites of $A(Q^{(k-1)})$ are below and adjacent to the
middle third of the bottom layer of $S^k$, i.e., adjacent to $F(S^k)$
(see \eqref{cen} for the definition of $F(S^k)$).
\item[{(s2)}] There exist two rooted $(k-1)$-seeds, $\wt Q_l^{(k-1)}$ and $\wt
Q_r^{(k-1)}$ say, such that their top lines are contained in the top line of
$D_l(S^k)$ and the top line of $D_r(S^k)$, respectively,
and such that there exist open oriented paths of 0-sites, entirely
contained in $S^k$, from 0-sites adjacent to $A(Q^{(k-1)})$ to $R(\wt Q_l^{(k-1)})$ and to
$R(\wt Q_r^{(k-1)})$.
\item[{(s3)}] The open cluster of $Q^{(k-1)}$ restricted to $Ker(S^k)$
contains at least $\rho cL/12$ $(k-1)$-sites in each of
$D^\cK_l(S^k)$ and $D_r^\cK(S^k)$ (see \eqref{Dkv}
for the definition of $D_\theta^\cK$).
\end{itemize}
\smallskip \noindent
{\bf Remark.} We shall denote the leftmost rooted $(k-1)$-seed which
fulfills the
requirements for $\wt Q_l^{(k-1)}$ in (s2) as
$Q_l(S^k)$. Similarly $Q_r(S^k)$ denotes the rightmost
rooted $(k-1)$-seed which fulfills the requirements for
$\wt Q_r{(k-1)}$.
%%Dec 2Harry ADDED NEXT 2 SENTENCES
We further
define $A(S^k) = A(Q_l(S^k)) \cup A(Q_r(S^k))$ and
call the sites in this set the {\it active sites of} $S^k$.
Note that in these definitions $Q_l(S^k), Q_r(S^k)$ and $A(S^k)$
also depend on $Q^{(k-1)}$, even though the notation does not indicate
this. However, in the definition below of the open cluster of a rooted
$k$ seed we shall use the more explicit notation
$Q_\theta(S^k,Q^{(k-1)})$ with $\theta =$ l or r
to indicate this dependence.

\smallskip
\noindent {\bf Rooted $k$-seed.} A rooted $k$-seed is formed by a
rooted $(k-1)$-seed $Q^{(k-1)}$ and three good $k$-sites
\begin{equation*}
S^k_{u,v}, S^k_{u-1,v+1} \; \text{and}\; S^k_{u+1,v+1};
\end{equation*}
such that
\begin{itemize}
\item{(i)} $S^k_{u,v}$ is s-passable from $Q^{(k-1)}$,

\item{(ii)} $S^k_{u-1,v+1}$ and $S^k_{u+1,v+1}$ are passable from
$Q_l(S^k_{u,v})$ and $Q_r(S^k_{u,v})$, respectively.
\end{itemize}
\noindent The corresponding $k$-seed is denoted by
\begin{equation*}
Q^{(k)}=S^k_{u,v} \cup S^k_{u-1,v+1}\cup S^k_{u+1,v+1}\cup
Q^{(k-1)}.
\end{equation*}
We set
\begin{eqnarray*}
 R(Q^{(k)})=& R(Q^{(k-1)}), \\\nonumber
 A(Q^{(k)})=& A(Q_l(S^k_{u-1,v+1}))\cup A(Q_r(S^k_{u-1,v+1}))\\\nonumber
 & \cup A(Q_l(S^k_{u+1,v+1}))\cup A(Q_r(S^k_{u+1,v+1})).
\end{eqnarray*}

The 0-site  $R(Q^{(k)})$ is called the root of $Q^{(k)}$; the
sites in $A(Q^{(k)})$ are called the active sites of $Q^{(k)}$.

\smallskip
\noindent
{\bf Remark.} We point out that the locations of $D_l(S^k_{u,v})$ and
$D_r(S^k_{u,v})$ are such that the definition of a rooted $k$-seed
makes sense. Specifically,
the top line of $D_l(S^k_{u,v})$ is adjacent to and just below
$F(S^k_{u-1,v+1})$ and so,
if $S^k_{u,v}$ is $s$-passable, then also $Q_l(S^k_{u,v})$
is adjacent to and just below $F(S^k_{u-1,v+1})$. Thus, it makes sense
to speak of $s$-passability of $S^k_{u-1,v+1}$ from
$Q_l(S^k_{u,v})$. Similar statements hold for $S^k_{u+1,v+1}$ and
$Q_r(S^k_{u,v})$.

\smallskip
\noindent
{\bf Remark.} Note that in the definition of a rooted 0-seed we
required the three 0-sites which make up the seed to be open. Starting
from this fact we deduce the following lemma.

\begin{lemma} \label{Lemma 4.1} In a rooted $k$-seed $Q^{(k)}$ there exists for
each $x \in A(Q^{k)})$ an
open oriented path of 0-sites from $R(Q^{(k)})$ to $x$.
\end{lemma}
\begin{proof} We use a proof by induction on $k$. For $k = 0$ the
conclusion of the lemma is obvious. For the induction step, let $k \ge
1$ and assume that the conclusion of the lemma with $k$ replaced by
$k-1$ has already been proven. Let further
$Q^{(k)}=S^k_{u,v} \cup S^k_{u-1,v+1}\cup S^k_{u+1,v+1}\cup
Q^{(k-1)}$ be a rooted $k$-seed and let $x \in
A(Q_l(S^k_{u-1,v+1}))$. The other possible locations for $x$ in
$A(Q^{(k)})$ can be handled in the same way. Then there exist open
paths of 0-sites $\pi_i$ as follows:
\newline
$\pi_1$ from $y := R(Q_l(S^k_{u-1,v+1}))$ to $x$ (by the
induction hypothesis);
\newline
$\pi_2$ from some point $z$ in $A(Q_l(S^k_{u,v}))$ to $y$ (because
$S^k_{u-1,v+1}$ is s-passable from $Q_l(S^k_{u,v})$);
\newline
$\pi_3$ from $w := R(Q_l(S^k_{u,v}))$ to $z$ (by the induction
hypothesis again);
\newline
$\pi_4$ from some point $a$ in $A(Q^{(k-1)})$ to $w$ (because
$S^k_{u,v}$ is s-passable from $Q^{(k-1)}$);
\newline
$\pi_5$ from $R(Q^{(k-1)})$ to $a$ (by the induction hypothesis once
more).

Now concatenation of the paths $\pi_5, \pi_4, \dots, \pi_1$ gives an
open path of 0-sites from $R(Q^{(k-1)})$ to $x$, as desired.
\end{proof}

\noindent
{\bf Remark.}
If the origin is connected to $R(Q^{(k-1)})$ by an open path, and
$S^k$ is $s$-passable from $Q^{(k-1)}$, it follows that the origin is
connected by an open path of 0-sites to all sites in
$A(S^k)$.

\smallskip
\noindent
{\bf Open cluster of a rooted $k$-seed with $k \ge 1$.} The open
cluster of a rooted $k$-seed  $Q^{(k)}=
S^k_{u,v} \cup S^k_{u-1,v+1}\cup S^k_{u+1,v+1}\cup
Q^{(k-1)}$ is defined as the collection of $k$-sites consisting of
$S^k_{u,v}, S^k_{u-1,v+1}, S^k_{u+1,v+1}$ and the $k$-sites $S$
for which there
exists a sequence $S(1), \dots, S(n)$ of $k$-sites with the
following properties:
\begin{equation}
\text{each $S(j)$ is good},
\label{clus1}
\end{equation}
\begin{equation}
S(n) = S,
\label{clus2}
\end{equation}
\begin{eqnarray}
&\text{for $0 \le j \le n, S(j)$ is s-passable from a rooted $(k-1)$-seed
$\wt Q(j-1)$},\\
&\text{ where, in the notation of the remark following condition (iiis)},\\
& \wt Q(j-1) = Q_{\theta(j-1)}^{(k-1)}\big(S(j-1),
Q_{\theta(j-2)}^{(k-1)}(S(j-2) \big).
\label{clus3}
\end{eqnarray}
Here the
$\theta(i)$ can be $l$ or $r$, independently of each other, and $S(-1)$ is
interpreted as $S^k_{u,v}$, and $S(0) = S^k_{u+\phi,v+1}$ with $\phi =
-1$ if $\theta(0)= l$ and $\phi =+1$ if $\theta(0) = r$.
Also, $\wt Q(-1) = Q^{(k-1)}$.
\smallskip
We define the {\it open cluster restricted to $Ker(S^{k+1})$ of the rooted
$k$-seed $Q^{(k)}$} in the same way as the open cluster of
$Q^{(k)}$, but now with the added restriction that $Q^{(k)}$ and
all $S(j), -1 \le j \le n$, are contained in $Ker(S^{k+1})\cap
\mathring{S}^{k+1}$ (see definition \eqref{ker} for Ker and
\eqref{part} for $k$).
\medskip

\smallskip
\noindent
{\bf $c$-Passable $k$-site.} A  good $0$-site is said to be $c$-{\it passable} if it is open.

For $k\ge 1$, a good $k$-site $S^k$ is said to be $c$-{\it passable}
if:
\begin{itemize}
%\item{(c1)} All 0-sites of $A(Q^{(k-1)})$ are below and adjacent to the
%middle third of the bottom layer of $S^k$, i.e., adjacent to $F(S^k)$
%(see \eqref{cen} for the definition of $F(S^k)$).
\item[(c1)] There exist two rooted $(k-1)$-seeds, $\wt Q_l^{(k-1)}$ and $\wt
Q_r^{(k-1)}$ say, such that their top lines are contained in the top line of
$D_l(S^k)$ and the top line of $D_r(S^k)$, respectively,
and such that there exist open oriented paths of 0-sites, entirely
contained in $S^k$, from the lowest 0-level layer of $F(S^k)$ to $R(\wt Q_l^{(k-1)})$ and to
$R(\wt Q_r^{(k-1)})$.
\item[(c2)] The open cluster of the lowest 0-level layer of $F(S^k)$ restricted to $Ker(S^k)$
contains at least $\rho cL/12$ $(k-1)$-sites in each of
$D^\cK_l(S^k)$ and $D_r^\cK(S^k)$ (see \eqref{Dkv}
for the definition of $D_\theta^\cK$).
\end{itemize}

\begin{defin}\label{denker} We say that a good $k$-site
$S^k$ has an $s$-{\it dense kernel from a seed} $Q^{(k-1)}$ if
condition (s3) holds. If (c2) holds we say that $S^k$ has
$c$-{\it dense kernel}.
\end{defin}

\medskip

\noindent {\bf Remark.} Taking into account the reversed partition, we analogously define the notions of
$\hat c$- and $\hat s$-passable sites.

\medskip

%%We say that a good $k$-site $S^k$ has a $c$-{\it dense kernel}
%%if condition (iiic) holds.

\begin{lemma}\label{Lemma 4.2}(a) Let $[\al^k_v, \om^k_v]$ be a
$k$-interval and let $\cK^k_v$ be defined as in
\eqref{348} and \eqref{349}. If $ \wt I^{k-1}$ is a $(k-1)$-interval
contained in $\cK^k_v$, then $\wt I^{k-1}$ is also good.

(b) Let $[\al^k_v, \om^k_v]$ be good of type 1 and let
$\wt \cC_{k,j}$ be the unique
cluster of $\bold C_k$ of mass at least $k$ in $[\wt \al^k_{j,0},
\wt \om^k_{j,0}]$.
If $[\wt \al^{k-1}_{q,i}, \wt \om^{k-1}_{q,i}]$ is the last
$(k-1)$-interval below
$[\wt\al^k_{j,0},\wt \om^k_{j,0}]$, then this $(k-1)$-interval is good of
type 2.
\end{lemma}

\noindent
{\bf Remark.} Roughly speaking part (a) says that a $(k-1)$-layer in
the kernel of a good $k$-layer is again good.

\begin{proof} (a) We give a proof by contradiction. So assume that $\wt
I^{k-1}$ is a bad $(k-1)$-interval. Then it contains a cluster $\wt
\cC_{k-1, j}\in \bold C_{k-1}$ for some $j$,
of mass at least $k$ and level at most $k-1$. Either $\wt {\cC}_{k-1,j} \in
\bold C_k$, or $\wt \cC_{k-1,j}$ is a constituent of some cluster in
$\bold C_k$. In any case, $\wt I^{k-1}$ intersects a cluster $\wt
\cC_{k,p} \in \bold C_k$ of mass $\ge k$, for some $p$.
Moreover, our choice of the partition $\bold H^k$ is such that for $i
\ge 1$ the interval $[\wt \al^k_{p,i},\wt\om^k_{p,i}]$ does not
intersect any cluster of mass $\ge k$, while for $i = 0$,
$[\wt \al^k_{p,0},\wt\om^k_{p,0}]$ contains $\wt \cC_{k,p}$ and no
other cluster of mass $\ge k$ (see \eqref{3.4e} and \eqref{3.4d}).
However, we assumed that $\wt I^{k-1} \subset
[\al^k_v, \om^k_v]$, so that $\wt \cC_{k,p} \cap [\al^k_v, \om^k_v]$
is non-empty. This implies that $[\al^k_v, \om^k_v]$ equals
$[\wt \al^k_{p,i},\wt\om^k_{p,i}]$ for some $i \ge 0$.

As we already stated, for $i \ge 1$, $[\wt
\al^k_{p,i},\wt\om^k_{p,i}]$ does not intersect any cluster of mass
$\ge k$, so that $i \ge 1$ is incompatible with a non-empty
intersection of $\wt I^{k-1} \subset
[\wt \al^k_{p,i},\wt\om^k_{p,i}]$ with $\wt \cC_{k,p}$.
But also the case $i=0$ is impossible by \eqref{348} and \eqref{349}.
Indeed, if \eqref{348} applies, then $\wt I^{k-1} \subset [\al^k_v,
\om^k_v]$ cannot intersect any cluster of mass $\ge k$, by virtue
of \eqref{3.40z} (note that the hypothesis $\chi(\ga) =0$ is not needed
for \eqref{3.40z}).
On the other
hand, if \eqref{349} applies we must have that $\wt I^{k-1} \subset
\cK^k_v$ lies strictly below $\wt \cC_{k,p}$ and therefore does not
intersect $\wt \cC_{k,p}$.
\medskip
\noindent
(b) Since $[\wt \al^k_{j,0}, \wt \om^k_{j,0}]$
is good of type 1 it contains
exactly one cluster $\bold C_k$ of mass at least $k$,
and this cluster has mass $k$
(see the remark following \eqref{3.50n}). In  our
previous notation this cluster is denoted as $\wt \cC_{k,j}$. By \eqref{3.44bb}
span$(\wt \cC_{k,j}) = [\wt \al^{k-1}_{p,0}, \wt \om^{k-1}_{p,0}]$ for some
$p$. The $(k-1)$-interval preceding this is
$[\wt \al ^{k-1}_{p-1, b_p^{k-1}-1},\wt \om^{k-1}_{p-1,
b_p^{k-1}-1}]$ (see \eqref{3.4d}). We then must have
$(q,i) = (p-1, b^{k-1}_p -1)$. Since $b^{k-1}_p   \ge \wt b^{k-1}_p \ge 3$,
$[\wt \al ^{k-1}_{p-1, b_p^{k-1}-1},\wt \om^{k-1}_{p-1,
b_p^{k-1}-1}]$ is a good $(k-1)$-interval of type 2
(see the lines following \eqref{333k}).
\end{proof}
\begin{lemma}
\label{dkernel} Let $k \ge 1$.
If a good $k$-site $S^k$ has an $s$-dense
kernel from a rooted seed $Q^{(k-1)}$, then $A(Q^{(k-1)})$ is connected by
open paths of 0-sites inside $Ker(S^k)$ to at
least $(\rho  cL/12)^k$ 0-sites in the top line of $Ker(S^k)$, but
to the left (respectively to the right) of the middle third
of this top line.
\end{lemma}

\begin{proof} The proof goes by induction on $k$. Start with $k=1$.
If $S^1$ has an $s$-dense kernel from the 0-rooted seed $Q^{(0)}$,
then there are at least $\rho cL/12$ 0-sites in the open cluster of
$Q^{(0)}$ restricted to $Ker (S^1)$ and in $D^{\cK}_\theta, \theta = l, r$.
Each such 0-site is just a
vertex $w$ for which there is an path in $Ker(S^1)$ of open good
0-sites from $A(Q^{(0)})$ to $w$ and in $D_\theta^\cK, \theta = l$ or $r$.
Moreover, such $w$ automatically lie
in the top line of $Ker(S^1_{u,v})$, because for $k=1$
the cardinality of $\cD^{1, \cK}_{j,i}$ equals 1 for each $(j,i)$ with
$m(\wt \cC_{1,j}) = 1, 0 \le i \le b^1_{j+1}-1$ (see \eqref{3.44aaa},
\eqref{3.44abc} and \eqref{3.50abc})
Thus, for $k=1$, the conclusion of the lemma
is immediate from the definitions of an $s$-dense kernel and of
the open cluster of $Q^{(0)}$.

Now assume that the lemma has already been proven for $k$ replaced by $k-1$.
Assume further that $S^k$ has an $s$-dense kernel from the rooted seed
$Q^{(k-1)} =
S_{u,v}^{k-1}\cup  S_{u-1,v+1}^{k-1} \cup S_{u+1,v+1}^{k-1} \cup
Q^{(k-2)}$. Let $\wt S^{k-1}$ be a $(k-1)$-site which belongs to the
open cluster of $Q^{(k-1)}$. Further, for the sake of argument,
let $\wt S^{k-1}$ lie in $D_l^\cK(S^k)$.Then there exists some $n$ and
sequences $S^{k-1}(0),\dots, S^{k-1}(n)$ and $\theta(0), \dots, \theta(n)$ such that
\eqref{clus1}-\eqref{clus3} with $k$ replaced by $k-1$ and $S$ by
$\wt S^{k-1}$ hold.
In particular, $S^{k-1}(j)$ is passable from the rooted $(k-2)$-seed
$\wt Q(j-1) =
Q^{(k-2)}_{\theta(j-1)}(S^{k-1}(j-1),Q^{(k-2}_{\theta(j-2)}(S^{k-1}(j-2))$.
Also, part
of the definition of $s$-passability gives that the
top line of $\wt Q(j-1)$ will be equal to the top line of $S^{k-1}(j-1)$. A
simple induction argument (with respect to $j$), similar to the proof
of Lemma 4.1, then shows that there
exists a path of open 0-sites in $Ker(S^k)$ from $A(Q^{(k-1)})$
to $R(\wt Q(j))$.
For $j=n$, an application of Lemma 4.1 then shows
that for each vertex $x$ in $A(\wt Q(n))$ there exists an open path
of 0-sites from $A(Q^{(k-1)})$ to $x$.
Since $S^k$ has a dense kernel there are at least $\rho c L/12$
choices for $\wt S^{k-1}$ which are contained in $D^\cK _l(S^k)$
(respectively in $D^\cK _r(S^k)$).
Since different $(k-1)$-sites are disjoint,
there are at least $\rho c L/12$ disjoint choices for $\wt
S^{k-1}$ in each of $D_l^\cK(S^k)$ and $D_r^\cK(S^k)$.
Moreover, if $\wt S^{k-1}= S^{k-1}(n)$ is
any fixed one of the possible
choices, then $\wt S^{k-1}$ is $s$-passable from a rooted seed
$\wt Q(n-1)$, as we just showed.
By the induction hypothesis,
there exist at least $[\rho cL/12]^{k-1}$ 0-sites $y$ in the top line of
$Ker(\wt S^{k-1})$, with the property that there exists an
open path in $Ker(\wt S^{k-1})$ from some $x \in A(\wt Q(n-1))$ to
$y$. Such a
connection can be concatenated with the connection from $A(Q^{(k-1)})$
to $x$, to obtain an open path in $Ker(S^k)$ from
$A(Q^{(k-1)})$ to $y$. But then there at least
$[\rho cL/12]^{k-1}$ choices for $y$ in each possible $\wt S^{k-1}$
and $\rho cL/12$ choices for $\wt S^{k-1}$. In total this gives at
least $[\rho cL/12]^k$ 0-sites with the required open connection
from $R(A^{(k-1)})$.

The 0-sites $y$ constructed in the preceding paragraph lie in the
top line of $Ker(\wt S^{k-1})$ for some $\wt S^{k-1}$, which itself
lies in $D_l^\cK(S^k) \cup D_r^\cK(S^k)$.
It remains to show that these $y$ lie in the top line of $Ker(S^k)$ itself.
However, Lemma \ref{Lemma 4.2} shows that
\begin{equation}
\text{each of the possible $\wt S^{k-1}$ is a good
$(k-1)$-site of type 2}.
\label{4.70}
\end{equation}
Now, as observed right after the definition \eqref{ker}, $Ker(\wt
S^{k-1})$ equals $\wt S^{k-1}$ if \eqref{4.70} holds. Thus \eqref{4.70}
implies that the possible $y$ lie in the top lines of the possible
$\wt S^{k-1}$ and,as we shall show now, these latter top lines
are contained in the top line
of $Ker(S^k)$. This is so because the projection on the vertical axis of
$\wt S^{k-1}$ is a whole interval of the
partition $\cH^{k-1}$ of $\Bbb Z_+$, and the same is true for the
projection on the vertical axis of $D_\theta^\cK, \theta =l$ or $r$ (see
\eqref{sitek}, \eqref{3.44aaa}, \eqref{3.44abc}, \eqref{3.44bbb}, \eqref{3.44ccc}). Since
$\wt S^{k-1} \subset D_\theta^\cK$, the projections of $\wt S^{k-1}$ and
$D_\theta^\cK$ must be equal (in fact the projections of $D_l^\cK$ and
$D_r^\cK$ are trivially equal; see \eqref{Dkv}).
Thus the top line of $\wt S^{k-1}$ must equal the top line of
$D_\theta^\cK(S^k)$, and this equals the top line of $Ker(S^k)$ by \eqref{3.50abc}.
This completes the proof of the induction step and the lemma.
\end{proof}

\begin{defin}\label{badlayer} For $\cC\subset \Bbb Z_+$ we write
\begin{equation*}
B(\cC)=\{(x,y) \in \widetilde{\Bbb Z}^2_+\colon y \in
\text{span}(\cC)\}.
\end{equation*}
\end{defin}
We use $\cC(m,\ell)$ to denote a generic
cluster of level $\ell$, i.e. an element of $\text{\bf
C}_{\ell,\ell}$,  with mass $m$. The corresponding horizontal layer
$B(\cC(m,\ell))$ is called the {\it bad layer of mass $m$ and level $\l$
associated with $\cC(m,\l)$}. If there is no ambiguity to which cluster
$\cC(m,\l)$ of mass $m$ and level $\l$ we are associating the bad
layer we will use $B(m, \l)$ instead of $B(\cC(m,\l))$.

\medskip

\begin{defin}\label{A2}({\it Matching pair}) Let $B(m,\ell)$
be as above. For any $k, \; \ell-1 \le k \le m-1$, we say that two good
$k$-sites $S^k_{(u,v-1)}$, $\widehat{S}^k_{(u',v' +1)}$ form a
{\it matching pair} with respect to $B(m,\ell)$ if either $u'=u$
or $u'=u\pm 1$, according as $v'-v$ is even or odd.
\end{defin}
\bigskip

At this point we are ready to give a more detailed description of
the inductive step.
%%Jan(Harry) INSERTED NEXT LINE AND LINE ABOUT $N$ HAS BEEN MOVED TO HERE.
The environment will be a fixed $\ga$ with $\chi(\ga) = 0$.
We assume $p_G>2/3$ and $N$ is a fixed integer for which
$[3(1-p_G)]^{N/5-2}\le 1/72$. This implies that
\begin{equation}
8(1-p^3)^{N/5}\le (1-p)^2, \text{ for all } p\ge p_G. \label{N}
\end{equation}

If $k\ge 1$ and $S^k$ is a good $k$-site of type 1 we
will be looking at a very particular way to obtain its
%%$c$-passability, as well as
$s$-passability from a given
$(k-1)$-seed $Q^{(k-1)}$. It will turn out to be enough
for Theorem 1.1 to consider the situation
when $S^k$ is good of type 1 and
%%Dec 18Harry CHANGED DESCRIPTION OF BAD (k-1)-LAYERS ON NEXT LINE
%%its bad $(k-1)$-sites are contained in
the bad $(k-1)$-sites contained in $S^k$ lie in a layer $B(k,\ell)$,
for some $\ell<k$. Span$(\cC(k,\l))$, the projection on the vertical
axis of $B(k,\l)$, equals $[\al^{k-1}_v, \om^{k-1}_v]$ for some $v$ in this
case. The kernel of $S^k$ will be the part of $S^k$ which
lies strictly below the horizontal line $y = \al^{k-1}_v$  and the top
line of the kernel is contained in the line $y = \al^{k-1}_v -1$.
The $(k-1)$-sites with their top line equal to the top line of the
kernel are the $(k-1)$-sites with projection onto the vertical axis
equal to $\cH^{k-1}_{v-1} = [\al^{k-1}_{v-1}, \om^{k-1}_{v-1}]$. These
are therefore of the form $S^{k-1}_{u,v-1}$ for some $u$.

In each case, passability of $S^k$
will be built from the occurrence of three events $W^k_1, W^k_2$
and $W^k_3$ which we define now. Further properties
of these $W^k_i$ will be given in Theorem \ref{indstep} at the end of this
section.

For $k \ge 1$
\begin{equation*}
W^k_1(s)= W^k_1(s,S^k, Q^{(k-1)}_0) =\{ S^k \text{ has $s$-dense
kernel from a seed $Q_0^{(k-1)}$}\},
\end{equation*}
where $Q_0^{(k-1)}$ is a given rooted $(k-1)$-seed
which fulfills condition (i) for $s$-passability of $S^k$.
If $W^k_1(s)$ occurs, then there exists for $\theta = l$ (left) and for $\theta =
r$ (right) in $D^\cK_\theta$ a collection $\cR^{k-1}_\theta$ of
at least $\theta c L/12$
$(k-1)$-sites $S^{k-1}$ in the open cluster of $Q_0^{(k-1)}$.
Each of these is $s$-passable from some rooted $(k-2)$-seed. We
remind the reader that this implies that each of the $S^{k-1}$ in
$\cR^{k-1}_\theta$ contains for $\la = l$ and for $\la = r$ a rooted
$(k-2)$-seed $Q^{(k-2)}_\la = Q_\la^{k-2}(S^{k-1})$ with top line contained
in $D^\cK_\la(S^{k-1})$
for which there exists an open path of 0-sites
in $Ker(S^k)$ from $A(Q_0^{(k-1)})$ to $R(Q^{(k-2)}_\la)$ (see the
proof of Lemma 4.3). The union
of the active sites of $Q_l^{(k-2)}(S^{k-1})$ and
$Q_r^{(k-2)}(S^{k-1})$ is denoted by $A(S^{k-1})$.

For $k \ge 2$ the event $W^k_2(s)$
occurs if and only if $W^k_1(s)$ occurs
and for $\theta=l$ and for $\theta=r$
there exist a collection $\cL_\theta^{k-1}$ of $(k-1)$-sites with the
following properties:
\begin{itemize}
\item{(i)} $\cL_\theta^{k-1} \subset \cR_\theta^{k-1}$ and
the cardinality of $\cL_\theta$ is at least $N$;
\item{(ii)}  for each $S^{k-1}= S^{k-1}_{u,v-1} \in \cL_\theta$ there
exists an index $p$ with $|p - u| \le 2$ and a rooted
$(k-2)$-seed $Q^{(k-2)}(p)$ say, in  $S^{k-1}_{p,v+1}$ and with
top line contained in the $D_l(S^{k-1}_{p,v+1}) \cup
D_r(S^{k-1}_{p,v+1})$
and such that there is an open path inside $S^k$ from
$A(S^{k-1}))$ to $R(Q^{(k-2)}(p))$.
\end{itemize}
When $k=1$ we modify (ii) somewhat because $Q^{(k-2)}$
is meaningless in this case. Recall that a 0-site is just a vertex of
$\wt {\Bbb Z}_+^2$. For $k=1$ $\cR^0_\theta$ will just be taken as the
collection of 0-sites $(u,v-1)$ in $D^\cK_\theta$ for which there
exists an open path in $Ker(S^k)$ from $A(Q_0^{k-1})$ to $(u,v-1)$. We
then replace (ii) by

\noindent (ii, k=1) for each $S^0_{u,v-1} = (u, v-1) \in \cL_\theta$,
there exists a $p$ with $|p-u| \le 2$ such that there is an open path
in $S^1$ from $(u,v-1)$ to $S^0_{p,v+1} = (p, v+1)$.

Finally, if $W_2^k(s)$ occurs, and $k \ge 2$,
then $W_3^k(s)$ occurs if and only if
there exist (at least) two rooted $(k-1)$-seeds in $S^k$,
$Q^{k-1}_{l,1}$ with top line in $D_l(S^k)$ and $Q^{k-1}_{r,1}$ with top
line in $D_r(S^k)$, and open
connections of 0-sites in $S^k$ from the collection of
the rooted $(k-2)$-seeds $Q^{(k-2)}(p)$ mentioned in (ii) above
to $R(Q^{k-1}_{l,1})$ as well as to $R(Q^{k-1}_{r,1})$. When $k=1$
we merely replace the collection of rooted $(k-2)$-seeds
$Q^{(k-2)}(p)$ here by the collection of 0-sites $S^0_{p,v+1}$
mentioned in (ii,k=1).

The definitions of the $W_i^k$ are unfortunately very involved. The
reader should think of $W_1^k$ as providing open connections from the
middle third of the bottom of $S^k$ to the top of its kernel; then
$W_2^k$ will provide open connections from the top of the kernel to
the top of the bad layer, and finally $W^k_3$ from the top of the bad
layer to the top of $S^k$. The connections required for $W_2^k$
from the bottom of the bad layer to its top are the most difficult to
come by. They will be constructed in the next section.
\bigskip

Before formulating our basic set of estimates we state a number of
properties of supercritical oriented site percolation on $\wt {\Bbb Z}^2_+$.
We start with a simple observation which holds for any Bernoulli percolation
as an immediate consequence of coupling.

\begin{lemma}
\label{Lemma 4.4}
Consider site percolation on a graph $\cG$ (possibly
partially oriented). Denote the probability measure under which all
sites are independently open with probability $p$ by $P_p$, and let
$\cE$ be some increasing event. If $p_0, p_0^\prime\in [0,1]$ and
$\tilde p = 1-(1-p_0)(1-p^\prime_0)$, then
\begin{equation}
P_p\{\cE\} \ge 1-  (1-P_{p_0}\{\cE\})(1-P_{p^\prime_0}\{\cE\}) \quad \text {for all } p\ge \tilde p.
\label{4.83}
\end{equation}
\end{lemma}
%\begin{proof}
%Consider two independent independent copies of the
%percolation process, each of which making all sites independently open
%with probability $p_0$. Define a new percolation process by declaring
%a site to be open if the site is open in both independent copies. More
%formally, we can define random variables $X_i(v), i= 1,2, v$ a vertex
%of $\cG$ each of which equals 1(0) with probability $p_0$
%(respectively $1-p_0$). Then we declare a site $v$ to be open in the
%new process if and only if $X^*(v):=\max[X_1(v), X_2(v)]= 1$. Clearly
%this event has
%probability $1-(1-p_0)^2 = p_1$. Moreover, $\cE$ can fail in the
%process defined by the $\{X^*(v)\}$ only if it fails in both the processes
%defined by $\{X_i(v)\}$ for $i=1$  and for $i=2$. Thus, for $p \ge p_1$
%\begin{equation*}
%1-P_p\{\cE\} \le 1- P_{p_1} \{\cE\} \le \big[1-P_{p_0}\{\cE\}\big]^2.
%\end{equation*}
%\end{proof}
\medskip

Now let us go back to oriented site percolation on $\wt {\Bbb Z}^2_+$
and let $P_p$ be as in the preceding lemma. For $\cA \subset
\wt {\Bbb Z}^2_+$ define
\begin{equation}
\Th(\cA) = \{\text{all vertices in $\cA$ are open}\}
\label{4.89vw}
\end{equation}
and let $|\cA|$ denotes the cardinality of $\cA$.
\begin{lemma}
\label{Lemma 4.5}
There exists some $\tilde p \in (0,1)$ and a universal
constant $c_5 < \infty$ such that for all $p \ge \tilde p$ and all subsets $\cA$ of
$2\Bbb Z \times \{0\}$ it holds
\begin{equation}
P_p\{\text{there is an open path from $\cA$ to $\infty
\big|\Th(\cA)$}\}\ge 1-c_5[9(1-p)]^{|\cA|+1}.
\label{4.82}
\end{equation}
\end{lemma}

\begin{proof} The lemma states that the conditional probability,
given $\Th(\cA)$, that percolation occurs from at
least one site in $|\cA|$ is at least $1-c_5(1-p)^{|\cA|+1}$. We shall
only need this if $\cA$ is an interval of $a$ integers, and therefore
we shall prove \eqref{4.82} only in this case. However \cite{Durrett}
(p. 1029) proves that this is the worst case, i.e., that if \eqref{4.82}
holds for $\cA$ an interval, then it holds in general. (\cite{Durrett}
discusses bond percolation, but a small modification of his argument works
for site percolation.)

Now let $\cA = \{0, 2, \dots, 2(a-1)\}\times \{0\}$
and let $\cF$ be the
collection of sites $(x,y) \in \wt {\Bbb Z}^2_+$
for which there exists an open path from $\cA$ to $(x,y)$ (with $(x,y)$
itself also open). Then
\begin{equation}
1-P_p\{\text{there is an open path from $\cA$ to $\infty \big|$ all of
$\cA$ is open}\} \le P_p\Big\{\bigcup_F \{\cF = F\}\Big\},
\label{4.84}
\end{equation}
where the union runs over all finite connected subsets $F$ of
$\wt {\Bbb Z}^2_+$ which contain all of $\cA$.
We bound the right hand side of \eqref{4.84} by the usual contour
method, as we explain now. As in Section 10 of \cite{Durrett} or
\cite{Liggett},  let $D$ be the diamond
$\{(x,y):|x|+|y|\le 1\} \subset \Bbb R^2$. For $F$ a finite connected
subset of $\wt {\Bbb Z}^2_+$ which contains $\cA$,
we define $\wt F = F + D$ and $\Ga(F) =$
the topological boundary of the infinite component of $\Bbb R^2
\setminus \wt F$. Then $\Ga(F)$ is made up of edges of the lattice
$\Bbb Z^2_{odd}:=\{(x,y) \in \Bbb Z^2\: x+y \text{ is odd}\}$ and it
separates $\cA \subset F$ from infinity.
Suppose that ${\cF = F}$ occurs and that $e$ is an edge
between two vertices of $\{(x,y) \in \Bbb Z^2:x+y \text{ is even}\}$
which crosses one of the sides of one of the diamonds $v+D, v \in
F$. In fact we must then have that one endpoint of $e$ equals $v$
and the other endpoint, $w$ say, lies in the
unbounded component of $\Bbb R^2 \setminus \wt F$. There are
then two possibilities. Either
\begin{equation}
w\text{ lies below $v$},
\label{4.85}
\end{equation}
so that a path on $\wt {\Bbb Z}^2_+$ is
prevented from going from $v$ to $w$ by the orientation of
$\wt {\Bbb Z}^2_+$. Or,
\begin{equation}
w \text{ lies above $v$},
\label{4.86}
\end{equation}
in which case $w$ must be
closed (otherwise $w$ would belong to $F$, since an open path to $v$
can be continued by going along $e$ from $v$ to $w$). It follows from
this argument that the event $\cup_F \{\cF = F\}$ is contained in the
event that there exists some contour $\Ga$ made up of sides
of the diamonds $u+D, u \in \wt {\Bbb Z}^2_+$,
which separates $\cA$ from infinity, and which has the following
property: if the edge $\{v,w\}$ crosses one of the sides which make up
$\Ga$ and  $v \in$ interior ($\Ga)$
and $w \in$ exterior $(\Ga)  \cap \wt {\Bbb Z}^2_+$ and
\eqref{4.86} holds, then $w$ is closed. Consequently, the right hand
side of \eqref{4.84} is bounded by
\begin{equation}
\sum_\Ga P_p\{\text{each $w \in \wt{\Bbb Z}^2_+$ as above for which \eqref{4.86}
holds is vacant}\}.
\label{4.87}
\end{equation}

It is shown in
\cite{Liggett} and \cite{Durrett}  that the number of $w$
for which \eqref{4.86} holds is at least $|\Ga|/2$, where $|\Ga|$
denotes the number of edges in $\Ga$. Moreover,
as one traverses the line $\{x= y+ 2i\}$, starting at $(2i,0) \in \cA$
and increasing $x$ (and $y$), the first vertex $w \in \wt {\Bbb
Z}^2_+$ in the unbounded component of $\Bbb R^2\setminus \Ga$
which one meets has to be closed. Since this holds for every $0 \le i
\le a-1$, the number of $w$ for which \eqref{4.86} holds is at least
$a$. In fact, there have to be at least $a+1$ such vertices $w$,
because the first vertex $w$ on the line $x = -y$ which lies in the
unbounded component of $\Bbb R^2 \setminus F$ also satisfies
\eqref{4.86}, but does not lie on any of the lines $x=y+2i$.
It follows that the term in \eqref{4.87} for a specific $\Ga$ is at most
$(1-p)^{(|\Ga|/2) \lor (a+1)}$. Moreover, the number of possible $\Ga$ with
$|\Ga| = n$ is at most $3^{n-1}$, because each possible $\Ga$
which separates $\cA$ from infinity must contain the lower left edge
of the diamond $(0,0)+D$, centered at the origin. It follows that
\eqref{4.87}, and hence also the right hand side of \eqref{4.84} is bounded by
\begin{equation*}
\sum_{n = 1}^\infty 3^{n-1}(1-p)^{(n/2) \lor (a+1)}.
\end{equation*}
The lemma follows.
\end{proof}

Again consider oriented site percolation on $\wt{\Bbb Z}^2_+$.
Write $\bold 0$ for the origin and define
\begin{eqnarray}
r_n := &\sup\{(x,n)\in \wt {\Bbb Z}^2_+:\text{ there exists an open
path from $\bold 0$ to } (x,n)\},\\\nonumber
\l_n := &-\inf\{(x,n)\in \wt {\Bbb Z}^2_+:\text{ there exists an open
path from $\bold 0$ to }(x,n)\},
\label{4.88}
\end{eqnarray}

\begin{equation*}
r_n=\l_n=0 \text{ if there is no open path from $\bold 0$ to }\Bbb Z
\times n.
\end{equation*}
We further remind the reader that the percolation probability $\theta(p)$
was defined in \eqref{4.90a}.
It is known (see \cite{Durrett}) that for all $p \ge \tilde p >p_c$  we have $\theta(p) \ge \theta (\tilde p) >0$
and there exists  $s(p) \in (0, +\infty)$ for which
\begin{eqnarray}
&\lim_{n\to\infty} \frac 1n r_n = \lim_{n\to\infty} - \frac 1n \l_n = s(p) \text{ a.s. $[P_p]$
on the event}\\\nonumber
&
\Om_0 := \text{\{there exists an open path from $\bold 0$ to $\infty$\}};
\label{4.90}
\end{eqnarray}
$s(p)$ is called the edge-speed (see \cite{Durrett} or \cite{Liggett}).

Finally we need the existence of a positive density in $[-ns(p),
ns(p)] \times \{n\}$ of sites which have an open connection from a
fixed finite nonempty set.
The next lemma gives the precise meaning of this statement.
We need the following definition: Let $\al \le \be$ and $\eta > 0$.
Also let $\cA = \{0, 2, 4,\dots 2a-2\}$ be some nonempty
interval of $a$ even integers.
%%Recall that $c$ has been chosen in \equ(4.00).
Then
\begin{eqnarray}
\nu_n(\al,\be) = &\;\nu_n(\al, \be, \cA, \eta ) := \text{ number of points
$(x,n)$ with $\al n \le x \le \be n $},\\\nonumber
& x+n \text{ even,
for which there exists an open path from}\\\nonumber
&\cA \times \{0\} \text{ to $(x,n)$
which stays inside $[-\eta n,\eta n] \times[0,n]$}.
\label{4.88a}
\end{eqnarray}

\begin{lemma}\label{Lemma 4.6} Let $0 <\ep, \eta \le 1$. There exists some $\bar p =
\bar p(\ep, \eta) < 1$ such that for $p \ge \bar p$ there exists an $n_0 =
n_0(\ep, \eta,p)$, such that for  $n \ge n_0$ and $-s(p) \le \al \le
\be \le s(p)$,
\begin{eqnarray}
P_p\big\{\frac 1n \nu_n(\al, \be,\cA, \eta) \ge
[\theta(p)(\be - \al)-\ep]\frac\eta {17}
\text{ for all } -s(p) \le \al \le \be
\le s(p)\big\} \ge 1-\ep.
\label{4.89}
\end{eqnarray}
\end{lemma}

\begin{proof} In general, and in particular in the definitions
\eqref{4.88} and \eqref{4.88a} of $\nu_n$, an open path has to have its
initial point and its endpoint open. For the sake of the proof of the
present lemma we shall call a path open if all its
vertices other than its initial point are open. Until the last three
sentences of the proof we allow its initial
point to be open or closed.

Clearly $\nu_n(\al, \be, \cA, \eta)$ is increasing in
$\cA$, so that it suffices to prove \eqref{4.89} for $\cA=$ the origin.
We shall restrict ourselves to $p \ge \tilde p$ as in Lemma \ref{Lemma 4.5}.
By obvious monotonicity we
then have $\theta(p) \ge \theta(\tilde p) > 0$.  In addition it is immediate
from the definition \eqref{4.90} that $s(p) \le 1$. In fact
\begin{equation}
r_n \le n \text{ and $\l_n \le n$ for all }n.
\label{4.90def}
\end{equation}
Thus, it holds
\begin{equation}
\theta(p) \ge \theta(\tilde p) > 0 \text{ and } 0 < s(\tilde p) \le s(p) \le 1
\label{4.90abc}
\end{equation}
for the $p$ which we are considering.

Now let $\ep > 0$ and $\eta > 0$ be given.
%%Without loss of generality we take $\eta \le s(p_2) \le s(p)$,
%%since $\nu_n$ is increasing in $\eta$.
We define
\begin{equation}
m = m(n, \eta) = \Big \lfloor \frac \eta 8 n \Big \rfloor,
k_0= k_0(\eta) = \Big \lfloor \frac nm \Big \rfloor -1, m' = n-k_0m.
\label{4.89c}
\end{equation}
Finally, we choose $\ep_1$ such that
\begin{equation}
0 < \ep_1 \le \frac \ep {2k_0}.
\label{4.89b}
\end{equation}
and then
$\bar p=\bar p(\ep, \eta) < 1$ so that $\bar p \ge \tilde p \lor(1-\ep/2) $ and
\begin{equation}
\theta(\bar p) := P_{\bar p}\{\Om_0\} \ge 1-\ep_1.
\label{4.90z}
\end{equation}
Such a $\bar p  < 1$ exists by \eqref{4.82}.

First we observe that \eqref{4.90} implies
that for every $p \ge \bar p$ and $\eta_1 > 0$ there exists a constant
$c_6=c_6(\ep_1, \eta_1,p)$ such that
\begin{eqnarray}
&P_p\big\{\big(|r_t - ts(p)| >  c_6 + \eta_1 t  \text{ or }
|\l_t + ts(p)| >  c_6 + \eta_1 t \text{ for some  }t \in \Bbb
Z_+\big)\big\} \\\nonumber
&\le P_p\big\{\big(|r_t - ts(p)| >  c_6 + \eta_1 t  \text{ or }
|\l_t + ts(p)| >  c_6 + \eta_1 t \text{ for some  }t \in \Bbb
Z_+\big) \cap \Om_0\big\}  +P_p\{[\Om_0]^c\}\\\nonumber
& \le 2\ep_1.
\label{4.90b}
\end{eqnarray}

We observe next that if $\Om_0$ occurs, then
$r_m$ and $\l_m$ are well defined for all $t$. Furthermore, for any $m$
there must exist open paths $\pi_\l =\pi_\l(m)$ and $\pi_r=\pi_r(m)$
from the origin to $(\l_m,m)$ and to $(r_m,m)$, respectively, and these
paths must lie in $[-m,m] \times [0,m]$ (see \eqref{4.90def}).
%%%Assume then that \equ(4.93) holds. By \equ(4.91) this event has
%%%conditional probability given $\Th(A)$, of at least $1-3\ep/4$.
Next let $x \in \Bbb Z$ with $x+m$ even be
such that $-\l_m \le x \le r_m$. Consider
the open paths starting at $(x,m)$ going {\it downwards}, that is
against the orientation on $\wt {\Bbb Z}^2_+$ assumed so far. Assume
that for a given $x \in [-\l_m,r_m]$ there exists an infinite downward
open path, $\wt \pi_x$ say, starting at $(x,m)$. Since this path starts
between $(-\l_m,m)$ and $(r_m,m)$, it must hit $\pi_l \cup \pi_r$.
Furthermore, the path $\wt \pi_x$ necessarily stays in $[x-m,x+m]
\times [0,m]$ up till time $m$.
For the sake of argument, let $\wt \pi_x$ first intersect $\pi_\l$ in a
point $(y,q)$ with $0 \le q \le m$. Then the piece of $\pi_\l$ from
the origin to $(y,q)$, followed by the piece of $\wt
\pi_x$, traversed in the forward direction, from $(y,q)$ to $(x,m)$
forms an open oriented path from the  origin to $(x,m)$. A similar
argument applies if $\wt \pi_x$ hits $\pi_r$. Thus, if there
exists a downward infinite open path from $(x,m)$, then there exists
an open path from the origin to $(x,m)$. By the estimates
on the locations of $\pi_\l, \pi_r$ and of $\wt \pi_x$ which we have
just given, this path must be contained in $[-2m,2m]\times [0,m]$.

Let us write $J_x$ for the indicator function of the event that there
is an open path contained in $[-2m,2m]\times [0,m]$
from the origin to $(x,m)$. Also, let $\wt I_x$ and $I_x$ be the indicator
functions of the events that there exists an infinite open backwards
path from $(x,m)$, respectively an infinite open forwards path from
$(x,0)$, which stays in $[-2m,2m]$ during $[0,m]$. The
preceding argument shows that for any $M \ge 0$, on the event
\begin{equation}
\cE(M,m) := \cap \{-\l_m \le -2M \le 2M \le r_m\},
\label{4.93a}
\end{equation}
it holds
\begin{eqnarray}
&\wt \nu_m(M, \bold 0): =
\text{ number of points
$(x,m)$ with }x \in [-2M,2M],\\\nonumber
& x+m \text{ even,
for which there exists an open path from } \bold 0\\\nonumber
& \text{to $(x,m)$
which stays inside $[-2m,2m] \times[0,m]$}\\\nonumber
&\ge \sum_{\substack {-2M \le x \le 2M\\x+m \text{ even}}}
J_x \ge  \sum_{\substack{-2M \le x \le 2M\\x+m \text{ even}}} \wt I_x.
\label{4.93f}
\end{eqnarray}
Now, the monotonicity of $s(\cdot)$ and \eqref{4.90abc} and \eqref{4.90}
imply that for each fixed $M >0$ there exists an $m_1=m_1(\ep_1,M)$
such that for $m \ge m_1$ and $p \ge \bar p$
\begin{equation}
P_p\{\cE(M,m) \text{ fails}\} \le P_p\{\cE(M,m) \text{ fails}\} \le
\ep_1.
\label{4.90d}
\end{equation}
Further, the joint distribution of the $\wt I_x,\; x +m$ even is
the same as the joint distribution of the $I_x,\; x$ even.
Also, if
\begin{equation}
 M+c_6 \le \frac m2 \text{ and }|x| \le \frac m2,
\label{4.90e}
\end{equation}
then
\begin{eqnarray*}
&I_x \ge K_x:= I[\text{there exists an open path from $(x,0)$ to
infinity which}\\
&\text{stays in $[x-M-c_6-t, x+M+c_6+t]$
for all $t \ge 0$}.
\end{eqnarray*}

By the ergodic theorem (see \eqref{4.90a} for $\theta$)
\begin{eqnarray}
\liminf_{M \to \infty} \frac 1M \sum_{\substack {x \in [-2M,2M],\\ x
\text{ even}}}I_x \ge \lim_{M \to \infty}
\frac 1M \sum_{\substack {x \in [-2M,2M],\\ x
\text{ even}}}K_x \ge 2\theta(p) \text{ a.s.}\; [P_p].
\label{4.90g}
\end{eqnarray}
Thus there exists an $M_0=M_0 (\ep_1)$ such that for all $p \in [\bar p,1)$
\begin{eqnarray}
&P_p\left\{\wt I_x = 1\text{ for some }x \in [-2M_0,2M_0]\right\}\\
&\ge P_{\bar p}\left\{\frac 1{M_0}  \sum_{\substack{ x \in [-2M_0,2M_0] \\ x+m
\text{ even}}}\wt I_x \ge (2-\ep_1)\theta\right\}
 \ge 1 -\ep_1.
\label{4.92}
\end{eqnarray}

We take $m_2 =m_2(\ep_1)$ such that $ m_2/2 \ge 2M_0+c_6$. Then
\eqref{4.90e} with $M_0$ for $M$ holds true for any
$|x| \le 2M_0, m \ge m_2$.

We now apply \eqref{4.93f}, \eqref{4.90d} and \eqref{4.92} to obtain for all
$m \ge m_2, p \ge p_3$
\begin{eqnarray*}
&P_p\{\text{there is at least one $x \in [-2M_0, 2M_0]$ with an open
path from}\\
&\bold 0 \text{ to $(x,m)$ which is contained in
$[-2m,2m] \times [0,m]$}\}\\
&\ge 1-2\ep_1.
\end{eqnarray*}
In other words, if we first determine the state of all
vertices $(x,y)$ with $0 \le y \le m$, we will find with
probability $1-2\ep_1$ at least one vertex $(x_1,m) \in [-2M_0,2M_0]$
with $x_1+m$ even and with an open connection from $\bold 0$ to $(x_1,m)$
which stays in $[-2m, 2m] \times [0,m]$. On the event that such
an $x$ exists, let $x_1$ be the
smallest $x$ in $[-2M_0,2M_0]$ with these properties. We can then
repeat the argument (after a shift by $(x_1,m))$, to find that with a
further conditional probability of at least $1-2\ep_1$, there exists
an $x_2 \in [x_1- 2M_0, x_1+2M_0] \subset [-4M_0,4M_0]$
with an open path from $x_1$ to
$x_2$ which stays in $[x_1-2m, x_1+2m] \subset [-4m,4m]$ during
$[m,2m]$. Concatenation of the open path from $\bold
0$ to $x_1$ and the path from $x_1$ to $x_2$ gives an open path from
$\bold 0$ to $(x_2, 2m)$ which stays in $[-4m,4m]$ during $[0,2m]$.
Similarly, we find by repeating the argument $k_0$ times that there is a
probability of at least $(1-2\ep_1)^{k_0}$ that $\bold 0$ is connected by an
open path which stays in $[-2k_0m,2k_0m] \times [0,k_0m]$ to a
vertex $(x_{k_0},k_0m)$ with $|x_{k_0}| \le 2k_0M_0$.

We need to concatenate paths once more. This time we replace $m$ by
$m' \in  [m,2m]$ (see \eqref{4.89c}) and the sum over $-2M_0 \le x\le
2M_0$ in \eqref{4.90g} by the sum over $\al m' \le x \le \be m'$ for
some fixed $-s(p) \le \al \le \be \le s(p)$. In essentially the
same way as before we conclude that for $p \ge \bar p$ and $m \ge m_3=m_3(\ep_1,p)$ (suitable)
\begin{eqnarray*}
&P_p \big\{\text{for all $-s(p) \le \al \le \be \le s(p)$
there are at least }[\theta(p)(\be -\al)/2 - \ep_1]m' \\
&\text{values of $x$ with $\al m' \le x  \le \be m',
x+m'$ even, for which there}\\
&\text{is an open path from $\bold 0$ to $(x,m')$
which stays inside}\\
& [-2m',2m'] \times [0,m']
\text{ during $[0,m']$}\big\}\\
&\ge 1-2\ep_1.
\end{eqnarray*}
If $x_{k_0}$ as described above exists, then there is a
conditional probability,
given the state of all vertices $(x,y) \in \wt{\Bbb Z}^2_+$ with $y
\le k_0m$, of at least $(1-2\ep_1)$ that $(x_{k_0},k_0m)$ is connected
to at least
\begin{equation*}
[\theta(\be-\al)/2-\ep_1]m' \ge [\theta(\be -\al)/2 -\ep_1]m \ge [\theta(\be -\al)/2
-\ep_1]\Big \lfloor \frac \eta 8 n \Big \rfloor
\end{equation*}
vertices $(x',km+m') = (x',n)$ in $[\al m'-2
k_0m, \be m'+2k_0m] \times n$ by open paths which stay in
\begin{equation*}
[-2k_0M_0-2m', 2k_0M_0+2m'] \times [0,n] \text{ during }
[0,k_0m+m'] = [0,n].
\end{equation*}
But by \eqref{4.89c} there exists some $n_0 = n_0(\ep,\eta)$
 such that for $n \ge n_0$ it holds $m \ge m_1 \lor
m_2 \lor m_3$ and
\begin{equation*}
2k_0M_0+2m' \le 2k_0M_0+4m \le 2\frac{n}{m} M_0 + 4\frac{\eta}{8} n \le
\eta n,
\end{equation*}
so that the constructed paths stay in $[-\eta n, \eta n] \times
[0,n]$, as is required for them to be counted in $\nu_n$.
Also, by our choice of $\ep_1$ in \eqref{4.89b}
\begin{equation*}
(1-2\ep_1)^{k_0+1} \ge 1 -2(k_0+1)\ep_1 \ge 1-\ep/2.
\end{equation*}
We had to concatenate $k_0+1$ paths, each of which existed with
a conditional probability of at least $1-2\ep_1$,
given the previously chosen paths. Thus the whole construction works
with a probability of at least $(1-\ep_1)^{k_0+1} \ge 1-\ep/2$.
This proves \eqref{4.89} when
$\cA = \bold 0$. As pointed out before this proves the lemma if we do
not insist that the starting point of an open path is open. However,
if we revert to our previous convention that an open path must have an
open initial and  final point, then our construction of open paths
from $\bold 0$ to the horizontal line $\{y=n\}$ is valid only on the
event $\{\bold 0 \text{ is open}\}$. We therefore have to discard the
event $\{\bold 0 \text{ is closed}\}$. Correspondingly, the probability
of finding the required open paths is at least $(1-\ep_1)^{k_0+1} -
(1-p) \ge 1-\ep/2 -\ep/2 = 1-\ep$ (recall that $p \ge \bar p\ge
1-\ep/2$; see the line before \eqref{4.90z}).
\end{proof}

\begin{coro}\label{cor1-sec4} Let $\Th(\cA)$ be as in \eqref{4.89vw} and
\begin{equation*}
\Om(\cA):=\{\text{there is an open path from $\cA$ to infinity}\}.
\end{equation*}
Then under the conditions of Lemma \ref{Lemma 4.6}, if n$\ge n_0$, it holds for any finite
set $\cA \subset \wt {\Bbb Z}^2_+$
\begin{eqnarray}
P_p\big\{\frac 1n \nu_n(\al, \be,\cA, \eta) \ge
[\theta(p) (\be - \al)-\ep] \frac \eta {17}
\text{ for all } -s(p) \le \al \le \be
&\le s(p)\big|\Th(\cA)\big\}\\\nonumber
& \ge 1-\ep.
\label{4.89yz}
\end{eqnarray}
and
\begin{eqnarray}
P_p\big\{\frac 1n \nu_n(\al, \be,\cA, \eta) \ge
[\theta(p)(\be - \al) -\ep]\frac \eta {17}
\text{ for all } -s(p) \le \al \le \be
&\le s(p)\big|\Om(\cA)\big\}\\\nonumber
& \ge 1-\ep.
\label{4.89xy}
\end{eqnarray}
\end{coro}
\begin{proof} Since $\Th(\cA)$ and $\Om(\cA)$ are increasing events of
the environment, these inequalities are immediate from \eqref{4.89} and the
Harris-FKG inequality.
\end{proof}

The $W^k_i$ in the next theorem were defined in the paragraph
following \eqref{N}. In this theorem the environment is fixed at
$\ga$ and the estimates are uniform in $\ga$. The probability in this
theorem refers only to the choice of the occupation variables, once the
nature (good or bad) of each site has been fixed.
$\chi(\ga)$ is defined
in \eqref{2.25z}.

\begin{thm}
\label{indstep}
Let $p_B>0$. Then there exist some $p^* = p^*(p_B) < 1$ and $L_1 = L_1(p_B,p_G) < \infty$
such that for $p_G \ge p^*$ and $L \ge L_1$ and for every $\ga$ with
$\chi(\ga)=0$, every $k \ge 1$ and good $k$-site of type 1 which
intersects exactly one bad layer $B(k,\l)$ (for some $\l \le k-1$),
the following bounds hold:

\noindent (a)\; If $Q^{(k-1)}$ is a rooted $(k-1)$-seed and $S^k$ a good of
type 1 $k$-site
which satisfy condition (i) for an $s$-passable $k$-site, then
\begin{equation}
P\{W^k_1(s)\mid Q^{(k-1)} \text{ is a rooted $(k-1)$-seed }\} \geq 1-
\frac{(1-p_G)^{k+1}}{4}. \label{5.h01}
\end{equation}

\noindent (b)\;
\begin{equation}
P\{W^k_2|W^k_1\} \ge 1- \frac{(1-p_G)^{k+1}}{4}.\label{5.h03}
\end{equation}
\noindent (c)\;
\begin{equation}
P\{W^k_3|W^k_2\} \ge 1-\frac{(1-p_G)^{k+1}}{4}. \label{5.h031}
\end{equation}
\noindent (d)\;
\begin{eqnarray}
\nonumber &&P\{S^k \text{ is } s\text{-passable from }Q^{(k-1)}\big|Q^{(k-1)} \text{is
a rooted }(k-1)-\text{seed}\} \geq 1- (1-p_G)^{k+1},\\
&&P\{ S^k \text{ is }c\text{-passable}\} \geq 1- (1-p_G)^{k+1}.
\label{5.h02}
\end{eqnarray}
\end{thm}
\medskip

\noindent {\it Proof.} The proof goes by induction. For simplicity
we write $(a_k),(b_k),(c_k)$ and $(d_k)$ for the corresponding
statement at level $k$. After proving $(a_1), (b_1)$ and $(c_1)$ we
shall prove here the following implications:
\begin{equation}
(a_k), (b_k) \text{ and $(c_k)$ together }\Longrightarrow (d_k);
\label{4.80}
\end{equation}
\begin{equation}
(d_k) \Longrightarrow (a_{k+1}) \text{ and } (c_{k+1}).
\label{4.81}
\end{equation}
To complete the proof we shall show in the next section that $(a_j)$
for $j \le k+1$ and $(b_j), (c_j)$ for $j \le k$ imply $(b_{k+1})$.

We first observe that \eqref{4.80} is immediate from the definitions.

Now start with $k = 1$. For $(a_1)$, assume for the sake of argument that
$S^1 = S^1_{0,v} =  \big((-cL/2, cL/2]$
$\times \cH^1_v\big) \cap \wt {\Bbb Z}^2_+$
(see \eqref{sitek}) with kernel $S^1_{0,v} \cap(\Bbb Z \times
\cK^1_v)$ (see \eqref{ker}). $\cK^1$ is some interval, $[y_0,y_1]$ say.

Since $S^1$ is assumed to be good of type 1,
it contains a unique cluster of mass $\ge 1$ and
this cluster has to be a single bad line. This cluster will be $\cC_{1,j}$
for some $j$ in the notation of Section \ref{sec2constr}. Since $Ker(S_1)$
is the part of $S^1$ below the bad line which intersects $S^1$, this
bad line is the line $y= y_1+1$. The bad lines before that will
be $\cC_{1, j-1}$ (or there will not be a previous bad line if $j = 1$).
In this situation $y_1-y_0$ is $\al(\cC_{1,j})
-1 -\om(\cC_{1,j-1}) \ge L-4$ (see \eqref{2.two} and \eqref{3.1}; if $j=1$ we
have to use the assumption $\chi(\ga)=0$ instead of \eqref{2.two}. In
any case, $y_1-y_0 \ge L/4$ if $L$ is taken
$\ge 6$, and we can take $L$ as large as desired if we take $\de>0$
small enough (see Lemma \ref{lemma1.1}). In the other direction,
$S^1$ is a good site of level 1, so $y_1-y_0 \le 2L$ by virtue of \eqref{3.50}.

We also point out that all horizontal lines $\Bbb Z
\times \{y\}$ with $y_0 \le y \le y_1$ lie in $Ker(S^1)$ and are good lines.

The condition that $Q^{(0)}$ is a rooted 0-seed gives us two adjacent
open vertices $(x,y_0-1)$ and $(x+2,y_0-1)$ for some odd $x \in [-cL/6 -1, cL/6+1]$.
For $W^1(s, S^1, Q^{(0)})$ to hold it suffices to have open paths in
$Ker(S^1)$ from $(x,y_0-1) \cup (x+2,y_0-1)$ to at least $\rho c L/12$ 0-sites in
$D_\theta^\cK$ for $\theta = l$ and $\theta = r$. In the simple case of $k=1$
these are just open paths in $S^1$ to $ [-\frac 5{12} cL, -\frac 13
cL] \times  \{y_1\}$ (if $\theta =l$) and to $[\frac 13 cL,\frac 5{12}cL] \times
\{y_1\}$. \eqref{5.h01} for $k=1$ can now be satisfied for
large $L$ by an application of Lemma \ref{Lemma 4.6} that guarantees positive density
of the open oriented cluster restricted to $S^1$ at a suitable hight proportional to $L$,
and then using unrestricted growth to achieve density $\rho$ in
$D_\theta^\cK$ for $\theta = l$ and $\theta = r$.

%\maltese (\eqref{4.00} WILL NEED CHANGE
%Indeed, take   $n= y_1-y_0+1 \in [L/4, 2L+1]$. Further take
%$\ep = (1-p_G)^2/8$ and $\eta$ and $c$
%such that $5c/6 \ge 2\eta$ (this will guarantee that paths with
%diameter $\le 2\eta n$ have diameter $\le 5c L/6$).
%Also, take $\al = cL/(3n)$ and $\be
%= 5cL/(12n)$ and, if necessary, reduce $c$ to make $\rho c/(12)
%< [\theta(p_G)(\be -\al)] \eta/(68)$.
%Finally, to guarantee $|\al|\le |\be| \le s(p)$ we take $c$ such that
%$5c/3 \le s(p_G)$. With these choices, Lemma \ref{ }
%shows that for $L \ge L_2$ for some $L_2(p_G, \eta)$
%there is a probability of at least $1-\ep$ that there exist
%open paths in $[-\eta n, \eta n] \times [y_0,y_n] \subset Ker(S^1)$
%from $(x,-1)$ to $[\theta(p)(\be -\al)\eta n/(17)$ sites in
%$D_r^\cK$. The same estimate with $D_r^\cK$ replaced by $D_l^\cK$ also holds.
%This proves \eqref{5.h01}.

Next, \eqref{5.h03} for $k=1$ is easy. If $W^1_1(s)$ occurs, then for
$\theta = l,r$ there exist sets $\cR^0_\theta$ containing at least
$\rho cL/(12)$ open 0-sites which have an open connection from the
origin. These sets are contained in the top line of $Ker(S^1)$, that
is, in the horizontal line $\{y = y_1-1\}$. It is important that these sets
$\cR^0_\theta$ are determined by the occupation variables $\eta_{(a,b)}$
with $b \le y_1-1$. For $W^1_2(s)$ to occur, there should be at least
$N$ (see \eqref{N}) sites $(a,y_1)$ in each of $\cR^0_\theta, \theta = l$
or $r$ which have an open
connection to $(x,y_1+2)$ (which is the line just above the
bad line $\{y=y_1+1\}$) for some $x \in [a-2,a+2]$. But the
cardinality of $\cR^0_\theta$ is at least $\rho c L/12$ and hence goes to
infinity with $L$. Thus if we keep $0< p_B,p_G$ and $N$ fixed, then
the conditional probability in the left hand side of \eqref{5.h03} tends
to 1 as $L \to \infty$. Indeed, given $W^0_1(s)$, the event that
$(a,y_1)$ has an open connection to $[a-2, a+2] \times \{y_1+2\}$
has a strictly positive conditional probability, and these events for
$a=a'$ and $a=a''$ are conditionally independent when $|a'-a''| \ge 5$.
Thus, by raising $L_1$ if necessary, \eqref{5.h03} follows.

We turn to \eqref{5.h031}. If $W_2^1(s)$ occurs, then let
$\cL_\theta^0$ be the subset of 0-sites $(a,y_1)$ in $\cR^0_\theta$ which
have an open connection to $(x,y_1+2)$ for some $x \in [a-2,a+2]$. On
the event $W_2^1(s)$ the cardinality of $\cL_\theta^0$ is at least
$N$. Denote the collection of 0-sites $(x, y_1+2)$ with an open
connection from some $(a,y_1) \in \cL^0_\theta$
as just mentioned $\wt {\cL}^0_\theta$. On the event $W^1_2(s)$, the
cardinality of $\wt {\cL}^0_\theta$ is at least $N/5$.
Then $W^1_3(s)$ occurs if for $\theta =l$ as well as $\theta=r$,
there is a $(x,y_1+2) \in \wt {cL}^0_\theta$ which has an open
connection to a rooted 0-seed with top line in the top line of
$S^1$. Note that the top line of $S^1$ is contained in the line $\{y =
y_1+4\}$, by virtue of \eqref{3.2}. Therefore, if $(x,y_1+2) \in \wt
{\cL}^0_\theta$, then the conditional probability that it has
such an open connection is bounded below by $p_G^3$. Since the cardinality
of $\wt{\cL}^0_\theta$ is at least $N/5$ one easily sees that
$P\{W^1_2(s)\mid W^1_1(s)\}\to 1$ as $N \to \infty$. In fact, \eqref{N}
suffices to guarantee \eqref{5.h031}. Thus if we first pick $N$ so that
\eqref{N} holds and then $L_1$ so that \eqref{5.h03} holds then both
\eqref{5.h03} and \eqref{5.h031} hold.

If $p_G>2/3$ and  $N$ has been chosen as above, we
see at once from \eqref{N} that $(d_k)$ implies $(c_{k+1})$. A coupling argument (Lemma \ref{Lemma 4.4})
easily shows that $(a_{k+1})$ follows from $(d_k)$.  It remains to show that having
$(a_j), j\le k+1$ and $(b_j), (c_j), j\le k$ we get $(b_{k+1})$.
This is the core of the proof and we postpone it to Sect. \ref{conclusion}. It
requires a more detailed study of clusters introduced in Sect.\ref{sec2constr},
and which is the object of the next section.

%%%%%%%%%%%%%%%%%%%%%%%%%%%%%%%%%%%%%%%%%%%%%%%%%%%%%%%%%%%%%%%%%%%%%%%%%%%%%%%%%%
%%%%%%%%%%%%%%%%%%%%%%%%%%%%%%%%%%%%%%%%%%%%%%%%%%%%%%%%%%%%%%%%%%%%%%%%%%%%%%%%%%

\section{Towards drilling. Structure of clusters}
\label{structure}

\begin{lemma}
\label{level}
If $\cC$ is a cluster of mass $m$
and level $\ell$, then it has at most $m-\ell+1$ constituents.
\end{lemma}

\begin{proof}  Assume that $\cC$ is formed from an
$\ell$-run of $r$ constituents: $\cC_1, \dots, \cC_r$, $r \geq
2 $, each $\cC_i$ being of level $\ell_i < \ell$, and mass $m_i$.
From the definition of  $\ell$-run we see that  $m_i \ge
\ell , \; i=1, \dots, r $.  On the other hand using the definition
of the mass of a cluster (first equality in \eqref{2.foura}) we see that
$m \ge m_1 + (r-1)$. The statement follows at once.
\end{proof}
\medskip

\noindent {\bf Notation.} Given   $\Ga (\om) = \{x \in \Bbb Z_+
\colon \xi_x = 1 \}$ and an interval $[a,b]$ we denote by
$\Gamma_{[a,b]}$ a new configuration on $\Bbb Z_+$:
$\Gamma_{[a,b]}\equiv \Gamma_{[a,b]} (\omega)= \Gamma (\omega)
\cap [a,b]$. Equivalently, $\xi_{[a,b]}(x)=\xi_x$ if $x \in [a,b]$,
and is zero otherwise.

\medskip

\begin{defin}
\label{pm} {\em (Porous medium)} We say that
the segment $[x_1, x_2]$ is porous medium of level $k$  (with
respect to $\Gamma$) if:

\noindent 1) ${\text{\bf C}}_{\infty}(\Gamma_{[x_1, x_2]})$ contains
no clusters of mass strictly larger than $k$;

\noindent 2) for any ${\cC} \in {\text{\bf C}}_{\infty}
(\Gamma_{[x_1, x_2]} )$ we have:
\begin{equation*}
d(\cC, x_1) \ge L^{m(\cC)}-1 \quad {\text{and}} \quad d(\cC, x_2) \ge L^{m(\cC)}-1.
\end{equation*}
In particular, $x_1, x_2 \not \in \Gamma$. When $k=0$ the definition reduces to
$\Gamma \cap [x_1,x_2]=\emptyset$.
\end{defin}

\begin{lemma}
\label{poroso}
\noindent a) If $\ell \ge 1$ and ${\cC_i}, \, {\cC_j} \in
{\text{\bf C}}_{\ell-1}(\Gamma)$ are two consecutive constituents
of an $\ell$-run, then the interval $[\omega ({\cC_i})+1,
\alpha ({\cC_j}) -1]$ is a porous medium of level $(\ell-1)$
with respect to $\Gamma$.

\noindent b) If $k \ge 1$ and ${\cC}_i$, ${\cC}_j \in {\text{\bf
C}}_{\infty}(\Gamma)$ are two consecutive clusters of mass at
least $k$, then the interval $[\omega ({\cC}_i)+1, \alpha
({\cC}_j) -1]$ is a porous medium of level $(k-1)$ with respect
to $\Gamma$.
\end{lemma}

\begin{proof}It follows at once from the construction of
${\text{\bf C}}_{\ell-1}$ and ${\text{\bf C}}_{\infty}$.
\end{proof}

\begin{lemma}
\label{decomp} {\em (Descending decomposition)} Each cluster ${\cC}
\in \cup_\ell{\text{\bf C}}_{\ell,\ell} (\Gamma)$ of mass $m\ge 2$
has the following representation: there exists an increasing
sequence of integers
\begin{equation*}
\alpha ({\cC}) = f_1 < g_1 <f_2 < g_2 < \dots < f_v < g_v  \le
\omega ({\cC})-1,
\end{equation*}
so that for each $1 \le s \le v$, the partition ${\text{\bf
C}}_{\infty}(\Gamma_{[f_s, g_s]})$ consists of a unique cluster,
denoted by $\widetilde {\cC}_s$, with $\alpha (\widetilde {\cC}_s) =
f_s$ and $\omega (\widetilde {\cC}_s)= g_s$, and the following
holds:

\noindent {\rm 1)}
$m(\widetilde {\cC}_1) = m-1$, $m(\widetilde {\cC}_s)= \widetilde m_s$ for $2 \le s \le v$,
where $m-1\equiv\widetilde m_1>\widetilde m_2>\dots>\widetilde m_v$.

\noindent {\rm 2)} the intervals $[g_{s-1} +1, f_{s}-1]$ are
porous media of level $\widetilde m_{s}$ with respect to $\Gamma$,
$2 \le s \le v$, and
\begin{eqnarray}
&L^{\widetilde m_s}\le f_{s}-g_{s-1} \le L^{\widetilde m_s+1};\\
&\omega(\cC)-L <g_v, \quad  [g_v+1,\omega(\cC)-1]\cap
\Gamma=\emptyset.
\label{dec}
\end{eqnarray}
\end{lemma}

\medskip
\begin{proof} Observe that the statement is obvious for
clusters of level $1$ and mass $m \ge 2$, in which case $v=1$. We therefore consider clusters of level at least 2. The proof
uses induction on the mass. Assuming the statement to be true for every cluster of
mass at most $m$ we prove it for every cluster of mass $m+1$. Fix
${\cC}\in \cup_\ell{\text{\bf C}}_{\ell, \ell}(\Gamma)$, such that $m({\cC})=m+1$.  We split the proof in two sub-cases.

\noindent {\it Case  $\ell\equiv\ell({\cC})=m$.} In this case it
follows from Lemma \ref{level}  that ${\cC}$ is formed as an
$m$-run of only two constituents, ${\cC}_1$ and ${\cC}_2$,
with $m({\cC}_1)=m({\cC}_2)=m$, and we take $f_1=\alpha ({\cC})$ and $g_1=\omega ({\cC}_1)$.  By Lemma \ref{poroso} we
have that $[\omega ({\cC}_1)+1, \alpha({\cC}_2)-1]$ is
porous media of level $m-1$ and from the definition of the run we
have that $L^{m-1}\le \alpha({\cC}_2) - \omega ({\cC}_1) <
L^{m}$. On the other hand from the induction assumption we know
that there exists a sequence of integers
$$
\alpha ({\cC}_2) = f_1^\prime < g_1^\prime <f_2^\prime <
g_2^\prime < \dots < f^\prime_{v^\prime} < g_{v^\prime}^\prime \le
\omega ({\cC}_2)-1,
$$
such that for each $1 \le s \le v^\prime$ the partition
${\text{\bf C}}_{\infty}(\Gamma_{[f_s^\prime, g_s^\prime]})$
consists of unique cluster, denoted by $\widetilde {\cC}_s^\prime$ with $\alpha (\widetilde {\cC}_s^\prime) = f_s$
and $\omega (\widetilde {\cC}_s^\prime)= g_s$, with
$$
m(\widetilde {\cC}_1^\prime) = m-1,\quad  {\text {and}} \quad
m(\widetilde {\cC}_s^\prime)= \widetilde m_s^\prime, \quad 2
\le s \le v^\prime,
$$
and the intervals $[g_{s-1} +1, f_{s}-1]$ are porous media with
respect to $\Gamma$ of level $\widetilde m_{s}^\prime, \; 2 \le s
\le v^\prime$, and
\begin{equation}
L^{\widetilde m_s^\prime} \le f_{s}-g_{s-1} \le L^{\widetilde
m_s^\prime+1}. \label{decprime}
\end{equation}
Taking $f_s = f^\prime_{s-1}$ and $g_s =
g^\prime_{s-1}, \; 2 \le s \le v^\prime$, we get the desired
representation of ${\cC}$.
\medskip

\noindent {\it Case $ 2\le \ell\equiv \ell({\cC}) < m$.} In
this case ${\cC}$ is formed as an $\ell$-run of $r$
constituents ${\cC}_1, \, \dots ,{\cC}_r$, $2\le r \le
 m- \ell +2$, with $m ({\cC}_i) \geq \ell, \; 1\le
i \le r$, and so  ${\text{\bf C}}_{\infty} (\Gamma_{[\alpha({\cC}_1), \omega({\cC}_{r-1})]}) $ consists of a unique cluster
which we denote by $\widehat {\cC}$.

If $m( \widehat {\cC} \,)=m$, we see from \eqref{2.foura} that $m(\cC_r)=\ell <
m$. In this case we set $ f_1 = \omega({\cC}_{r-1})$, and using
the inductive assumption for $\cC_r$, we complete the
representation as in the previous case.

If $m(\widehat {\cC}\,)<m$, we have that $\ell + 1 \le m(\cC_r) = m - m(\widehat {\cC}\,)+ \ell \le m $. By the inductive
assumption applied to ${\cC}_r$ as the unique element of
${\text{\bf C}}_{\infty}(\Gamma_{[\alpha({\cC}_r), \omega({\cC}_r) ]})$ there are integers
$$
\alpha ({\cC}_r) = \widetilde f_1 < \widetilde g_1 <\widetilde
f_2 < \widetilde g_2 < \dots < \widetilde f_{\tilde v} <
\widetilde g_{\tilde v} \le \omega ({\cC}_r)-1
$$
for which properties $1)-2)$ of the lemma hold. Moreover, the
unique cluster $\widetilde {\cC}_1$ of ${\text{\bf C}}_{\infty}
(\Gamma_{[\widetilde f_1, \widetilde g_1]})$ has mass $ m - m(\widehat
{\cC}\,)+ \ell -1 \geq \ell$. In the configuration
$\Gamma_{[\alpha ({\cC}_1), \widetilde g_1]}$ the clusters ${\cC}_1, \dots , {\cC}_{r-1}$ and $\widetilde {\cC}_1$  will
form an $\ell$-run, producing a cluster of mass $m$. Therefore,
taking $f_s = \widetilde f_{s},\,  s \ge 2$ and $g_s = \widetilde
g_{s}, \, 1 \le s \le \tilde v$, we get the desired representation
of ${\cC}$.
\end{proof}

\medskip

\begin{defin}
\label{itnr1} {\em (Itinerary of a bad layer)} In the notation of
the previous lemma, the sequence $\{\widetilde  m_s \}_{s=1}^v$
will be called the {\it itinerary of the descending decomposition}.
\end{defin}

\smallskip It follows from the construction that for any $1\le k\le
m-1$ one can find $i_k\le i'_k$ so that $B(m,\ell)=\cup_{i_k\le s\le
i^\prime_k}H^k_{s}$, and if $\ell<k$ we have $i_k=i'_k$. In
particular it exists $j$ so that $B(m,\ell)=H^{m-1}_{j}$. (For
$\ell=0$ this is a single bad line, and $m=1$.)

\smallskip
Notice that if $k > \ell -1$ in the above definition, then it is
always the case that $i'=i$; the case $i'=i\pm 1$ may occur only if
$k= \ell -1$.

\medskip
\begin{defin}
\label{A22} {\em (Zones and tunnels)} If $S^k_{(i,i_k-1)},
\widehat{S}^k_{(i',i'_k +1)}$ form a matching pair with respect to
$B(m,\ell)$, we set
$$
{\cZ}(S^k_{(i,i_k-1)},\widehat S^k_{(i',i'_k +1)})=\left[
(cL)^k\left(\frac{i-L^{1/2}}{2}\right),
(cL)^k\left(\frac{i+L^{1/2}} {2}\right)\right] \times
\cH^{m-1}_{j},
$$
which will be called the {\it zone} associated to
$S^k_{(i,i_k-1)}$ and $\widehat S^k_{(i',i'_k +1)}$.

For $k \ge 1$ and if $i=i'$, we  set
$$
T(S^k_{(i,i_k-1)}, \widehat S^k_{(i',i'_k +1)})=\left[\frac {i-1}{2}
(cL)^{k},\frac {i+1}{2} (cL)^{k}\right] \times \cH^{m-1}_{j},
$$
which we call the  {\it tunnel} associated to $S^k_{(i,i_k-1)}$
and $\widehat S^k_{(i',i'_k +1)}$.

\noindent If $|i'-i|=1$, the {\it tunnel} associated with
$S^k_{(i,i_k-1)}$ and $\widehat S^k_{(i',i'_k +1)}$ is defined in
the following way:
$$
T(S^k_{(i,i_k-1)}, \widehat S^k_{(i',i'_k +1)})=\left[\frac {i\wedge i'
}{2} (cL)^{k},\frac {i\vee i'}{2} (cL)^{k}\right] \times
\cH^{m-1}_{j}.
$$

\noindent And finally, when $k=0$ we set
\begin{equation*}
T((i,i_0-1), (i', i'_0+1))=\begin{cases} [i-1,i]\times \cH^{m-1}_{j}
\text{
if } i=i'\\
[i\wedge i',i\vee i']\times \cH^{m-1}_{j} \text{ if }
|i-i'|=1.\end{cases}
\end{equation*}
\end{defin}

\begin{defin}
\label{vertical} {\em(Vertical sequences)} A collection of
$k$-sites $\{S^k_{(u_s,v_s)}\}_{ s=1}^r$, with  $v_{n+1} = v_n+1,
\; n=1, \dots , r-1$, is called {\it a vertical sequence} if
$|u_1-u_{s}|\le 1$ for all $1<s<r $.
\end{defin}

\begin{defin}
\label{above}  We say that the $k$-site $S^k_{(u_2,v_2)}$ lies
{\it above } $S^k_{(u_1,v_1)}$, or, equivalently, $S^k_{(u_1,v_1)}$
lies {\it below} $S^k_{(u_2,v_2)}$, if $v_1 < v_2$, and
$|u_1-u_{2}|\le 1$.
\end{defin}

\noindent We will use the above definition also in the case of
sequences of reversed sites.
\medskip

\begin{defin}
\label{A3} {\em (Separated pairs)} Two matching pairs
$S^k_{(i,i_k-1)}, \widehat S^k_{(i',i'_k +1)}$ and
$S^k_{(j,i_k-1)}$, \newline $\widehat S^k_{(j',i'_k +1)}$ are said to be {\it
separated} if $ |j-i|\ge 2L^{1/2}$.
\end{defin}

\noindent Notice that if two matching pairs are separated, their
corresponding zones do not intersect.
\medskip

\noindent {\bf Notation.} For an horizontal segment $I =\{(x,y)\in
\widetilde{\Bbb Z}^2_+\colon a\le x\le b\}$  we denote
\begin{eqnarray*}
I_{\Lsh} &= \{(x,y)\in
\widetilde{\Bbb Z}^2_+\colon a+ {\lfloor}
(b-a)/12{\rfloor}+1\le x\le a+
2{\lfloor}(b-a)/12{\rfloor}-1 \},\\
I_{\Rsh} &=\{(x,y)\in
\widetilde{\Bbb Z}^2_+\colon a+ 10{\lfloor}(b-a)/12{\rfloor}+1\le x \le a +
11{\lfloor} (b-a)/12{\rfloor}-1 \}.
\end{eqnarray*}

\begin{defin}
\label{segmentito} An horizontal segment $I =\{(x,y)\in
\widetilde{\Bbb Z}^2_+\colon a\le x\le b\}$ with $b-a=(cL)^k$ is
called $k$-segment either if it is contained in some good $k$-site
$S^k$, or if there is a bad layer $B(m,\ell),\; \ell < k,$ and two
good $k$-sites $S^k$, $\widehat{S}^k$ forming a matching pair with
respect to $B(m,\ell)$ such that $I\subset T_{(S^k,
\widehat{S}^k)}$. We denote a $k$-segment $I$ by $I^k$.
\end{defin}

\begin{defin}
\label{hset} {\em (hierarchical $k$-set)} Given a $k$-segment
$I^k$,  a collection $\overline {I^k}$ of $\ell$-segments
$\{I^\ell_j\}_{j}$, $\ell=0,\dots,k$,  contained in
$I^k$ is  a hierarchical $k$-set
associated with $I^k$ if:
\smallskip

\noindent i) $I^k$ is the unique $k$-segment in the collection;

\noindent ii) $I^\ell_j \cap I^\ell_{j'}=\emptyset$ if $j \ne j'$,
for any $\ell\in \{0,\dots, k-1\}$;

\noindent iii) for $k \ge 1$ and $\ell\in \{1,\dots,k\}$, each interval $(I^\ell_{j})_\Lsh$
 and $(I^\ell_{j})_\Rsh$ contains
at least $\frac{1}{12}\rho c L$ \; $(\ell-1)$-segments
in $\overline {I^k}$.
\end{defin}

When $k=0$ we simply have $\overline{I^0}=\{I^0\}=\{S^0\}$ for a
$0$-site $S^0$ and we identify $\overline{I^0}$ with $I^0$. For $k \ge 1$ and having fixed ${I^k}$,
it is convenient to label the elements of $\overline {I^k}$:
going down, from $\ell=k-1$ to $\ell=0$, we label all $\ell$-segments
contained in each $I^{\ell+1}$ from left to right, starting the numbering
within each $I^{\ell+1}$ every time from 1. Proceeding in this way, we
have a multi-index $ {\mu_{\langle k,\ell \rangle}} = \langle
\mu_{k-1} , \mu_{k-2}, \dots, \mu_{\ell} \rangle$ which indicates
the ``genealogical tree"  down to scale $\ell$. We denote the
corresponding $\ell$-segment with this index by $I^{\ell}_{
{_{\mu_{\langle k,\ell \rangle}}}}$.
We shall also use the following convention. If $\mu_{\ell}=j$, we
will write:
\begin{equation}
\label{ind-hierar}
\langle \mu_{k-1}, \dots, \mu_{\ell+1},j \rangle =
%\langle \mu_{k-1}, \dots, j \rangle =
\langle {\mu_{\langle k, \ell+1 \rangle}}, j \rangle.
\end{equation}

\begin{defin}
\label{A3lis}

\noindent (a)  For any type 2 good $k$-site $S^k, \; k \ge 1$,
$\Psi^{k} {(S^k)}$ denotes its top $0$-layer. Analogously,
if $\widehat S^k$ is a good reverse $k$-site of type 2,
$\Upsilon^{k} {(\widehat S^k)}$ denotes its bottom $0$-layer.
When $k=0$, we set $ \Psi^{0} (S^0)= S^0$ and
$\Upsilon^{0} (\widehat S^0)= \widehat S^0$.

\noindent (b) If the $k$-sites $S^k$ and $\widehat S^k$ form a
matching pair with respect to $B (m, \ell)$ in the sense of Definition \ref{A2}, and
$\ell-1\le k \le m-1$, we say that $\Psi^{k} {(S^k)}$ and $\Upsilon^{k} {(\widehat
S^k)}$ also form a matching pair;

\noindent (c) Two hierarchical $k$-sets $\overline{\Psi^k}$ and
$\overline{\Upsilon^k}$ whose $k$-segments $\Psi^k$ and
$\Upsilon^k$ form a matching pair with respect to $B(m,\ell)$, with
$\ell-1\le k \le m-1$, are also called a matching pair with respect
to $B(m,\ell)$.
\end{defin}

\smallskip
For the proof in Section \ref{conclusion}  we shall use the
following hierarchical $k$-sets: let $S^k$ be a good $k$-site with
dense kernel. In this case, there will be  at least
$\frac{1}{12}\rho c L$\, $(k-1)$-sites in $D^K_l (S^k)$ and $D^K_r
(S^k)$ respectively, and each of them will have dense kernel. The
same happens at all scales down 0. The top 0-layers of the kernel of
these dense kernel sites at all scales form a hierarchical $k$-set,
which we denote as $\overline{\Psi^k} (S^k)$. The analogous
hierarchical $k$-set for a reverse $k$-site $\widehat S^k$ we will
denote by $\overline{\Upsilon^k}(\widehat S^k)$. We shall use this
in the case when $S^k$ is of type 2, lying immediately below a bad
layer of mass larger than $k$ (so that $S^k$ coincides with its
kernel) or when it contains a bad layer of mass $k$ (type 1).

\medskip
%%%%%%%%%%%%%%%%%%%%%%%%%%%%%%%%%%%%%%%%%%%%%%%%%%%%%%%%%%%%%%%%%%%%%%%%%%%%%%%%%%%%%%%%%%%%%%%%%%%%%%%%%%%%%%%%%%%%%%%%%%
 \noindent {\bf Notation.} It will be convenient to single out
the class of level $1$ bad layers $B(m,1)$ consisting of $m$
consecutive bad lines. We call such bad layers  ``monolithic'',
and refer to them as bad $1_M$-layers. In this case we write
$B(m,1_M)$.
\medskip

\noindent The following concept of {\it chaining} plays an important
role in the proof in Section \ref{onclusion}. We split it into two
definitions, for ``large" and ``small" hierarchical sets, where
large or small has to do with the level of the bad layer, as made
precise below.

\medskip
%%%%%%%%%%%%%%%%%%%%%%%%%%%%%%%%%%%%%%%%%%%%%%%%%%%%%%%%%%%%%%%%%%%%%%%%%%%%%%%%%%%%\Definition (A4) ({\it Chained pair.})

\begin{defin}
\label{chained2} {\em (``Large" chained hierarchical $k$-sets)} Let
$B(m,\ell)$ be a bad layer of level $\ell$ and mass $m$. If
$k\in\{\ell -1,\dots, m-1\}$, two hierarchical $k$-sets $\overline
{\Psi^k}$ and $\overline {\Upsilon^k}$ forming a matching pair with
respect to $B(m,\ell)$ are said to be {\it chained through
$B(m,\ell)$} if the following holds:
\medskip

\noindent {\bf The case $k= 0$.}  In this situation $\ell=1$, and
we distinguish the $1_M$-layers from the remaining  $B(m,1)$
layers.

\noindent $\bullet$ $1_M$-layer.  We say that $\overline {\Psi^0}$
and $\overline {\Upsilon^0}$ are chained if there exists an open
vertical path of $0$-sites from a nearest neighbor of $\Psi^0$
to nearest neighbor of  $\Upsilon^0$;

\smallskip

\noindent $\bullet$ Non-monolithic layer. In this case $B(m,1)$ is
formed by $m$ bad lines grouped into $r>1$
$1_M$-layers\footnote{this is slight abuse of our previous
notation}, separated among themselves by at most $L-1$ good lines.
We denote these parts by $B_v(m_v,1_M), \; 1 \le v \le r$, where
$m_v$ is the number of bad lines it contains. We say that $\overline
{\Psi^0}$ and $\overline {\Upsilon^0}$ are chained through $B(m,1)$
if there exist a vertical sequence of hierarchical $0$-sets
$$
\widehat S^0(1),S^0(2), \widehat S^0(2),\dots, S^0(r),
$$
such that

a) all sites $S^0(v), \; v= 2, \dots, r,$ and $\widehat S^0(s)$,
$v=1, \dots, r-1$, are passable;

b) the sites $S^0(v)$ and $\widehat S^0(v), \; v=2, \dots , r-1,$
form a matching pair with respect to $B_v(m_v,1_M)$;

c) $\Psi^0$ and $S^0(1)$ are chained through $B_1(m_1, 1_M)$;
$S^0(v)$ and $\widehat S^0(v)$ are chained through $B_v(m_v, 1_M)$,
for each $v=2, \dots, r-1$; $S^0(r)$ and $\Upsilon^0$ are chained
through $B_r(m_r, 1_M)$.

d) each pair of sites $\widehat S^0(s)$ and  $S^0(s+1)$, $1\le s
\le r-1$, is connected by an open path of 0-sites lying within
${\cZ}(S^0_{(i,i_k-1)}, \widehat S^0_{(i',i'_k +1)})$.

\medskip

\noindent {\bf The case $k \ge 1$.} Again we distinguish two
cases:

\smallskip

\noindent $\bullet$ $k\ge \ell $.  Since $\rho >1/2$, letting
$\widehat \rho=\rho-1/2$, the definition of hierarchical set implies
the existence of at least $ \widehat\rho \frac{c}{6} L$ matching
pairs $\overline {\Psi^{k-1}}, \, \overline {\Upsilon^{k-1}}$ with
respect to $B(m,\ell)$, with $\overline {\Psi^{k-1}} \subset
\overline {\Psi^k}$ and $\overline {\Upsilon^{k-1}} \subset
\overline {\Upsilon^k}$. Let $\mathcal{M}$ be the set formed by the
first (from left to right) $ \widehat\rho \frac{c}{6} L^{1/2}$ such
pairs which are separated. We say that $\overline {\Psi^k}$ and
$\overline {\Upsilon^k}$ are chained if at least one matching pair
in $\cM$ is chained through $B(m,\ell)$.

\medskip

\noindent $\bullet$ $k =\ell-1$. Assume that $B(m,\ell)$ has $r>1$
constituents of masses $m_v$ and levels $\ell_v < \ell$, hereby
denoted as $B_v(m_v, \ell_v)$, $v=1,\dots,r$. We say that $\overline
{\Psi^k}$ and $\overline {\Upsilon^k}$ are chained if there exist a
vertical sequence of good $k$-sites
$$
 \widehat S^k(1),S^k(2), \widehat
S^k(2),\dots, S^k(r),
$$
such that

a)  For each $1\le v \le r-1$,  $S(\widehat S^k(v))$ and $S^k(v+1)$ are connected by a
passable $k$-path, lying entirely  in ${\cZ}(S^k_{(i,i_k-1)}, \widehat S^k_{(i',i'_k +1)})$.
(In particular, $S^k(v)$ has $s$-dense
kernel, $v=2,\dots,r$.)

b) $\widehat S^k(v)$ has $c$-dense kernel (reversed),   $v=1,$ $
\dots,r-1$.

c) $S^k(v)$ and $\widehat S^k(v)$ form a matching pair which respect
to $B_v(m_v, \ell_v)$, $v=2, \dots,r-1$.

d) $\overline {\Psi^k}$ and $\overline {\Upsilon^{k}}(\widehat
S^k(1))$ are chained through $B_1(m_1, \ell_1)$; $\overline
{\Psi^{k}} (S^k(v))$ and $\overline {\Upsilon^{k}}(\widehat S^k(v))$
are chained through $B_v(m_v, \ell_v)$, $v=2, \dots,r-1$; and
finally $\overline {\Psi^{k}} (S^k(r))$ and $\overline {\Upsilon^k}$
are chained through $B_r(m_r, \ell_r)$.

\smallskip

 The connections by open oriented paths that are examined using the iterative
 procedure just defined will be called {\it restricted}.
\end{defin}

\medskip

\noindent {\bf Notation.} For easiness of notation we shall write
$J=\widehat \rho \frac{c}{6}L^{1/2}$.

\medskip

\noindent {\bf Remarks.}

\noindent a) Let $B(m,\ell)$ be a bad layer of mass $m$ and level
$\ell$ through which a matching pair of hierarchical $k$-sets
$\overline {\Psi^k}$ and $\overline {\Upsilon^k}$ is chained. If $k
\ge \ell$, there exists an open oriented (restricted) path of
0-sites crossing $B(m,\ell)$ and lying in $T(S^k_{(i,i_k-1)},
\widehat S^k_{(i',i'_k
 +1)})$; if $k=\ell -1$ such a path exists in
 $\mathcal{Z}(S^k_{(i,i_k-1)}, \widehat S^k_{(i',i'_k+1)})$.

\noindent b) Notice that in Definition \ref{chained2}, for each $j
\in \{0,\dots,k-1\}$, we examine at each step (according to the set
$\mathcal{M}$ in Definition \ref{chained2}) exactly $J$
\;$j$-segments within each checked $j+1$-segment in $\overline
{\Upsilon^k}$, and similarly for $\overline {\Psi^k}$; each checked
to be connected to different $j$-segments within $B(m,\ell)$. The
algorithm for selecting $\mathcal{M}$ at each smaller scale depends
(except in the trivial case of layers $B(m,1_M)$) on what happens
within $B(m,\ell)$ as explained therein. With some abuse of notation
we call $\mathcal{M}(\overline {\Upsilon^k})$ and similarly
$\mathcal{M}(\overline {\Psi^k})$ the collection of these checked
segments at all scales ($J$ at each scale).

\noindent c) The estimates in the next section become easier to
formulate once the number of $j$-segments to be examined within a
$j+1$-segment is fixed at all times.  But it does not depend on the
exact algorithm to define the set $\mathcal{M}$ and for the
construction in Section \ref{conclusion} this will be slightly
different then the one used in the above definition, though we shall
use the same notation.

\medskip

\begin{defin}
\label{chained2'} {\em (``Small" chained hierarchical sets)} Let
$\ell \ge 2$, $k\in\{\ell -1,\dots, m-1\}$, and let $B(m,\ell)$ be a
bad layer of mass $m$ and  level $\ell$ for which $\overline
{\Psi^k}$ and $\overline {\Upsilon^k}$ form a matching pair of
hierarchical $k$-sets, assumed to be chained according to Definition
\ref{chained2}.

\smallskip

\noindent a) We say that a 0-site $\Upsilon^0_{\mu_{\langle k,
0\rangle }} \in \mathcal{M}(\overline {\Upsilon^k})$ is chained to
$\overline {\Psi^k}$, if there exists a 0-site $\Psi^0_{\mu_{\langle
k, 0 \rangle }} \in \mathcal{M}(\overline {\Psi^k})$, and a
restricted open path from a nearest neighbor (from above) of
$\Psi^0_{\mu_{\langle k, 0 \rangle }}$ to a nearest neighbor of
$\Upsilon^0_{\mu_{\langle k, 0\rangle}}$ from below.

\noindent b) We say that the $r$-segment $\Upsilon^r_{\mu_{\langle
k, r\rangle }} \in \mathcal{M}(\overline {\Upsilon^k})$ , $0< r <
\ell -1$, is chained to $\overline {\Psi^k}$ if it contains an
$(r-1)$-segment $\Upsilon^{r-1}_{\mu_{\langle k, r-1 \rangle }}\in
\mathcal{M}(\overline {\Upsilon^k})$ which is chained to $\overline
{\Psi^k}$.
\end{defin}

\noindent {\bf Remark.} A $0$-path as in a) above will be open,
oriented, and will lie entirely in the tunnel $T(\Upsilon^k,
\Psi^k)$ when $k \ge \ell$, and for $k=\ell-1$ it will lie in
$Z(\Upsilon^{\ell-1}, \Psi^{\ell-1})$. \footnote{Defined analogously
to $T(S, \widehat S)$ and $Z(S,\widehat S)$.}

\begin{defin}
\label{chnd2} {\em(chained $k$-sites)} Let $B(m,\ell)$ be a
bad layer of level $\ell$ and mass $m$, and $k\in\{\ell -1,\dots,
m-1\}$. Two $k$-sites $S^k$ and $\widehat S^k$ forming a matching
pair with respect to $B(m,\ell)$ are said to be chained through
$B(m,\ell)$, if the corresponding hierarchical $k$-sets
$\overline{\Psi^k}(S^k)$ and $\overline{\Upsilon^k}(\widehat S^k)$ are
chained through $B(m,\ell)$.
\end{defin}

\medskip \noindent {\bf Notation.} The event of two hierarchical
$k$-sets $\overline {\Psi^k}$, $\overline {\Upsilon^k}$, or analogously
two $k$-sites $S^k$, $\widehat S^k$ being chained through
$B(m,\ell)$ is denoted by
\begin{equation}
\overline {\Psi^k}  \leftrightsquigarrow_{_{\! \! \! \! \! \!\! \!
\! \! \! \!\!\!  B(m,\ell)}} \overline {\Upsilon^k} \label{eqchnd2}
\end{equation}
and respectively
\begin{equation}
S^k  \leftrightsquigarrow_{_{\! \! \! \! \! \!\! \! \! \! \!
\!\!\!  B(m,\ell)}} \widehat S^k. \label{eqchained}
\end{equation}
\bigskip

%%%%%%%%%%%%%%%%%%%%%%%%%%%%%%%%%%     %%%%%%%%%%%%%%%%%%%%%%%%%%%%%%%%%%%%%%%%%%%%%%%%%%%%

\section {Conclusion of the proof of Theorem \ref {indstep}}
\label{conclusion}
\noindent {\bf Notation.} Let $\varkappa>0$ to be fixed later. We
recursively define for all $m\ge 1$:
\begin{eqnarray}
\label{6N.1}
\nonumber p_{0,m}&:=&p_{_B}^m p_{_G}^{\varkappa(m-1)},\\
p_{k,m}&:=&(1-(1-p_{k-1,m})^J)p_k^{\varkappa(m-k-1)},\qquad 1\le
k\le m-1,\\\nonumber p_{m,m}&:=&1-(1-p_{m-1,m})^J,\\\nonumber
\end{eqnarray}
where $J={\widehat\rho
\frac{c}{6}L^{1/2}}$, $\widehat \rho=\rho -1/2$ as in Definition \ref{chained2},
 $p_k=1-q_k$, $q_k=q_0^{k+1}$, $p_0=p_{_G}, q_0=1-p_G$, as in Theorem \ref{indstep}.

%\maltese TALVEZ MELHOR DEIXAR aqui $q_1=(1-p_G)^t$ e determinar $t$ depois (junto com
%$L$ ou entao em duas etapas,... detalhe menor \maltese
\bigskip

Recalling the statement of Theorem \ref{indstep} and what has been
proven in Section \ref{passability}, it remains to verify that
$(b_{m+1})$ follows from $(a_j),(b_j),(c_j),(d_j)$ for all $j\le m$
and $(a_{m+1})$. To get such estimates we need a more detailed
analysis, as developed in the last section. For $m\ge 1$, let
$p_{k,m}$ be given as above, and set:
\medskip

\noindent $(b_m)^\prime$  For every bad layer $B(m,\ell)$ of mass $m$ (any level $\ell$), every  $j \in\{\ell-1,\dots, m-1\}$  and every
hierarchical $j$-sets $\overline {\Psi^j},\overline { \Upsilon^j}$ that form a matching pair with respect to $B(m,\ell)$,
\begin{equation}
P(\overline {\Psi^j}  \leftrightsquigarrow_{_{\! \! \! \! \! \!\! \!
\! \! \! \!\!\!  B(m,\ell)}} \overline {\Upsilon^j})\ge p_{j,m}.
\label{6N.2}
\end{equation}
\medskip
For $m \ge 2$:

\noindent $(b_m)^{\prime\prime}$ For every $B(m,\ell)$, $j$,
$\overline {\Psi^j},\overline { \Upsilon^j}$ as in $(b_m)^\prime$,
and every $s\in\{0,\dots,j-1\}$, the distribution of the number of
$\Upsilon^s_{\langle {\mu_{\langle j,s+1 \rangle}},i \rangle}\in
\mathcal{M}(\overline{\Upsilon ^j})$ that are chained to $\overline
{\Psi^j}$, conditioned on $\Upsilon^{s+1}_{{\mu_{\langle j,
s+1\rangle }}}$
  being chained to $\overline {\Psi^j}$, is stochastically larger than $F_{p_{s,m}}$, where $F_p$ denotes
   the distribution of a Binomial random variable with $J$ trials and success probability $p$,
    conditioned to have at least one success. That is,
\begin{eqnarray}
 | \{i \colon
\Upsilon^s_{\langle {\mu_{\langle j,s+1 \rangle}},i \rangle}\in
\mathcal{M}(\overline{\Upsilon ^j}) \colon\;\Upsilon^s_{\langle
{\mu_{\langle j,s+1 \rangle}},i \rangle}
 {\text {chained to }} \overline {\Psi^j} \}| \Big|
\big[   \Upsilon^{s+1}_{{\mu_{\langle j, s+1\rangle }}} \, {\text {
chained to }} \overline {\Psi^j} \big] \succeq
F_{p_{s,m}}\label{6N.3}
\end{eqnarray}
with $\succeq$ standing for stochastically larger in the usual
sense.

\bigskip
\begin{thm}
\label{inductive}
The properties $(a_m)$, $(d_m)$ of Theorem \ref{indstep}
and the above $(b_m)^{\prime}$ hold for every
$m \ge 1$;  $(b_m)^{\prime\prime}$ holds for every
$m\ge 2$.
\end{thm}
\medskip

\noindent {\it Proof.} The proof is by induction in $m$.

\noindent {\it Initial step.}   $(b_1)^{\prime}$ follows directly from the
definitions, and  $(a_1)$, $(d_1)$ have already been verified in Section
\ref{passability}. $(b_2)^{\prime\prime}$ is also trivially verified.

\medskip

\noindent {\it Induction step.} We first
establish $(b_{m+1})^\prime$. Throughout the proof we  use the descending decomposition
representation of the bad layer $B(m+1,\ell)$; we also construct a class of particularly chosen
hierarchical sets that will play a role in the induction.

\medskip

\noindent Let $(\overline {\Psi^{m}}, \overline {\Upsilon^{m}})$
form a matching pair with respect to a bad layer $B(m+1,\ell)$, and
let $\{ \widetilde m_s \}_{s=1}^v$ denote the itinerary of the
descending decomposition of $B(m+1,\ell)$, with $\{ \widetilde
{{\cC}_s} \}_{s=1}^v$ its corresponding clusters, and
$\{B_{\widetilde {{\cC}_s}}\}_{s=1}^v$ the corresponding bad layers.
The interval between any two consecutive clusters $\widetilde
{\cC}_s$ and $\widetilde{\cC}_{s+1}$ is always porous media of level
$\widetilde m_{s+1}$ (see Lemma \ref{decomp}).

\medskip
\noindent An {\it entrance set} $\overline {\Psi^{m}} (s), \, s= 2,
\dots, v$, will be a suitable hierarchical $m$-set located at the
$0$-layer just below $B_{\widetilde {\cC}_s}$, and an {\it exit set}
$\overline {\Upsilon^{m}} (s), \, s= 1, \dots, v$,  a suitable
hierarchical $m$-set located at the $0$-layer just above
$B_{\widetilde {\cC}_s}$ for $s= 1, \dots, v-1$, with $\overline
{\Upsilon^{m}} (v)$ located at the $0$-layer just above
$B_{\widetilde {\cC}_v}$, or at its last $0$-layer, according to
$g_v < \omega(\cC)-1$ or $g_v = \omega(\cC)-1$ (Lemma \ref{decomp}).

\medskip {\it Large segments of exit and entrance sets.} For each $s=1,\dots,v-1$,
the $m$-segment ${\Upsilon}^{m} (s)$ and all $j$-segments
${\Upsilon}^{j}_{\mu_{\langle m, j \rangle}} (s)$, $\widetilde
m_{s+1} \le j < m$,  in $\overline {\Upsilon^m}(s)$, are obtained by
taking the corresponding segments ${\Upsilon}^{m}$ and
${\Upsilon}^{j}_{\mu_{\langle m, j \rangle}}$ and projecting them
vertically on the $0$-layer located just above $B_{\widetilde
{\cC}_s}$, and then by taking as ${\Upsilon}^{j}_{\mu_{\langle m, j
\rangle}}(s)$ a $j$-segment which intersects this projection:  when
there are two such $j$-segments, to avoid ambiguities we take the
one which intersects the left half of the projection. For $s=v$ the
only difference is that when $g_v = \omega(\cC)-1$  the segments
will be located at the last $0$-layer of $B_{\widetilde {\cC}_v}$.

For the entrance sets $\overline{\Psi^m}(s)$ with $s=2,\dots, v$ we
proceed in the same way: the $m$-segment ${\Psi}^{m} (s)$ and all
$j$-segments ${\Psi}^{j}_{\mu_{\langle m, j \rangle}} (s)$,
$\widetilde m_{s} \le j < m$,  in $\overline {\Psi^m}(s)$, are
obtained by taking the corresponding segments ${\Upsilon}^{m}$, and
${\Upsilon}^{j}_{\mu_{\langle m, j \rangle}}$ and projecting them
vertically on the $0$-layer located just below $B_{\widetilde
{\cC}_s}$, with the same selection rule as above in case there are
two such $j$-segments.

\medskip
\noindent {\it Construction of the exit sets $\overline {\Upsilon^m}
(s)$.}  Consider first the case $1\le s <v$. To continue the
construction of the $j$-segments at scales smaller than $\widetilde
m_{s+1}$, we consider, for each already defined $\widetilde
m_{s+1}$-segment of this collection,  the reversed $\widetilde
m_{s+1}$-site for which this segment is the last 0-layer, i.e.
$\widehat S^{\widetilde m_{s+1}}$ such that $\Upsilon (\widehat
S^{\widetilde m_{s+1}}) = {\Upsilon}^{\widetilde
m_{s+1}}_{\mu_{\langle m, \widetilde m_{s+1} \rangle}} (s)$, and
check if this site has $c$-(reverse) dense kernel. If the answer is
affirmative, we take ${\Upsilon}^{j}_{\mu_{\langle m, j \rangle}}
(1), \; j = 0, \dots, \widetilde m_{s+1}-1$  as the bottom 0-layers
of the reverse dense kernel sites of $\widehat S^{\widetilde
m_{s+1}}$, or in other words all the scales from $\widetilde
m_{s+1}$ down to zero correspond to $\overline{\Upsilon^{\widetilde
m_{s+1}}}(\widehat S^{\widetilde m_{s+1}})$. These are called
``compatible" segments. For those sites $\widehat S^{\widetilde
m_{s+1}}$ that do not have $c$-(reverse) dense kernel, we select the
${\Upsilon}^{j}_{\mu_{\langle m, j \rangle}} (s), \; j = 0, \dots,
\widetilde m_{s+1}-1$ in an arbitrary way among the correspondent
sub-segments of bottom 0-layers of the site. We call such choice of
segments ``incompatible" with the process. Only compatible segments
will play a role in the construction.

In the case $s=v$ we make essentially the same construction, as if $\widetilde m_{s+1}=0$,
with the difference that when  $\omega (B(m+1, \ell)) = \omega (\widetilde {\cC}_v) +1$
we locate the hierarchical set at the last $0$-layer of $B_{\widetilde{\cC}_v}$.

\medskip

\noindent Observe that the construction of $\overline {\Upsilon^
m} (s)$ and the compatibility of its segments depend on $\Gamma$ and
on the occupation variables in between $B_{\widetilde{\cC}_s}$ and
$B_{\widetilde{\cC}_{s+1}}$.

\medskip

\noindent {\it Step 1, part 1.} We check if
%$\overline {\Psi^{m}}$ and
%$\overline {\Upsilon^m} (1)$ are chained through $B_{\widetilde{\mathcal C}_1}$,
%according to the Definition \ref{chained2}, i.e. if
at least one among the pairs of hierarchical $(m-1)$-sets $\overline
{\Psi^{m-1}_{\mu_{\langle m, m-1 \rangle}}}$ and $ \overline
{\Upsilon^{m-1}_{\mu_{\langle m, m-1 \rangle}}}(1)$ is chained
through $B_{\widetilde{\cC}_1}$. If so, we move to the next item;
otherwise we stop the procedure and say that $\overline {\Psi^{m}}$
and $\overline {\Upsilon^{m}}$ are not chained through
${B(m+1,\ell)}$.
\medskip

\noindent {\it Step 1, part 2.} ({\em Zooming})  For each pair of
hierarchical $(m-1)$-sets $\overline {\Psi^{m-1}_{\mu_{\langle m,
m-1 \rangle}}}$ and $ \overline {\Upsilon^{m-1}_{\mu_{\langle m, m-1
\rangle}}}(1)$ chained through $B_{\widetilde{\cC}_1}$ we select all
multi-indices ${\mu_{\langle m, \widetilde{m}_2 \rangle}}$ and the
corresponding $\widetilde m_{2}$-segments  $\Upsilon^{\widetilde
m_{2}}_{\mu_{\langle m,\widetilde m_{2}\rangle}}(1)$, which are
compatible and from which there exists an open oriented 0-level path
that connects to $\Psi^{m-1}_{\mu_{\langle m, m-1 \rangle}}$ through
$B_{\widetilde{\cC}_1}$.
%\maltese \footnote{Em principio nao podemos
%dizer que seria conectado a $\Psi^{\widetilde m_{2}}$ como escrito
%na versao velha, pois podemos ter $\widetilde m_{2} < \ell(\tilde
%C_1)-1$ de acordo?\maltese }

\noindent {\it Step 1, part 3.} ({\em Transfer}) For the $\widetilde m_{2}$-segments
$\Upsilon^{\widetilde m_{2}}_{\mu_{\langle m,\widetilde m_{2}\rangle}}(1)$ selected in the
previous item, we first check whether

\noindent (i) the forward site $S(\widehat S^{\widetilde m_2}_{\x})$
 is $c$-(forward) passable.
 \medskip

\noindent If the answer is positive, it implies that at least one of the seeds of $Q_l
(S(\widehat S^{\widetilde m_2}_{\x}))$ or $Q_r (S(\widehat
S^{\widetilde m_2}_{\x}))$ is also connected to $\Psi^{m-1}$ (we may call it ``active").
This gives us a way of completing the construction of the hierarchical set $\overline {\Psi^{m}}(2)$
at scales smaller than $\widetilde m_{2}$:
\smallskip

\noindent {\it Construction of the entrance set $\overline {\Psi^{m}}(2)$.}  Take $S^{\widetilde m_{2}}_{\x'}$ such that $\Psi
(S^{\widetilde m_{2}}_{\x'}) = {\Psi}^{\widetilde m_2}_{\mu_{\langle m,\widetilde m_{2}\rangle}}(2)$, and check
whether

\noindent (ii) there exists an oriented passable $\widetilde
m_2$-path starting from the $\widetilde m_{2}$-site which is
$s$-passable from the active seed of  $S(\widehat S^{\widetilde
m_2}_{\x})$ to  $S^{\widetilde m_2}_{\x'}$, and entirely contained
in ${\cZ}(S(\widehat S^{\widetilde m_2}_{\x}),S^{\widetilde
m_2}_{\x'})$.

\smallskip

%\noindent We denote
%\begin{eqnarray*}
%\big[ \Upsilon^{\widetilde m_2}_{\mu_{\overline{\widetilde m_2}}}(1)) \leftrightsquigarrow \overline
%{\Psi}^{\widetilde m_2} _{\overline {\mu_{\widetilde m_2}}}(2)\big] = \, & \big[\text
%{$\Upsilon^{\widetilde m_2}_{\mu_{\overline{\widetilde m_2}}}(1))$,
%${\Psi}^{\widetilde m_2} _{\overline {\mu_{\widetilde m_2}}}(2)$ are compatible
%with the} \cr & \text{process and $\exists$ restricted open
%$\widetilde m_2$ path linking them}\big].
%\end{eqnarray*}

\medskip
%%%%%%%%%%%%%%%%%%%%%%%%%%%%%%%%%%%%%%%%%%%%%%%%%%%%%%%%%%%%%%%%%%%%%%%%%%%%%%%%%%%%%%%%%%%%%%%%%%%%%%%%%%%%%% MELECOLIS

\noindent If the answer to (i) and (ii) is positive we say that
$\overline{\Upsilon^{\widetilde m_2}_{\mu_{\langle m,\widetilde m_2}\rangle}}(1)$
and $\overline{\Psi^{\widetilde m_2}_{\mu_{\langle m,\widetilde m_2}\rangle}}(2)$ are ``active".
Otherwise the procedure of building connection from
$\Upsilon^{\widetilde m_2}_{\mu_{\langle m,\widetilde m_2}\rangle}(1)$ is stopped. This completes Step 1.

\begin {remark1}
\label{chain1} Notice that $\overline{\Psi^m}(2)$ lies just below
$B_{\widetilde C_2}$, but a positive answer to (i) and (ii) above,
besides guaranteeing the connection of $\overline{\Psi^{\widetilde
m_{2}}}(2)$ to the corresponding $\Upsilon^{\widetilde m_{2}}(1)$
(and therefore to $\Psi^{m-1}$ by force of the previous sub-step)
also gives connection by open oriented path of $0$-sites to suitable
sites at the top $0$-layer of $B_{\widetilde C_2}$ (according to the
definition of passability at the scale $\widetilde m_{2}$), which
then implies the existence of an open path to a $0$-site in
$B_{\widetilde C_2}$  which is nearest neighbor of a corresponding
$\widetilde m_{2}$-segment $\Upsilon^{\widetilde m_2}_{\mu_{\langle
m,\widetilde m_2\rangle}}(2)$. The first part does not depend on the
occupation variables in $B_{\widetilde C_2}$, and one might find
convenient to think of the event in (ii) as the intersection of
these two conditions involving disjoint sets of $0$-sites.
\end{remark1}

%\noindent If the answers to (i)  and (ii) are all positive, it
%implies that $\overline {\Psi^m}_{\overline{\mu_{\widetilde m_2}}}(2)$ is
%compatible with the process, and  it also implies that if
%$\overline {\Psi^m}_{\overline\mu_{\widetilde m_2}}(1)$ is active, then
%$\overline {\Psi^m}_{\overline\mu_{\widetilde m_2}}(2)$ is active too.
%(Sometimes it will be convenient to say that a path described in
%ii) is linking  $\overline {\Upsilon^m}_{\overline\mu_{\widetilde m_2}}(1)$
%to $\overline {\Psi^m}_{\overline\mu_{\widetilde m_2}}(2)$.)

\bigskip

\noindent {\it Step $s, \; 1< s \le v$.} Having determined the
``active" $\overline{\Upsilon^{\widetilde m_s}_{\mu_{\langle
m,\widetilde m_s}\rangle}}(s-1)$ and $\overline{\Psi^{\widetilde
m_s}_{\mu_{\langle m,\widetilde m_s}\rangle}}(s)$, the process
continues only from the ``compatible" corresponding sub-segments
$\Upsilon^{\widetilde{m}_{s+1}}_{\mu_{\langle
m,\widetilde{m}_{s+1}\rangle}}(s)$.

\noindent {\it Sub-case $s<v$.} The construction repeats what was done above for $s=1$:

\medskip

%\begin{remark1}
%\label{chain2}
%As described above (Remark \ref{chain1}) we may say that $\overline{\Upsilon^{m_{s+1}}_{\mu\langle m,m_{s+1}\rangle}}(s)$
%is ``active" when it is compatible and $\overline{\Psi^{m_{s}}_{\mu\langle m,m_{s}\rangle}}(s)$ is active, which will
%imply the existence of a $0$-level open oriented path from  $\overline{\Psi^{m}}$ to this latest hierarchical set.
%\end{remark1}

\begin{itemize}
\item {} We check if $\overline{\Psi^{\widetilde m_{s}}_{\mu_{\langle m, \widetilde m_{s}\rangle}}}(s)$ and $\overline{\Upsilon^{\widetilde m_{s}}_{\mu_{\langle m, \widetilde m_{s}\rangle}}}(s)$ are
chained  through $B_{\widetilde C_s}$;

\item {} For each pair of hierarchical $(\widetilde m_{s}-1)$-sets  $\overline
{\Psi^{\widetilde m_{s}-1}_{\mu_{\langle m, \widetilde m_{s}-1 \rangle}}}(s)$  and  $\overline
{\Upsilon^{\widetilde m_{s}-1}_{\mu_{\langle m, \widetilde m_{s}-1 \rangle}}}(s)$
chained through $B_{\widetilde C_s}$, we
select all multi-indices ${\mu_{\langle m, \widetilde{m}_{s+1} \rangle}}$ and corresponding
$\widetilde m_{s+1}$-segments  $\Upsilon^{\widetilde m_{s+1}}_{\mu_{\langle m,\widetilde m_{s+1}\rangle}}(s)$
which are compatible and for which there exists a 0-level path (open, oriented) connecting them to
$\Psi^{\widetilde{m}_s-1}_{\mu_{\langle \widetilde{m}_s, \widetilde{m}_s-1 \rangle}}$ through $B_{\widetilde C_s}$.

%If there are compatible $\overline{\Upsilon^m}_
%{\overline{\mu_{\widetilde m_s}}}(s)$, we firstly will perform the ``zooming''
%procedure: select {\it all} active  ${\Upsilon}^{\widetilde m_s-1}_
%{\overline {\mu_{\widetilde m_s-1}}}(s)$
%(again, this set
%is non-empty if at least one $\Upsilon^{\widetilde m_s-1}_{
%\mu_{\widetilde m_s}}(s)$ is active). Continuing forward, we keep
%selecting all active $\Upsilon^j_{\overline{\mu_{j }}}(s)$ until we reach level
%$\widetilde m_{s+1}$ and select all active $\overline{\Upsilon^m}_{\overline{\mu_{\widetilde m_{s+1}}}}$
%(again, this set is non-empty if at t
%least one $\overline{\Upsilon^m}_{\overline{\mu_{\widetilde m_s}}}(s)$ is
%active).

\item{}  item 2, called ``transfer", and the construction of  $\overline {\Psi^{m}}(s+1)$ both follow the same procedure as when $s=1$, replacing $\widetilde m_2$ by $\widetilde m_{s+1}$.
     We then say that $\overline{\Upsilon^{\widetilde m_{s+1}}_{\mu_{\langle m,\widetilde m_{s+1}\rangle}}}(s)$ and $\overline{\Psi^{\widetilde m_{s+1}}_{\mu_{\langle m,\widetilde m_{s+1}\rangle}}}(s+1)$ are active if the analogue of the previous (i)--(ii) both hold.

\end{itemize}
%\noindent For each active $\Upsilon^{\widetilde m_{s+1}}_{
%\mu_{\widetilde m_{s+1} }}(s)$ we define the event
%\begin{eqnarray*}
%\big[ \overline{\Upsilon^m}_{\overline{\mu_{\widetilde
%m_{s+1}}}}(s) ) \leftrightsquigarrow  \overline {\Psi^m}_{\overline{\mu_{\widetilde
%m_{s+1}}}}(s+1)\big]  =  &\big[\text {$\overline{\Upsilon^m}_{\overline{\mu_{\widetilde
%m_{s+1}}}}(s)$ is active and is
%connected to} \cr & \text { $\; \overline {\Psi^m}_{\overline{\mu_{\widetilde
%m_{s+1}}}}(s+1)$ by open  $\widetilde m_{s+1}$ path
%restricted to }\cr & \text { $\;$the tunnel ${\cZ}_{(S(\widehat
%S^{\widetilde m_{s+1}}), S)}$ $\big].$}
%\end{eqnarray*}
%and if the event
%$$
%\big[ \overline{\Upsilon^m}_{\overline{\mu_{\widetilde
%m_{s+1}}}}(s) ) \leftrightsquigarrow  \overline {\Psi^m}_{\overline{\mu_{\widetilde
%m_{s+1}}}}(s+1)\big]\cap \big[\overline {\Psi^m}_{\overline{\mu_{\widetilde
%m_{s+1}}}}(s+1)  \leftrightsquigarrow_{_{\! \! \! \!
%\! \!\! \! \! \! \! \!\!\! B(\widetilde C_{s+1})}}
%\overline{\Upsilon^m}_{\overline{\mu_{\widetilde
%m_{s+1}}}}(s+1) \big]
%$$
%occurs, we say that $\overline{\Upsilon^m}_{\overline{\mu_{\widetilde
%m_{s+1}}}}(s+1)$ is active. TO ADD SOMETHING?? This completes the step $s, \; s<v$.

\medskip

\noindent {\it Sub-case $s=v$.} This splits into two situations:

a) $\omega (B(m+1, \ell)) > \omega (\widetilde {\cC}_v) +1$. In this
case we act as if $ \widetilde m_{s+1} = 0$, i.e. we first of all
perform ``zooming" selecting all active elements down to $0$ level,
and repeat the ``transfer" procedure.

\medskip

b) $\omega (B(m+1, \ell)) = \omega (\widetilde {\cC}_v) +1$. In this
case we act as if $ \widetilde m_{s+1} = 0$, i.e. we first of all
perform the``zooming" selecting all active elements down to $0$ level,
however the ``transfer" procedure  reduces to connecting over
the last bad line of $B(m+1,\ell)$.
\bigskip

\noindent {\bf Estimates needed for the induction step.}

\medskip

%\noindent {\bf Proof.} {\it Step 1: $m=1$.}

%\smallskip

%Only $(b_1)^\prime$ makes sense to consider if $m=1$, and its verification
%is straightforward. The basis of induction for the second estimate is $(b_2)^\prime{^\prime}$,
%which is also checked in a simple way.

%We assume the validity of all estimates up to step $m$, and now prove that ($(b_{m+1})^\prime$).

%%%%                                      %%%%%%%%%%%%%%%%%                 %%%%%%%%%%%              PAGE 46

\noindent In what follows, we will repeatedly use the following basic result on the standard
oriented percolation model on $\widetilde {\mathbb Z}^2_+$:

For $a\ge 1$ a large integer, consider the rectangle $R_a = ([0, a] \times [0, a^2]) \cap \widetilde{\mathbb Z}^2_+$,
and let $(x,0)$, $(y,a^2)$ be two points lying on the two horizontal faces,
with $|x-a/2|,|y-a/2| \le a/10$, and define the following event of vertical crossing:
\begin{eqnarray*}
V(R_a)= [\text{there exists an
open oriented path from } (x,0) \text{ to } (y,a^2)
\text{ lying entirely in } R_a]
\end{eqnarray*}

%spanning between two horizontal} \\
%&{\text{faces (of length $\sqrt L $) of the rectangle $R(L)$,
%connecting the middle point}} \\ &\text{of bottom side to the
%middle point of its top side, and lying entirely within} \\
%&\text{the rectangle}].
%\end{eqnarray*}

Then the following holds:

\begin{lemma}
\label{vertical-crossing} There exist  $a_0\ge 1, 0 < \tilde p < 1$
and $\varkappa^\prime>0$, such that for any $a \ge a_0$ and $p \ge
\tilde p$ we have

%\maltese \footnote{would like any $p>p_c$?? to
%improve also the text $\theta(p)$ notation}
$$
P_p (V(R_a)) \ge p^{\varkappa^\prime}. \label{help}
$$
\end{lemma}

\begin{proof} The proof of the above inequality is rather standard. We sketch it briefly:
let us split the rectangle into $a$ disjoint squares with sides $a$, and choose $a$ large enough,
with $\tilde p$ close enough to $1$ so that the probability of survival from a starting point (centered) in the
first (from the bottom, say) square $a \times a$ is larger than $p^{\varkappa'}$, and given that the process
survives in the square, at its upper boundary it is close to its asymptotic shape and asymptotic density.
Then we repeatedly request survival and approximation to asymptotic density in the next $a-1 $
consecutive squares, starting from $a/(5\theta (p_0))$
centrally located points. The probability of each of such events is
exponentially $\exp(-c_1 a)$, with constant $c_1 > 0$, and
uniformly bounded away from $0$, for $p$ large. We easily get
the desired result.
\end{proof}

\medskip

\noindent
\begin{remark1}
\label{transfer-remark}
The previous lemma is used in the part of the procedure called ``transfer" above. It will be used
at the various scales $k \le m$, with $\tilde p=p_k$, and $a^2\ge L$ which we may assume large enough so
that the estimate applies.
\end{remark1}

\noindent Let  $(\overline {\Psi^m}, \overline {\Upsilon^m})$ be a
matching pair  with respect to $B(m+1,\ell)$ under consideration. By
the induction assumption $(b_{m})^\prime$, we have that for each
pair of indices $\mu_{\langle m,m-1\rangle}$
\begin{equation}
P\left(\,\overline
{\Upsilon^{m-1}_{\mu_{\langle m, m-1 \rangle}}}(1) \leftrightsquigarrow_{_{\! \!
\! \! \! \!\! \! \! \! \! \!\!\! B(\widetilde \cC_1)}} \overline
{\Psi^{m-1}_{\mu_{\langle m, m-1 \rangle}}}\,\right)\geq p_{m-1,m}. \label{6N.6H}
\end{equation}
For a fixed family of hierarchical $(m-1)$-sets, the events in \eqref{6N.6H} are
(conditionally) independent, so that the distribution of the number of
chained pairs, given that at least one of them is chained, is stochastically larger than
$F_{p_{m-1,m}}$.

On the other hand, from the induction assumption $(b_m)^{\prime\prime}$ we have
that for each $0\le j< m-1$ and each pair $\mu_{\langle m,j+1\rangle}$
\begin{equation}
|\{i\colon \overline{\Upsilon^{j}_{\langle\mu_{\langle m,j+1\rangle},i\rangle}}(1) \leftrightsquigarrow_{_{\!
\! \! \! \! \!\! \! \! \! \! \!\!\! B(\widetilde \cC_1)}}   \overline{\Psi^{m-1}_{\mu_{\langle m,m-1\rangle}}}\}| \big |
\big[\overline{\Upsilon^{j+1}_{\mu_{\langle m,j+1\rangle}}}(1) \leftrightsquigarrow_{_{\!
\! \! \! \! \!\! \! \! \! \! \!\!\! B(\widetilde \cC_1)}}
\overline{\Psi^{m-1}_{\mu_{\langle m,m-1\rangle}}}] \succeq
F_{p_{j,m}}, \label{6N.6}
\end{equation}
i.e. conditioned on
$\overline{\Upsilon^{j+1}_{\mu_{\langle m,j+1\rangle}}}(1)$  being
chained to $\overline{\Psi^{m-1}_{\mu_{\langle m,m-1\rangle}}}$, the number of indices $i$ so that
$\overline{\Upsilon^{j}_{\langle\mu_{\langle m,j+1\rangle},i\rangle}}(1)$
is chained to $\overline{\Psi^{m-1}_{\mu_{\langle m,m-1\rangle}}}$ is stochastically larger than
$F_{p_{j,m}}$. We shall use \eqref{6N.6} for $j$ going down to
$j=\widetilde m_2$.
%Each active element once again is "approved" with
%probability $p_{m-1}^\varkappa$, which corresponds to a tossing of
%an independent coin.

\smallskip

%\noindent By the assumption ({\bf2}$_m$) for all active and
%approved $\overline {\Upsilon}^{m+1}_1 ({\{\overline {\mu_m},
%j\}})$ we have
%\begin{eqnarray}
% \big| \{j: \; \overline {\Upsilon}^{m+1}_1
%({\{\overline{\mu_{m-1}}, j\}})\; {\text{is active}}\}\big| \, &
%\Big| \big[ \overline {\Upsilon}^{m+1}_1 ({{\mu_{m-1}}})\;
%{\text{is active}}\big] \succeq \cr  X(\rho c L^{1/2},p_{m-2,m})\,
%& \Big| [X(\rho c L^{1/2},p_{m-2,m})\ge 1],  \label{6N.7}
%\end{eqnarray}
%and continuing to apply ({\bf2}$_m$) for decreasing elements down
%to the level $\widetilde m_2$, we get that
%\begin{eqnarray}
% \big| \{j: \; \overline {\Upsilon}^{m+1}_1
%({\overline{\mu_{\widetilde m_2 +1}}}, j)\; {\text{is
%active}}\}\big| \, & \Big| \big[ {\text {$\overline
%{\Upsilon}^{m+1}_1 ( {\mu_{\widetilde m_2+1}}) $ is active}}\big]
%\succeq \cr  X(\rho c L^{1/2},p_{\widetilde m_2,m})\, & \Big|
%X(\rho c L^{1/2},p_{\widetilde m_2,m})\ge 1. \label{6N.7a}
%\end{eqnarray}
%%%%%%%%%%%%%%%%%%%%%%%%%%%%%%%%%%%%%%%%%%%%%%%%%%%%%%%%%%%%%%%%%%%%%%%%%%%%%%%%%%%%%%%%%%%%%%% Maybe useful?
%Select all active and separated $\overline {\Upsilon}^{m+1}_1
%(\mu_{\widetilde m_2})$ in $\overline {\Upsilon}^{m+1}_1$. We can
%perform it by starting from the first leftmost active $\overline
%{\Upsilon}^{m+1}_1 (\mu_{\widetilde m_2})$ and moving rightwards.
%%%%%%%%%%%%%%%%%%%%%%%%%%%%%%%%%%%%%%%%%%%%%%%%%%%%%%%%%%%%%%%%%%%%%%%%%%%%%%%%%%%%%%%%%%%%%%%%%%%%%%%%%%%%%%

\medskip

Assume $\widetilde m_2 \ge 1$, i.e. $v\ge 2$. For each index $\mu_{\langle m,\widetilde m_2\rangle}$ which yields a chained
set at all steps from $m-1$ down to $\widetilde m_2$ one now checks the $\widetilde m_2$-set
$\overline {\Upsilon^{\widetilde m_2}_{\mu_{\langle m,\widetilde m_2\rangle}}}$ is ``compatible"
and if the conditions (i) and (ii) described in the previous construction hold. Using the
induction assumption, which guarantees the validity of conditions $(a_i)-(d_i)$ for
all $i\le m$. Applying this and Lemma \ref{vertical-crossing}, we get from \eqref{help},
for each such index $\mu_{\langle m,\widetilde m_2\rangle}$:
\begin{equation}
P\left (\overline{\Upsilon^{\widetilde m_2}_{\mu_{\langle m,\widetilde m_2}\rangle}}(1) \text{ and }
\overline{\Psi^{\widetilde m_2}_{\mu_{\langle m,\widetilde m_2}\rangle}}(2) \text { are ``active"}\right) \ge {p^\kappa_{\widetilde m_2}}
\end{equation}
%\begin{equation}
%P \big( {\text {$\exists$ restricted open $\widetilde m_2$ path
%linking $\overline {\Upsilon}^{m+1}_{1}(\mu_{\widetilde m_2})$ to
%$\overline {\Psi}^{m+1}_{2}(\mu_{\widetilde m_2})$}}  \big) \ge
%p_{\widetilde m_2}^\varkappa , \label{6N.9root}
%\end{equation}
where $\varkappa = \varkappa^\prime + 2$ (the $+2$ appears since we need
to check that the starting $\widetilde m_2$-site at the bottom has
reverse $c$-dense kernel (compatible), and is forward $c$-passable).
Using then $(a_m)$ we get that for all the previous indices
$\mu_{\langle m,\widetilde m_2\rangle}$ as above (for them we have
$\overline{\Upsilon^{\widetilde m_2}_{\mu_{\langle m,\widetilde m_2\rangle}}}(1)$
is chained to $\overline{\Psi^{m-1}_{\mu_{\langle m,m-1\rangle}}}$\,) one gets
\begin{equation}
\label{lift2}
P\left(\exists i\colon \overline{\Upsilon^{\widetilde m_2-1}_{\mu_{\langle\langle m,\widetilde m_2\rangle,i\rangle}}}(2)
\leftrightsquigarrow_{_{\! \!
\! \! \! \!\! \! \! \! \! \!\!\! B(\widetilde \cC_2)}}\overline{\Psi^{\widetilde m_2-1}_{\mu_{\langle\langle m,\widetilde m_2\rangle,i\rangle}}}(2)\right) \geq p_{\widetilde m_2,\widetilde m_2}.
\end{equation}
The event on the l.h.s. of \eqref{lift2} we naturally denote as
\begin{equation*}
\left [\overline{\Upsilon^{\widetilde m_2}_{\mu_{\langle m,\widetilde m_2\rangle}}}(2)
\leftrightsquigarrow_{_{\! \!
\! \! \! \!\! \! \! \! \! \!\!\! B(\widetilde \cC_2)}}\overline{\Psi^{\widetilde m_2}_{\mu_{\langle m,\widetilde m_2\rangle}}}(2)\right]
\end{equation*}
and by the induction assumption we can write, analogously to \eqref{6N.6}, for each $j < \widetilde m_2-1$:
\begin{equation}
|\{i\colon \overline{\Upsilon^{j}_{\langle\mu_{\langle m,j+1\rangle},i\rangle}}(1) \leftrightsquigarrow_{_{\!
\! \! \! \! \!\! \! \! \! \! \!\!\! B(\widetilde \cC_2)}}   \overline{\Psi^{\widetilde m_2-1}_{\mu_{\langle m,\widetilde m_2-1\rangle}}}(2)\}| \big |
\big[\overline{\Upsilon^{j+1}_{\mu_{\langle m,j+1\rangle}}}(2) \leftrightsquigarrow_{_{\!
\! \! \! \! \!\! \! \! \! \! \!\!\! B(\widetilde \cC_2)}}
\overline{\Psi^{\widetilde{m}_2-1}_{\mu_{\langle m,\widetilde m_2-1\rangle}}}(2)] \succeq
F_{p_{j,\widetilde m_2}}. \label{6N.6-2}
\end{equation}
But it is very simple to check that  $F_p\succeq F_{\tilde p}$ when
$1 \ge p\ge \tilde p
> 0 $, and we may therefore replace $p_{j,\widetilde m_2}$ by $p_{j,m}$ on the
r.h.s. of \eqref{6N.6-2}:
\begin{equation}
|\{i\colon \overline{\Upsilon^{j}_{\langle\mu_{\langle m,j+1\rangle},i\rangle}}(1) \leftrightsquigarrow_{_{\!
\! \! \! \! \!\! \! \! \! \! \!\!\! B(\widetilde \cC_2)}}   \overline{\Psi^{\widetilde m_2-1}_{\mu_{\langle m,\widetilde m_2-1\rangle}}}(2)\}| \big |
\big[\overline{\Upsilon^{j+1}_{\mu_{\langle m,j+1\rangle}}}(2) \leftrightsquigarrow_{_{\!
\! \! \! \! \!\! \! \! \! \! \!\!\! B(\widetilde \cC_2)}}
\overline{\Psi^{\widetilde{m}_2-1}_{\mu_{\langle m,\widetilde m_2-1\rangle}}}(2)] \succeq
F_{p_{j,m}}. \label{6N.6-2b}
\end{equation}

\noindent Again we shall use \eqref{6N.6-2b} for all $j$ down to $\widetilde m_{3}$.

\medskip

\noindent Continuing for $s\le v-1$ we extend the lower bounds for the probability of an active
$\overline{\Upsilon^{\widetilde m_{s+1}}_{\mu_{\langle m, \widetilde m_{s+1}\rangle}}}(s)$ given the indices
$\mu_{\langle m,\widetilde m_{s+1}\rangle}$ yielded ``active" hierarchical sets in the previous steps.

The construction at the final step $s=v$ is slightly different as remarked above, and we consider two cases:
a) $\omega (B(m+1, \ell) > \omega (\widetilde \cC_v) +1$; b) $\omega (B(m+1, \ell) = \omega (\widetilde \cC_v) +1$.

\smallskip

\noindent In both cases we proceed as before as if $\widetilde m_{s+1}=0$, so that we use the analogue of \eqref{6N.6-2} all the way
down to $j=0$. The only difference is that in case a) we again have a transfer operation, and we once more use Lemma \ref{vertical-crossing},
this time at scale 0, but in a space without bad layers and of vertical length at least $L$. In case b) we do not have the transfer operation,
and the hierarchical set $\overline{\Upsilon^{m}}(v)$ stays on the last bad layer of $B_{\widetilde \cC_v}$.

In both cases, the final step to connect each final $\Upsilon^{0}_{\mu_{\langle m, 0 \rangle}}(v)$ to the matching $\Upsilon^0_{\mu_{\langle m, 0\rangle}}$ has probability bounded from below by $p_G^{\kappa}p_B$.

\medskip

\noindent {\it Computing the probability. Verification  of $(b_{m+1})^\prime$.} It is useful to establish a comparison with the following
simple auxiliary scheme. Consider the following system of boxes: a unique $(m+1)$-box (or box of scale $m+1$)
contains $J$  $m$-boxes, each of them containing $J$ boxes
of scale $m-1$, and so on down to scale 1: each $1$-box contains $J$ boxes of
scale $0$, thought as points.

\begin {defin}
\label{calculadora}
Checking procedure:

\noindent (a) Each $0$-box is ``good'' with probability $p_G^{m \kappa } p_B^{m+1}$, all independently.

\noindent (b) For each $k=1,\dots, m-1$ a $k$-box is ``good" if:
\begin{itemize}
\item{} it contains at least one ``good" $(k-1)$-box;
\item{} it is ``approved" at $k$-step, which happens with probability $p_k^{\kappa(m-k)}$
independently of everything else.
\end{itemize}

\noindent (c) For $k=m,m+1$ a $k$-box is ``good" if it contains at least one ``good" $(k-1)$-box.
\end{defin}
\smallskip

With all ``approvals" taken independently, and independent of the initial assignments (good/ not good),
it is straightforward to see that for each $k=0,\dots, m+1$, each $k$-box will be ``good" with probability $p_{k,m+1}$.

Of course we could think of the previous procedure in two stages:

\noindent {\it Stage 1}

\noindent (a) Each $0$-box is ``pre-good'' with probability $p_G^{(m-1)\kappa} p_B^{m}$.

\noindent (b) For each $k=1,\dots, m-1$, a $k$-box is ``pre-good" if:
\begin{itemize}
\item{} it contains at least one ``pre-good" $(k-1)$-box;
\item{} it is ``pre-approved" at $k$-step, which happens with probability $p_k^{\kappa(m-k-1)}$
independently of everything else.
\end{itemize}

\noindent (c) For $k=m,m+1$ a $k$-box is ``pre-good" if it contains at least one ``pre-good" $(k-1)$-box.

\smallskip

\noindent {\it Stage 2} Each ``pre-good" $0$-box is ``tested" again with probability $p_G^{\kappa}$; if successful, it is declared ``good".
In increasing order each $k$-box ($k=1,\dots m-1$) is ``tested" again with probability  $p_k^{\kappa}$, all ``tests" being independently; if
test is successful and if it contains at least one ``good" $(k-1)$-box, it is then declared ``good". For $k=m,m+1$, a $k$-box is
declared ``good" if it contains at least one ``good" $(k-1)$-box.

\medskip

After taking into account the estimates obtained with the procedure based on the itinerary of the descending decomposition
of the bad layer $B(m+1,\ell)$, we see that it is comparable (in the sense of stochastic order)
with the previous ``auxiliary scheme" with two stages: the first corresponds to the estimates provided by \eqref{6N.6}, \eqref {6N.6-2b}
(at all steps $s=1,\dots, v$), and the ``testing at stage 2" comes from the ``transfer" part, with the difference that the ``test" with probability $p_k^{\kappa}$ takes place only at $k=\widetilde m_{s+1}$, for $s=1,\dots, v$ along the itinerary (recall $\widetilde m_{v+1}=0$). At the scales which do not appear in the itinerary, the ``test" is automatically successful with probability one.

\bigskip

\noindent {\it Verification of $(b_{m+1})^{\prime\prime}$.} The
scheme used to define when a matching pair of $j+1$-sets is chained,
by taking at each step $J$ separated matching $j$-sets then yields
(conditional) independence (at each step), and allows to easily
conclude $(b_{m+1})^{\prime\prime}$ from $(b_{m+1})^\prime$.

\medskip
To conclude the proof of Theorem \ref{inductive}, and therefore also
of Theorem \ref{indstep} in Section \ref{passability} (where  $p_G$
is taken close enough to 1), it remains essentially to show that by
taking $L$ large one can compare the numbers $p_{m-1,m}$ given by
\eqref{6N.1} with $p_m$ defined immediately after \eqref{6N.1} for
all $m$. This will allow to conclude the induction step for
($b_{m}$) given by \eqref{5.h03}, summarized in the following:

\noindent {\bf Claim}

Let $m \ge 2$. Assuming the validity of $(a_j),(b_j),(c_j),(d_j)$
for all $j \le m-1$, and  ($a_m$), as explained immediately after
\eqref{4.80} and \eqref{4.81}, we can prove that ($b_m$) holds.

\smallskip

Taking into account what has been proven earlier in this section, it
remains to verify that
\begin{equation}
\label{final1} 8N(1-p_{m-1,m})^{\rho \frac{c}{6}\frac{L}{N}}  \le
q_m,   \text{ for all  } m \ge 2,
\end{equation}
where $N$ is given by \eqref{N}, and $p_{m-1,m}$, $q_m$, $p_m$ are
as in \eqref{6N.1} and the line that follows it.

For this, and since $L$ will be taken large it suffices to obtain
\begin{equation}
\label{final2} p_{m,m} \ge p_m \quad  \forall m \ge 2.
\end{equation}

Let
\begin{equation*}
\Theta=\prod_{k=0}^\infty p_k>0,
\end{equation*}
which is an increasing function of $p_G=p_0$, as also
$\rho=\rho(p_G)$.

We recall the interpretation of $p_{m,m}$ given in Definition
\ref{calculadora} (with $m+1$ now replaced by $m$), and proceed with
a similar checking procedure, leaving the $p_B^m$-probability for
the final step of the $0$-boxes, i.e. with the trivial observation
that if one has $t$ (a fixed integer) independent Bernoulli random
variables with probability of success given by $p\tilde p$, then the
probability of no success is bounded from above by
\begin{equation*}
(1-\tilde p)^{tp/2} + e^{-tI_p(p/2)}
\end{equation*}
where $I_p(x)=x\log(x/p)+ (1-x) \log((1-x)/(1-p))$, for $x \in
(0,1)$, is the Cram\'er transform.
(This follows at once from the
decomposition of the Bernoulli essays into two independent ones, of
probabilities $\tilde p$ and $p$ respectively, and Cram\'er Theorem
for the second one.)

At all steps $i$ from $0$ to $m-2$ each $i$-box is tested
independently of anything else with probability
$p_i^{\kappa(m-i-1)}$, and at the end the $0$-box has to be approved
with probability $\tilde p=p_B^m$.\footnote{the $m$ box and its $m-1$ boxes are not tested, according to \eqref{6N.1} }
Using Cram\'er Theorem we can
then estimate from above the probability that the $m$-box is not
``good", by splitting it into cases: (a) for each $i$, the number of
tested $i-1$-boxes which are successful is not smaller then half of
its expected number; (b) the event in (a) fails at some step $i$.
Thus,
\begin{equation}
\label{final3} 1-p_{m,m} \le (1-p_B^m)^{4(J/2)^{m}
\prod_{i=0}^{m-2} p_i^{\kappa(m-i-1)}}+\sum_{i=1}^{m-1}
 e^{-4(J/2)^{i+1}\prod_{j=2}^{i}p_{m-j}^{\kappa(j-1)} f(p_{m-i-1}^{i\kappa})}
\end{equation}
with
\begin{equation}
\label{1106}
f(p)=I_p (p/2)=(1-\frac{p}{2})\log(\frac{2-p}{1-p})-\log2
\end{equation}

It follows at once that $L_0$ large can be taken so that for all
$m\ge 2$, and all $L \ge L_0(p_B,p_G)$,
\begin{equation*}
(1-p_B^m)^{4(J/2)^{m} \prod_{i=0}^{m-2} p_i^{\kappa(m-i-1)}}\le (1-p_B^m)^{4(J/2)^m(\Theta^\kappa)^{m-1}}\le \frac12 q_m.
\end{equation*}

For the second term in \eqref{final3}, we split it into two pieces.
For the piece corresponding to large values of $i$ we use
\begin{equation*}
\sum_{i=m/2}^{m-1}
 e^{-2(J/2)^{i+1}\prod_{j=2}^{i}p_{m-j}^{\kappa(j-1)} f(p_{m-i-1}^{i\kappa})}\le
\frac{m}{2} exp\left\{-2 (\frac J2)^{m/2}\Theta^{\kappa m}f(p_G^{\kappa(m-1)})\right\}
\end{equation*}
which we can bound from above by $\frac 14 q_m$ for all $m\ge 2$,
provided $L\ge L'_0$ similarly as above. It remains to estimate
\begin{equation*}
\sum_{i=1}^{m/2-1}
 e^{-4(J/2)^{i}\prod_{j=2}^{i}p_{m-j}^{\kappa(j-1)} f(p_{m-i-1}^{i\kappa})}.
\end{equation*}
Since we may assume (by taking $L$ large) that $J\Theta^\kappa> 2$,
this last term is bounded from above by
\begin{equation*}
\frac {m}{2}-1 \exp\{-2 J\Theta^\kappa f(p_{m/2}^{\kappa m/2})\}\le
\frac{m}{2} \exp\{-4 f(p_{m/2}^{\kappa m/2})\}.
\end{equation*}

To have this bounded from above by $\frac14 q_m$ we need $4
f(p_{m/2}^{\kappa m/2}) > (m+1) \log q_0^{-1}+ \log (2m)$, and a
simple analysis of $f$ given by \eqref{1106} shows this is
the case provided $q_0$ is chosen sufficiently small. Indeed,
writing for convenience $q_0=e^{-y}$, it remains to check
\begin{equation}
-\ln (1-p_{m_2}^{\varkappa m/2}) \ge \frac{1}{4}(\ln(4m)+(m+1)y)
\end{equation}
Assuming ${\kappa m/2}$ is an integer (small modification otherwise)
\begin{align}
1- p_{m_2}^{\varkappa m/2} = & 1- \sum_{i=0}^{{\varkappa m/2}} \left(\begin{matrix} {\varkappa m/2}\\i\end{matrix} \right) (-1)^{i} e^{-i (m/2+1)y} \notag\\
= & \sum_{i=1}^{{\varkappa m/2}} \left(\begin{matrix} {\varkappa m/2}\\i\end{matrix} \right)(-1)^{i} e^{-i (m/2+1)y} \notag \\
\leq & \sum_{i=1}^{{\varkappa m/2}} \left(\begin{matrix} {\varkappa m/2}\\i\end{matrix} \right) e^{-i (m/2+1)y} \notag \\
\leq & 2^{\varkappa m/2}e^{- (m/2+1)y} \notag
\end{align}
Thus for all such $m$
\begin{equation*}
-\ln (1- p^{\varkappa m/2}_{m/2}) \geq   - \frac{\varkappa m}{2} \ln 2 + (\frac{m}{2} + 1)y  \geq \frac{1}{4} (m+1) y,
\end{equation*}
provided $\frac{\varkappa}{2} \ln 2 < y$, which holds for $p_G$ sufficiently close to 1.

%\left(\begin{matrix} {\varkappa m/2}\\i\end{matrix} \right)

\section{Extension to $p_G>p_c$}
\label{boosting}

In order to extend the main result to all values $p_G>p_c$, several modifications of the scheme described in the previous sections are needed.

We start by choosing $L$ large enough so that conditions $(c_1)$ and $(d_1)$ become satisfied. For this, the first thing
is to enlarge the size of the $0$-seed $Q^{(0)}$; it keeps the triangular shape but has $K$ sites at its top line, with $K$ large enough
so that the probability of an infinite open oriented path (in the homogeneous $p_G$ percolation model) starting from its adjacent sites from above
has probability at least $1 - (1-p^*)/4$. Adjusting $c$ and increasing $L$ one can check that the conditional probability of $S^1$ being
$s-$ passable given $Q^{(0)}$  is larger $p^*$, where the notion of passability at the level 1 includes new enlarged seeds on the top left
and top right parts of $S^1$.

Starting from this scale, the renormalization scheme repeats the previous one for the definitions of renormalized
sites at scales $k \ge 2$, in particular the $k$-seeds, contain only three passable sites of scales $1\le j\le k-1$.

At this point the only non-trivial modification involves the induction step for $(b_m)$ done in Sections \ref{structure} and \ref{conclusion}.
From level $m$ down to level $1$ we follow the same procedure as before. The key change is in the definition of a pair of
matching $1$-sites being chained, and the corresponding probability estimate of such event.
Assume that two sites $S^1$ and $\wh S^1$ form a matching pair with respect to a bad layer of mass $m$ and level $1$,
and have $s$-dense kernel and respectively reverse $\hat c$-dense kernel.
To concatenate open $0$-sites in the cluster within $Ker(S^1)$ to some open $0$-site in the reverse cluster within $Ker(\wh S^1)$,
we will act differently from the case of large $p_G$ since the density can now be arbitrary small. We look at the probability that out of
the (order $L$) $0$-sites lying below the bad layer and in the cluster within $Ker(S^1)$, at least $K^\prime \geq K$ have disjoint open path crossing the bad
layer. If this occurs, one can see that with a probability compatible with the estimates in Section \ref{conclusion} at least one of these points will have an open $0$ level path going to the top of $\wh S^1$. Due to the planarity, this establishes the desired connection.
\medskip

\noindent {\bf Acknowledgement.} The authors thank IMPA, CBPF, and Cornell University for the warm hospitality during the preparation of this work. V. S. is partially supported by CNPq grant 484801/2011-2. M.E.V is partially supported by CNPq grant 304217/2011-5.

\end{document}